\documentclass[a4paper]{article}
\usepackage[a4paper,margin=0.8in]{geometry}
\usepackage[utf8]{inputenc}
\usepackage{amsmath,amsfonts,amssymb,amsthm,bbm,mathrsfs}
\usepackage[shortlabels]{enumitem}
\usepackage{xcolor}
\usepackage{graphicx}
\usepackage{appendix}
\usepackage{setspace}
\usepackage{hyperref}
\hypersetup{
    colorlinks=true,
    linkcolor=blue,
    filecolor=magenta,      
    urlcolor=blue,
    pdftitle={Stack-Nash contro. for abstract stoch. evol. equs.},
    citecolor=blue,
    pdfpagemode=FullScreen,
}

\newtheorem{thm}{Theorem}[section]
\newtheorem{df}{Definition}[section]

\newtheorem{rmk}{Remark}[section]

\newtheorem{prop}{Proposition}[section]
\newtheorem{lm}{Lemma}[section]

\numberwithin{equation}{section}

\makeatletter
\def\@setauthors{%
  \begingroup
  \def\thanks{\protect\thanks@warning}%
  \trivlist
  \centering\footnotesize \@topsep30\p@\relax
  \advance\@topsep by -\baselineskip
  \item\relax
  \author@andify\authors
  \def\\{\protect\linebreak}%
  \authors%
  \ifx\@empty\contribs
  \else
    ,\penalty-3 \space \@setcontribs
    \@closetoccontribs
  \fi
  \endtrivlist
  \endgroup
}
\makeatother

\title{\textbf{Stackelberg-Nash Controllability for Abstract Stochastic Evolution Equations and Applications}}

\author{Abdellatif Elgrou$^\dagger$ \;\;and\;\; Omar Oukdach$^\ddagger$}
\date{}

\begin{document}

\maketitle
\vspace{-2em}

\begin{center}
\small
$^\dagger $Cadi Ayyad University, Faculty of Sciences Semlalia, Marrakesh, Morocco. \\
Email: \texttt{abdoelgrou@gmail.com} \\
\vspace{0.4cm}
$^\ddagger $Moulay Ismaïl University of Meknes, FST Errachidia, PAM Laboratory, BP 50, Boutalamine, Errachidia, Morocco. \\
Email: \texttt{omar.oukdach@gmail.com}
\end{center}

\vspace{1em}

\begin{abstract} 
This paper presents the concepts of exact, null, and approximate controllability in the Stackelberg-Nash sense for abstract forward and backward stochastic evolution equations, involving two types of controls: leaders and followers. We begin by proving the existence and uniqueness of the Nash equilibrium, as well as its characterization for fixed leader controls. We then establish a duality between these controllability concepts and the corresponding observability properties. Finally, we apply our theoretical results to the forward and backward stochastic heat equations. The results for the backward heat equation are obtained by deriving a new Carleman estimate.
\end{abstract}
\vspace{1em}
\section{Introduction}
Let \( T > 0 \) and \( (\Omega, \mathcal{F}, \mathbf{F}, \mathbb{P}) \) (with \( \mathbf{F} = \{\mathcal{F}_t\}_{t \in [0,T]} \)) be a fixed complete filtered probability space satisfying the usual condition. Let \( W(\cdot) \) be a \( V \)-valued, \( \mathbf{F} \)-adapted, cylindrical Brownian motion on \( (\Omega, \mathcal{F}, \mathbf{F}, \mathbb{P}) \), where \( \mathbf{F} \) is the corresponding natural filtration (generated by \( W(\cdot) \)). 

We consider the following forward stochastic linear  evolution equation
\[
\begin{cases}
\begin{array}{l}
dy = \left[ \mathcal{A}y + F(t)y + D_1u_1 + \sum_{i=1}^m B_iv_i \right] dt + \left[ G(t)y + D_2u_2 \right] dW(t) \quad \textnormal{in} \, (0,T], \\
y(0) = y_0,
\end{array}
\end{cases}
\]
where \( y_0 \) is the initial condition, and \( y \) is the state variable. Additionally, we examine the following backward stochastic linear  evolution equation
\[
\begin{cases}
\begin{array}{l}
dy = \left[ -\mathcal{A}y + F(t)y + G(t)Y + D_1u_1 + \sum_{i=1}^m B_iv_i \right] dt + \left[ Y + D_2u_2 \right] dW(t) \quad \textnormal{in} \, [0,T), \\
y(T) = y_T,
\end{array}
\end{cases}
\]
where \( y_T \) is the terminal condition, and \( (y,Y) \) is the state variable. In both systems, \( (u_1, u_2) \) and \( (v_1, \dots, v_m) \) represent the controls in appropriate spaces. The operators appearing in these systems will be precisely defined in the next section.

This paper presents the concepts of exact, null, and approximate controllability in the context of Stackelberg and Nash strategies for the forward and backward stochastic evolution equations outlined above, as studied in the framework of game theory. As in classical control theory, we establish a duality between these controllability concepts and suitable observability properties. To briefly describe the Stackelberg-Nash strategy, we assume that the controls are divided into two classes: leaders and followers. Each control has a specific objective to achieve, optimally. In this paper, we assume that the leaders aim to achieve a controllability goal, while the followers seek to keep the state close to fixed targets throughout the time interval. In the next section, we will introduce suitable cost functions that precisely describe the aforementioned objectives.

In contrast to existing literature, the present paper is the first to explore the duality between Stackelberg-Nash controllability and the associated observability properties in the context of abstract stochastic evolution equations. While there is a rich body of work on Stackelberg-Nash controllability for various deterministic partial differential equations (PDEs), as outlined in the references below, the stochastic case has received relatively limited attention, even for standard examples of stochastic PDEs.

Control problems are among the most fundamental issues in dynamical systems theory, and controllability is a key problem within control theory. For a given dynamic system, the controllability problem can be stated as follows: for two distinct states, is it possible to drive the system from the first state to the second within a specified finite time using one or more appropriate forcing functions? Such problems arise in various fields, including industry, engineering, biology, and physics. For results on classical exact, null, and approximate controllability notions for forward and backward stochastic linear evolution equations, we refer to \cite[Chapter 7]{lu2021mathematical} and \cite{lulin24}. Once the controllability problem is solved, an important question is how to select the `best' control from the admissible options. This leads us to optimal control problems. In these problems, the objective is to modify the system's dynamics in such a way that some quantity—such as control energy or control time—is optimized.

In many real-life situations, introducing multiple controls into a dynamic system is more realistic. In such cases, each control has a specific objective to achieve, and this needs to be done optimally. This leads to multi-objective control problems. Moreover, when there is a hierarchical structure in the decision-making process between the different controls, we encounter hierarchical multi-objective control problems. Unlike classical optimal control problems, several distinct notions of optimal solutions are considered, depending on the specific characteristics of the problem at hand. In general, it is not possible to find a single solution that is optimal for all objectives simultaneously. Therefore, the concepts of strategies or equilibria are used as solution frameworks. Among the well-known strategies in the literature are the Pareto Strategy \cite{Pa96}, Stackelberg Strategy \cite{St34}, and Nash Strategy \cite{Na51}.

In the context of PDE control, several papers address the topic of Stackelberg-Nash strategies. In the seminal papers \cite{LiPa} and \cite{LiHy}, J.-L. Lions introduced and studied the Stackelberg strategy for hyperbolic and parabolic equations. Subsequently, the authors in \cite{D02} and \cite{DL04} studied the existence and uniqueness of the Stackelberg-Nash equilibrium, as well as its characterization. The paper \cite{GRP02} addresses these questions from both a theoretical and numerical perspective. Additionally, the Stackelberg-Nash strategy for Stokes systems has been studied in \cite{GMR13}. All of the aforementioned results focus on hierarchical control for evolution equations, specifically in the case of approximate controllability. For exact Stackelberg-Nash controllability, we refer to \cite{AAF,ArFeGu24,ArCaSa15,BoMaOuNash,BoMaOuNash2}, where Carleman estimates play a crucial role. We also refer to \cite{AFS}, where the authors addressed the hierarchical exact controllability of parabolic equations with distributed and boundary controls. The same method was applied to parabolic coupled systems in \cite{HSP18,HSP16}.

Stackelberg-Nash strategies are rarely explored in the context of control problems for stochastic dynamic systems, particularly for stochastic PDEs. To the best of our knowledge, the works \cite{oukBouElgMan}, \cite{stakNahSPEs}, \cite{StNashdeg}, and \cite{egoukfcoupls} are among the few addressing Stackelberg-Nash null controllability for certain classes of stochastic parabolic equations.  The first study investigates stochastic parabolic equations with dynamic boundary conditions, while the second focuses on those with Dirichlet boundary conditions. The third examines degenerate stochastic parabolic equations, and the fourth explores coupled stochastic parabolic systems.

Now, let us first give the precise definition of Stackelberg strategy. Let \( \mathscr{H}_1 \) and \( \mathscr{H}_2 \) be two Hilbert spaces, and \( J_i: \mathscr{H}_1 \times \mathscr{H}_2 \longrightarrow \mathbb{R} \), for \( i = 1, 2 \), be two cost functions. For Stackelberg's strategy, we assume that we have a hierarchical game with two players: the leader and the follower, with associated cost functions \( J_1 \) and \( J_2 \), respectively. The Stackelberg solution to the game is achieved when the follower is forced to wait until the leader announces his policy before making his own decision. More precisely, we have the following definition, see for instance \cite{StDf}.
\begin{df}\label{defst}
A vector \( (v_1^*, v_2^*) \in \mathscr{H}_1 \times \mathscr{H}_2 \) is a Stackelberg strategy for \( J_1 \) and \( J_2 \) if 
\begin{enumerate}[label=(\textbf{\roman*})]
\item There is a function \( F: \mathscr{H}_1 \longrightarrow \mathscr{H}_2 \) such that
\[
J_2(v_1, F(v_1)) = \inf_{v \in \mathscr{H}_2} J_2(v_1, v), \quad \forall v_1 \in \mathscr{H}_1.
\]
\item The control \( v_1^* \in \mathscr{H}_1 \) satisfies 
\[
J_1(v_1^*, F(v_1^*)) = \inf_{v \in \mathscr{H}_1} J_1(v, F(v)).
\]
\item The pair \( (v_1^*, v_2^*) \) satisfies \( v_2^* = F(v_1^*) \).
\end{enumerate}
\end{df}
\noindent We now define the Nash strategy for $m$ players, $P_1, \dots, P_m$ (with $m \geq 2$ being an integer). Roughly speaking, it consists of a combination of strategies such that no player improves their gain by changing their strategy, while the other players do not change their strategy. Let $\mathscr{H}_1, \dots, \mathscr{H}_m$ be Hilbert spaces, representing the strategy spaces of the players $P_1, \dots, P_m$, respectively, and $\mathscr{H} = \mathscr{H}_1 \times \dots \times \mathscr{H}_m$. Let $J_i: \mathscr{H} \longrightarrow \mathbb{R}$, for $i=1,2,\dots,m$, be the cost function for the player $P_i$. We have the following definition.
\begin{df}
A vector $v^*=(v^*_1,\dots, v^*_m)\in \mathscr{H}$ is   a Nash equilibrium for $(J_1,\dots,  J_m)$ if 
 \begin{equation*}
J_i(v^*_1,\dots, v^*_m)= \inf\limits_{v\in\mathscr{H}_i} J_i(v^*_1,\dots v_{i-1}^*,v,v_{i+1}^*, \dots ,v^*_m), \quad\textnormal{for all}\;\; i=1,2,\dots,m.
\end{equation*}
\end{df}
\noindent By the well-known results of convex analysis and optimization, we have the following characterization of the Nash equilibrium.
\begin{rmk}
Assume that $J_i: \mathscr{H} \longrightarrow \mathbb{R}$, $i = 1,2, \dots, m$, are differentiable and convex.  
A vector $(v_1^*, \dots, v_m^*) \in \mathscr{H}$ is a Nash equilibrium for $(J_1, \dots, J_m)$ if and only if 
\begin{equation*}
\dfrac{\partial J_i}{\partial v_i}(v_1^*, \dots, v_m^*) = 0, \quad \textnormal{for all} \;\; i=1,2,\dots,m.
\end{equation*}	
\end{rmk}
To present the Stackelberg-Nash equilibrium, we suppose that there is a player $P_{m+1}$ with a cost function $J: \mathscr{K} \times \mathscr{H} \longrightarrow \mathbb{R}$, where $\mathscr{K}$ is another Hilbert space. We assume that the player $P_{m+1}$ is the leader, and the players $P_1, \dots, P_m$ with cost functions $J_1, \dots, J_m$, respectively, are the followers. Here, $J_1, \dots, J_m$ are defined on $\mathscr{K} \times \mathscr{H}$. The Stackelberg-Nash equilibrium is defined as follows.
\begin{df}
A vector $(v^*, v_1^*, \dots, v_m^*) \in \mathscr{K} \times \mathscr{H}$ is a Stackelberg-Nash equilibrium for $(J, J_1, \dots, J_m)$ with player $P_{m+1}$ as the leader and players $P_1, \dots, P_m$ as the followers if
\begin{enumerate}[label=(\textbf{\roman*})]
\item There exist functions $F_i: \mathscr{K} \longrightarrow \mathscr{H}_i$ ($i=1,2,\dots,m$) such that
\begin{equation*}
    J_i(v, F_1(v), \dots, F_m(v)) = \inf\limits_{w \in \mathscr{H}_i} J_i(v, F_1(v), \dots, F_{i-1}(v), w, F_{i+1}(v), \dots, F_m(v)), \quad \forall v \in \mathscr{K}.
\end{equation*}
\item $v^* \in \mathscr{K}$ satisfies
$$ J(v^*, F_1(v^*), \dots, F_m(v^*)) = \inf\limits_{v \in \mathscr{K}} J(v, F_1(v), \dots, F_m(v)). $$
\item $(v^*, v_1^*, \dots, v_m^*)$ satisfies $v_i^* = F_i(v^*)$, for $i=1,2,\dots,m$.
\end{enumerate}
\end{df}

The paper is organized as follows. In Section \ref{sec2}, we provide the formulation of the problem under consideration. In Section \ref{sec3}, we establish the well-posedness of the Nash equilibrium problem and its characterization. In Section \ref{sec04}, we characterize the Stackelberg-Nash controllability concepts using observability properties of the corresponding adjoint stochastic evolution systems. Finally, Section \ref{sec55} presents some illustrative examples of the Stackelberg-Nash controllability of forward and backward stochastic heat equations.

\section{Problem formulation}\label{sec2}
We first introduce some notations. For Banach spaces $X$ and $Y$, we denote by $\mathcal{L}(X;Y)$ the Banach space of all bounded linear operators from $X$ to $Y$, equipped with the usual operator norm. When $X = Y$, we simply write $\mathcal{L}(X)$ instead of $\mathcal{L}(X;X)$. For any operator $L \in \mathcal{L}(X;Y)$, we denote by $L^* \in \mathcal{L}(Y'; X')$ its adjoint operator, where $X'$ and $Y'$ are the dual spaces of $X$ and $Y$, respectively. We also denote by $\mathcal{S}(X)$ the Banach space of all self-adjoint operators on $X$. The norm and inner product of a given Hilbert space \( \mathbf{H} \) will be denoted by \( |\cdot|_\mathbf{H} \) and \( \langle \cdot, \cdot \rangle_\mathbf{H} \), respectively.

Let $\mathcal{H}$ and $V$ be separable Hilbert spaces, which are identified with their dual spaces. Denote by $\mathcal{L}^0_2 = \mathcal{L}_2(V; \mathcal{H})$ the Hilbert space of all Hilbert-Schmidt operators from $V$ to $\mathcal{H}$, equipped with the inner product
$$\langle \mathbb{L}, \mathbb{S} \rangle_{\mathcal{L}^0_2} = \sum_{j=1}^\infty \langle \mathbb{L}e_j, \mathbb{S}e_j \rangle_{\mathcal{H}}, \qquad \forall \mathbb{L}, \mathbb{S} \in \mathcal{L}^0_2,$$
where $\{e_j\}_{j \geq 1}$ be an orthonormal basis of \(V\). Assume that \((\mathcal{A}, \mathcal{D}(\mathcal{A}))\) is a linear operator that generates a \(C_0\)-semigroup \((S(t))_{t \geq 0}\) on \(\mathcal{H}\), and let \(\mathcal{A}^*\) denote its adjoint operator. It follows that \(\mathcal{A}^*\) generates the \(C_0\)-semigroup \((S(t)^*)_{t \geq 0}\), the adjoint of \((S(t))_{t \geq 0}\). In this paper, \(\mathcal{H}\) is taken as the state space for the systems under consideration.

Let $T > 0$ and $(\Omega, \mathcal{F}, \mathbf{F}, \mathbb{P})$ (with $\mathbf{F} = \{\mathcal{F}_t\}_{t \in [0,T]}$) be a fixed complete filtered probability space satisfying the usual condition. Let $W(\cdot)$ be a $V$-valued, $\mathbf{F}$-adapted, cylindrical Brownian motion on $(\Omega, \mathcal{F}, \mathbf{F}, \mathbb{P})$. In the sequel, to simplify the presentation, we only consider $\mathbf{F}$ is the corresponding natural filtration (generated
by $W(\cdot)$). For a Banach space $\mathcal{X}$, we denote by $C([0,T]; \mathcal{X})$ the space of all continuous $\mathcal{X}$-valued functions defined on $[0,T]$. We use $L^2_{\mathcal{F}_t}(\Omega; \mathcal{X})$ to refer to the space of all $\mathcal{X}$-valued $\mathcal{F}_t$-measurable random variables $g$ such that $\mathbb{E}\left(|g |_{\mathcal{X}}^2\right) < \infty$.  We define $L^2_{\mathcal{F}}(0,T; \mathcal{X})$ as the space of all $\mathcal{X}$-valued $\mathbf{F}$-adapted processes $\varphi(\cdot)$ such that  $ \mathbb{E}\left( \left| \varphi(\cdot) \right|^2_{L^2(0,T; \mathcal{X})} \right) < \infty. $ Similarly, $L^\infty_{\mathcal{F}}(0,T; \mathcal{X})$ is the space of all $\mathcal{X}$-valued $\mathbf{F}$-adapted essentially bounded processes. $L^2_{\mathcal{F}}(\Omega; C([0,T]; \mathcal{X}))$ is the space of all $\mathcal{X}$-valued $\mathbf{F}$-adapted continuous processes $\varphi(\cdot)$ such that $\mathbb{E}\left( \left| \varphi(\cdot) \right|^2_{C([0,T]; \mathcal{X})} \right) < \infty. $ Finally, $C_{\mathcal{F}}([0,T]; L^2(\Omega; \mathcal{X}))$ denotes the space of all $\mathcal{X}$-valued $\mathbf{F}$-adapted processes $\varphi(\cdot)$ such that $\varphi:[0,T] \to L^2_{\mathcal{F}_T}(\Omega; \mathcal{X})$ is continuous. All the above spaces are Banach spaces equipped with the canonical norms. For further details on these stochastic concepts, we refer to \cite[Chapter 2]{lu2021mathematical}. Throughout this paper, unless otherwise stated, we shall denote by \( C \) a positive constant, which may vary from one place to another.
 
Let $m \geq 2$ be an integer, and let $\mathcal{U}_1$, $\mathcal{U}_2$, and $\mathcal{V}_i$ ($i = 1, 2, \dots, m$) be Hilbert spaces representing the control spaces. Throughout this paper, we adopt the following notations:
$$
\mathcal{U}_T = L^2_{\mathcal{F}}(0,T; \mathcal{U}_1) \times L^2_{\mathcal{F}}(0,T; \mathcal{U}_2), \quad \mathcal{V}_T = L^2_{\mathcal{F}}(0,T; \mathcal{V}_1) \times \dots \times L^2_{\mathcal{F}}(0,T; \mathcal{V}_m).
$$

We first consider the following forward stochastic linear  evolution equation
\begin{equation}\label{asteqq1.1}
\begin{cases}
\begin{array}{l}
dy = \left[\mathcal{A}y + F(t)y + D_1u_1 + \sum_{i=1}^m B_i v_i \right] dt + \left[G(t)y + D_2u_2 \right] dW(t), \quad \text{in} \, (0,T], \\
y(0) = y_0,
\end{array}
\end{cases}
\end{equation}
where $y_0 \in L^2_{\mathcal{F}_0}(\Omega; \mathcal{H})$ is the initial state, $y$ is the state variable, $(u_1, u_2) \in \mathcal{U}_T$ are the leader controls, and $(v_1, \dots, v_m) \in \mathcal{V}_T$ are the follower controls. Additionally, $F \in L^\infty_{\mathcal{F}}(0,T; \mathcal{L}(\mathcal{H}))$, $G \in L^\infty_{\mathcal{F}}(0,T; \mathcal{L}(\mathcal{H}; \mathcal{L}^0_2))$, and the control operators are $D_1 \in \mathcal{L}(\mathcal{U}_1; \mathcal{H})$, $D_2 \in \mathcal{L}(\mathcal{U}_2; \mathcal{L}^0_2)$, $B_i \in \mathcal{L}(\mathcal{V}_i; \mathcal{H})$ for $i = 1, 2, \dots, m$.

For deterministic systems, a backward controllability problem can be transformed into a forward one by simply reversing the time variable, i.e., \( t \mapsto T - t \). However, this approach no longer applies to stochastic problems due to the adapted-ness requirement of stochastic processes. Therefore, we subsequently investigate the Stackelberg-Nash controllability of the following backward stochastic linear evolution equation
\begin{equation}\label{asteqq1.1back}
\begin{cases}
dy = \left[-\mathcal{A}y + F(t)y + G(t)Y + D_1u_1 + \sum_{i=1}^m B_iv_i\right] dt + \left[Y + D_2u_2\right] dW(t)\quad \text{in} \ [0,T), \\
y(T) = y_T,
\end{cases}
\end{equation}
where \(y_T \in L^2_{\mathcal{F}_T}(\Omega; \mathcal{H})\) is the terminal state, \((y, Y)\) is the state variable, \((u_1, u_2) \in \mathcal{U}_T\) are the leader controls, \((v_1, \dots, v_m) \in \mathcal{V}_T\) are the follower controls, \(F \in L^\infty_\mathcal{F}(0,T; \mathcal{L}(\mathcal{H}))\), and \(G \in L^\infty_\mathcal{F}(0,T; \mathcal{L}(\mathcal{L}^0_2; \mathcal{H}))\). The control operators are \(D_1 \in \mathcal{L}(\mathcal{U}_1; \mathcal{H})\), \(D_2 \in \mathcal{L}(\mathcal{U}_2; \mathcal{L}^0_2)\), and \(B_i \in \mathcal{L}(\mathcal{V}_i; \mathcal{H})\) for \(i = 1, 2, \dots, m\).

\begin{rmk}
In \eqref{asteqq1.1}, we introduce two leaders: one in the drift term and the other in the diffusion term. This dual-leader structure is often essential in the classical controllability theory for stochastic PDEs, such as stochastic parabolic equations, as it plays a key role in deriving controllability results. In the case of the equation \eqref{asteqq1.1back}, it may initially seem that only one leader is required. However, we will show that the Stackelberg-Nash controllability problem for \eqref{asteqq1.1back} reduces to the classical controllability problem of a coupled backward-forward evolution system. Therefore, to appropriately address the forward equation, a second leader may also be introduced in the diffusion term of \eqref{asteqq1.1back}. This motivates the investigation of the specific form of \eqref{asteqq1.1back}.
\end{rmk}

From \cite[Chapter 3]{lu2021mathematical}, the system \eqref{asteqq1.1} admits a unique mild solution $y\in C_\mathcal{F}([0,T];L^2(\Omega;\mathcal{H}))$. Moreover, there exists a constant $C>0$ such that
\begin{align*}
\vert y\vert_{C_\mathcal{F}([0,T];L^2(\Omega;\mathcal{H}))} \leq C \Big(\vert y_0\vert_{L^2_{\mathcal{F}_0}(\Omega;\mathcal{H})} + \vert (u_1,u_2)\vert_{\mathcal{U}_T} + \vert(v_1,\dots,v_m)\vert_{\mathcal{V}_T}\Big).
\end{align*}
By \cite[Chapter 4]{lu2021mathematical}, the system \eqref{asteqq1.1back} admits a unique mild solution $$(y,Y)\in L^2_\mathcal{F}(\Omega;C([0,T];\mathcal{H}))\times L^2_\mathcal{F}(0,T;\mathcal{L}^0_2).$$ 
Furthermore, there exists a constant $C>0$ such that
\begin{align*}
\vert y\vert_{L^2_\mathcal{F}(\Omega;C([0,T];\mathcal{H}))}+\vert Y\vert_{L^2_\mathcal{F}(0,T;\mathcal{L}^0_2)} \leq C \Big(\vert y_T\vert_{L^2_{\mathcal{F}_T}(\Omega;\mathcal{H})} + \vert (u_1,u_2)\vert_{\mathcal{U}_T} + \vert(v_1,\dots,v_m)\vert_{\mathcal{V}_T}\Big).
\end{align*}

Let us now formulate the problem under consideration: For fixed target trajectories  \( y_{i,d} \in L^2_{\mathcal{F}}(0,T; \mathcal{H}) \) and \( Y_{i,d} \in L^2_{\mathcal{F}}(0,T; \mathcal{L}^0_2) \), $i=1,2,\dots,m$, we begin by introducing two kinds of functionals:
\begin{enumerate}[(1)]
    \item The \textbf{main functional} \( J \) is given by
    \[
    J(u_1, u_2) = \frac{1}{2} \mathbb{E} \int_0^T \left( |u_1|^2_{\mathcal{U}_1} + |u_2|^2_{\mathcal{U}_2} \right) \, dt.
    \]

    \item The \textbf{secondary functionals} are defined as follows: For each \( i = 1, 2, \dots, m \),
    \begin{itemize}
        \item For the forward system \eqref{asteqq1.1}:
\begin{align}\label{functji1sec}
J_i^F(u_1,u_2;v_1,\dots,v_m) =  \frac{\alpha_i}{2} \mathbb{E} \int_0^T  |K_i(y - y_{i,d})|^2_{\mathcal{H}} \,dt + \frac{\beta_i}{2} \mathbb{E}\int_0^T \langle R_iv_i,v_i\rangle_{\mathcal{V}_i} \,dt,
\end{align}
        where \( \alpha_i > 0 \), \( \beta_i \geq 1 \), \( K_i \in \mathcal{S}(\mathcal{H}) \), and \( R_i \in \mathcal{S}(\mathcal{V}_i) \) (with \( R_i \) are invertible).
        
        \item For the backward system \eqref{asteqq1.1back}:
       \begin{align}\label{functji1secback}
\begin{aligned}
J_i^B(u_1,u_2;v_1,\dots,v_m) = &\, \frac{\alpha_i}{2} \mathbb{E} \int_0^T  |K_i(y - y_{i,d})|^2_{\mathcal{H}} \,dt+\frac{\widetilde{\alpha}_i}{2} \mathbb{E} \int_0^T  |\widetilde{K}_i(Y - Y_{i,d})|^2_{\mathcal{L}^0_2} \,dt\\
&+ \frac{\beta_i}{2} \mathbb{E}\int_0^T \langle R_iv_i,v_i\rangle_{\mathcal{V}_i} \,dt,
\end{aligned}
\end{align}
        where \( \alpha_i,\widetilde{\alpha}_i > 0 \), \( \beta_i \geq 1 \), \( K_i \in \mathcal{S}(\mathcal{H}) \), \( \widetilde{K}_i \in \mathcal{S}(\mathcal{L}^0_2) \), and \( R_i \in \mathcal{S}(\mathcal{V}_i) \) (with \( R_i \) are invertible). 
    \end{itemize}
\end{enumerate}
We also assume that there exists a constant \( r_0 > 0 \) such that
\begin{align}\label{asumRi}
\langle R_i \xi, \xi \rangle_{\mathcal{V}_i} \geq r_0 |\xi|_{\mathcal{V}_i}^2, \quad \forall \xi \in \mathcal{V}_i, \quad i = 1, 2, \dots, m.
\end{align}
Throughout this paper, we assume that the pair \( (u_1, u_2) \) of players, with the main cost functional \( J \), are the leaders, while the \( m \)-tuple \( (v_1, \dots, v_m) \) of players, with the secondary cost functionals \( (J_1^F, \dots, J_m^F) \) (resp. \( (J_1^B, \dots, J_m^B) \)), are the followers.

To achieve the controllability of both deterministic and stochastic PDEs with source terms, we generally impose these source terms in appropriately weighted spaces, using a suitable weight function \( \rho \). In our case, we fix a positive (deterministic) weight function \( \rho = \rho(t) \) and introduce the following weighted \( L^2 \)-space:
\[
L^{2,\rho}_{\mathcal{F}}(0,T; \mathbf{H}) = \left\{ \varphi: (0,T) \times \Omega \to \mathbf{H} : \, \varphi(\cdot) \,\text{ is } \,\mathbf{F}\text{-adapted, and } \mathbb{E} \int_0^T \rho^2 |\varphi|_{\mathbf{H}}^2 \, dt < \infty \right\},
\]
equipped with the norm
\[
|\varphi|_{L^{2,\rho}_{\mathcal{F}}(0,T; \mathbf{H})} = \left( \mathbb{E} \int_0^T \rho^2 |\varphi|_{\mathbf{H}}^2 \, dt \right)^{1/2}.
\]
Moreover, we assume that there exists a constant \( \rho_0 > 0 \) such that
\begin{align}\label{esttforrho0}
    \rho^{-1}(t) \leq \rho_0, \quad \forall t \in [0,T].
\end{align}

We now define the notions of exact, null, and approximate controllability in the sense of Stackelberg-Nash for the stochastic evolution equations \eqref{asteqq1.1} and \eqref{asteqq1.1back}.
\begin{df}\label{deff2.1}
\begin{enumerate}[label=\textbf{\arabic*)}]
\item  System \eqref{asteqq1.1} is optimally exactly  controllable at time $T$ in the sense of Stackelberg-Nash if for any   $y_0\in L^2_{\mathcal{F}_0}(\Omega;\mathcal{H})$, $y_T\in L^2_{\mathcal{F}_T}(\Omega;\mathcal{H})$, and target functions $y_{i,d}\in L^{2,\rho}_{\mathcal{F}}(0,T;\mathcal{H})$, $i=1,2,\dots,m$, there exist leader controls  $(\widehat{u}_1,\widehat{u}_2)\in \mathcal{U}_T$  minimizing the functional  $J$ and an associated Nash equilibrium $(v^*_1,\dots,v^*_m)\in \mathcal{V}_T$  for $(J_1^F,\dots,J_m^F)$ such that the corresponding solution $\widehat{y}$ of  the system \eqref{asteqq1.1} satisfies that $\widehat{y}(T)=y_T,\;\mathbb{P}\textnormal{-a.s.}$
\item System \eqref{asteqq1.1}  is optimally null controllable at time $T$ in the sense of Stackelberg-Nash if for any  $y_0\in L^2_{\mathcal{F}_0}(\Omega;\mathcal{H})$, and target functions $y_{i,d}\in L^{2,\rho}_{\mathcal{F}}(0,T;\mathcal{H})$, $i=1,2,\dots,m$, there exist leader controls $(\widehat{u}_1,\widehat{u}_2)\in \mathcal{U}_T$  minimizing the functional  $J$ and an associated Nash equilibrium $(v^*_1,\dots,v^*_m)\in \mathcal{V}_T$ for $(J_1^F,\dots,J_m^F)$ such that the corresponding solution $\widehat{y}$ of  the system \eqref{asteqq1.1}  satisfies $\widehat{y}(T)=0,\;\mathbb{P}\textnormal{-a.s.}$
\item System \eqref{asteqq1.1} is optimally approximately  controllable at time $T$ in the sense of Stackelberg-Nash if for any   $y_0\in L^2_{\mathcal{F}_0}(\Omega;\mathcal{H})$, $y_T\in L^2_{\mathcal{F}_T}(\Omega;\mathcal{H})$, $\varepsilon>0$  and target functions $y_{i,d}\in L^{2,\rho}_{\mathcal{F}}(0,T;\mathcal{H})$, $i=1,2,\dots,m$, there exist  leader controls  $(\widehat{u}_1,\widehat{u}_2)\in \mathcal{U}_T$  minimizing the functional  $J$ and an associated Nash equilibrium $(v^*_1,\dots,v^*_m)\in \mathcal{V}_T$ for $(J_1^F,\dots,J_m^F)$ such that the corresponding solution $\widehat{y}$ of  the system \eqref{asteqq1.1} satisfies that $\mathbb{E}|\widehat{y}(T)-y_T|_{\mathcal{H}}^2\leq \varepsilon.$
\end{enumerate}
\end{df}

\begin{df}
\begin{enumerate}[label=\textbf{\arabic*)}]
\item  System \eqref{asteqq1.1back} is optimally exactly  controllable at time $0$ in the sense of Stackelberg-Nash if for any   $y_T\in L^2_{\mathcal{F}_T}(\Omega;\mathcal{H})$, $y_0\in L^2_{\mathcal{F}_0}(\Omega;\mathcal{H})$, and target functions $y_{i,d}\in L^{2,\rho}_{\mathcal{F}}(0,T;\mathcal{H})$ and $Y_{i,d}\in L^{2,\rho}_{\mathcal{F}}(0,T;\mathcal{L}^0_2)$, $i=1,2,\dots,m$, there exist leader controls  $(\widehat{u}_1,\widehat{u}_2)\in \mathcal{U}_T$  minimizing the functional  $J$ and an associated Nash equilibrium $(v^*_1,\dots,v^*_m)\in \mathcal{V}_T$ for $(J^B_1,\dots,J^B_m)$ such that the corresponding solution $(\widehat{y},\widehat{Y})$ of  the system \eqref{asteqq1.1back} satisfies that $\widehat{y}(0)=y_0,\;\mathbb{P}\textnormal{-a.s.}$
\item System \eqref{asteqq1.1back}  is optimally null controllable at time $0$ in the sense of Stackelberg-Nash if for any $y_T\in L^2_{\mathcal{F}_T}(\Omega;\mathcal{H})$, and target functions $y_{i,d}\in L^{2,\rho}_{\mathcal{F}}(0,T;\mathcal{H})$ and $Y_{i,d}\in L^{2,\rho}_{\mathcal{F}}(0,T;\mathcal{L}^0_2)$, $i=1,2,\dots,m$, there exist leader controls $(\widehat{u}_1,\widehat{u}_2)\in \mathcal{U}_T$  minimizing the functional  $J$ and an associated Nash equilibrium $(v^*_1,\dots,v^*_m)\in \mathcal{V}_T$ for $(J^B_1,\dots,J^B_m)$ such that the corresponding solution $(\widehat{y},\widehat{Y})$ of  the system \eqref{asteqq1.1back}  satisfies $\widehat{y}(0)=0,\;\mathbb{P}\textnormal{-a.s.}$
\item System \eqref{asteqq1.1back} is optimally approximately   controllable at time $0$ in the sense of Stackelberg-Nash if for any   $y_T\in L^2_{\mathcal{F}_T}(\Omega;\mathcal{H})$, $y_0\in L^2_{\mathcal{F}_0}(\Omega;\mathcal{H})$, $\varepsilon>0$  and target functions $y_{i,d}\in L^{2,\rho}_{\mathcal{F}}(0,T;\mathcal{H})$ and $Y_{i,d}\in L^{2,\rho}_{\mathcal{F}}(0,T;\mathcal{L}^0_2)$, $i=1,2,\dots,m$, there exist leader controls  $(\widehat{u}_1,\widehat{u}_2)\in \mathcal{U}_T$  minimizing the functional  $J$ and an associated Nash equilibrium $(v^*_1,\dots,v^*_m)\in \mathcal{V}_T$ for $(J^B_1,\dots,J^B_m)$ such that the corresponding solution $(\widehat{y},\widehat{Y})$ of  the system \eqref{asteqq1.1back} satisfies that $\mathbb{E}|\widehat{y}(0)-y_0|_{\mathcal{H}}^2\leq \varepsilon.$
\end{enumerate}
\end{df}
\section{Nash equilibrium: Existence, uniqueness, and characterization}\label{sec3}
\subsection{For the forward system \eqref{asteqq1.1}}
In this subsection, for sufficiently large \( \beta_i \), we establish the existence and uniqueness of the Nash equilibrium for the functionals \( (J_1^F, \dots, J_m^F) \), as well as its characterization  using an adjoint backward system.
\begin{prop}\label{propp4.1}
There exists a large constant \( \overline{\beta} \geq 1 \) such that, if \( \beta_i \geq \overline{\beta} \) for \( i = 1, 2, \dots, m \), then for each \( (u_1, u_2) \in \mathcal{U}_T \), there exists a unique Nash equilibrium \( (v^*_1(u_1, u_2), \dots, v^*_m(u_1, u_2)) \in \mathcal{V}_T \) for the functionals \( (J_1^F, \dots, J_m^F) \) associated with \( (u_1, u_2) \).
\end{prop}
\begin{proof}
Consider the linear bounded operators \( \Lambda_i \in \mathcal{L}(L^2_\mathcal{F}(0,T;\mathcal{V}_i); L_\mathcal{F}^2(0,T; \mathcal{H})) \) defined by
\begin{align*}
\Lambda_i(v_i) = y_i, \quad i = 1, 2, \dots, m,
\end{align*}
where \( y_i \) is the solution of the following forward equation
\begin{equation}\label{1.10}
\begin{cases}
\begin{array}{l}
dy_i = (\mathcal{A}y_i + F(t)y_i + B_iv_i)dt + G(t)y_i \, dW(t), \quad \textnormal{in} \, (0,T], \\
y_i(0) = 0,\quad i = 1, 2, \dots, m.
\end{array}
\end{cases}
\end{equation}
It is straightforward to observe that the solution \( y \) of the system \eqref{asteqq1.1} can be expressed as
\begin{align}\label{equofy}
y = \sum_{i=1}^m \Lambda_i(v_i) + q,
\end{align}
where \( q \) is the solution of the forward equation
\begin{equation}\label{1.132}
\begin{cases}
\begin{array}{l}
dq = (\mathcal{A}q + F(t)q + D_1u_1)dt + (G(t)q + D_2u_2) \, dW(t), \quad \textnormal{in} \, (0,T], \\
q(0) = y_0.
\end{array}
\end{cases}
\end{equation}
For every \( (u_1, u_2) \in \mathcal{U}_T \), it is easy to verify that for any \( v_i \in L_\mathcal{F}^2(0,T; \mathcal{V}_i) \),
\begin{align*}
\left\langle \frac{\partial J_i^F}{\partial v_i}(u_1, u_2; v^*_1, \dots,v^*_m), v_i \right\rangle_{L_\mathcal{F}^2(0,T; \mathcal{V}_i)} 
&= \alpha_i \left\langle K_i^2\left(\sum_{j=1}^m \Lambda_j(v^*_j) + q - y_{i,d}\right), \Lambda_i(v_i) \right\rangle_{L_\mathcal{F}^2(0,T; \mathcal{H})} \\
&\quad + \beta_i \left\langle R_i v^*_i, v_i \right\rangle_{L_\mathcal{F}^2(0,T; \mathcal{V}_i)}, \quad i = 1, 2, \dots, m.
\end{align*}
Thus, \( (v^*_1, \dots, v^*_m) \) is a Nash equilibrium for \( (J_1^F,\dots,J_m^F) \) if and only if
\begin{align}\label{4.3nashcara}
\alpha_i \left\langle K_i^2\left(\sum_{j=1}^m \Lambda_j(v^*_j) + q - y_{i,d}\right), \Lambda_i(v_i) \right\rangle_{L_\mathcal{F}^2(0,T; \mathcal{H})} + \beta_i \left\langle R_i v^*_i, v_i \right\rangle_{L_\mathcal{F}^2(0,T; \mathcal{V}_i)} = 0, \quad \forall v_i \in L_\mathcal{F}^2(0,T; \mathcal{V}_i).
\end{align}
Equivalently, we have
\begin{align*}
\alpha_i \left\langle \Lambda_i^* K_i^2\left(\sum_{j=1}^m \Lambda_j(v^*_j) + q - y_{i,d}\right), v_i \right\rangle_{L_\mathcal{F}^2(0,T; \mathcal{V}_i)} + \beta_i \left\langle R_i v^*_i, v_i \right\rangle_{L_\mathcal{F}^2(0,T; \mathcal{V}_i)} = 0, \quad \forall v_i \in L_\mathcal{F}^2(0,T; \mathcal{V}_i),
\end{align*}
which leads to
\[
\alpha_i \Lambda_i^* K_i^2 \left[ \sum_{j=1}^m \Lambda_j(v^*_j) \right] + \beta_i R_i v^*_i = \alpha_i \Lambda_i^* K_i^2 (y_{i,d} - q) \quad \text{in} \;\; L_\mathcal{F}^2(0,T; \mathcal{V}_i), \quad i = 1, 2, \dots, m,
\]
where \( \Lambda_i^* \in \mathcal{L}(L_\mathcal{F}^2(0,T; \mathcal{H}); L_\mathcal{F}^2(0,T; \mathcal{V}_i)) \) is the adjoint operator of \( \Lambda_i \). The problem now reduces to proving the existence and uniqueness of \( (v_1^*, \dots, v_m^*) \in \mathcal{V}_T \) such that
\begin{align}\label{oureqfornash}
\mathcal{M}(v_1^*, \dots, v_m^*) = (\alpha_1 \Lambda_1^* K_1^2(y_{1,d} - q), \dots, \alpha_m \Lambda_m^* K_m^2(y_{m,d} - q)),
\end{align}
where \( \mathcal{M} : \mathcal{V}_T \to \mathcal{V}_T \) is the bounded operator defined by
\[
\mathcal{M}(v_1, \dots, v_m) = \left( \alpha_1 \Lambda_1^* K_1^2 \left[ \sum_{j=1}^m \Lambda_j(v_j) \right] + \beta_1 R_1 v_1, \dots, \alpha_m \Lambda_m^* K_m^2 \left[ \sum_{j=1}^m \Lambda_j(v_j) \right] + \beta_m R_m v_m \right).
\]
It is easy to verify that for any \( (v_1, \dots, v_m) \in \mathcal{V}_T \), we have
\[
\langle \mathcal{M}(v_1, \dots, v_m), (v_1, \dots, v_m) \rangle_{\mathcal{V}_T} = \sum_{i=1}^m \beta_i \langle R_i v_i, v_i \rangle_{L_\mathcal{F}^2(0,T; \mathcal{V}_i)} + \sum_{i,j=1}^m \alpha_i \langle K_i^2 \Lambda_j(v_j), \Lambda_i(v_i) \rangle_{L_\mathcal{F}^2(0,T; \mathcal{H})}.
\]
Recalling \eqref{asumRi}, it follows that
\begin{align}\label{estimonL}
\langle \mathcal{M}(v_1,\dots,v_m), (v_1,\dots,v_m) \rangle_{\mathcal{V}_T} \geq \sum_{i=1}^m (\beta_ir_0 - C) |v_i|_{L_\mathcal{F}^2(0,T; \mathcal{V}_i)}^2.
\end{align}
By selecting a sufficiently large \( \overline{\beta} \geq 1 \) in \eqref{estimonL}, it follows that for any \( \beta_i \geq \overline{\beta} \), \( i = 1, 2, \dots, m \), we obtain
\begin{align}\label{estimonLsecond}
\langle \mathcal{M}(v_1, \dots, v_m), (v_1, \dots, v_m) \rangle_{\mathcal{V}_T} \geq C |(v_1, \dots, v_m)|_{\mathcal{V}_T}^2.
\end{align}
Next, we introduce the bilinear continuous functional \( \mathbf{a}: \mathcal{V}_T \times \mathcal{V}_T \to \mathbb{R} \) defined as
\[
\mathbf{a}((v_1, \dots, v_m), (\widetilde{v}_1, \dots, \widetilde{v}_m)) = \left\langle \mathcal{M}(v_1, \dots, v_m), (\widetilde{v}_1, \dots, \widetilde{v}_m) \right\rangle_{\mathcal{V}_T},
\]
for any \( \big((v_1, \dots, v_m), (\widetilde{v}_1, \dots, \widetilde{v}_m)\big) \in \mathcal{V}_T \times \mathcal{V}_T \). We also define the linear continuous functional \( \Psi: \mathcal{V}_T \to \mathbb{R} \) by
\[
\Psi(v_1, \dots, v_m) = \left\langle (v_1, \dots, v_m), \left( \alpha_1 \Lambda_1^* K_1^2(y_{1,d} - q), \dots, \alpha_m \Lambda_m^* K_m^2(y_{m,d} - q) \right) \right\rangle_{\mathcal{V}_T},
\]
for any \( (v_1, \dots, v_m) \in \mathcal{V}_T \). From \eqref{estimonLsecond}, it is clear that the bilinear functional \( \mathbf{a} \) is coercive. Therefore, by the Lax-Milgram theorem, there exists a unique \( (v^*_1, \dots, v^*_m) \in \mathcal{V}_T \) such that
\[
\mathbf{a}((v^*_1, \dots, v^*_m), (v_1, \dots, v_m)) = \Psi(v_1, \dots, v_m) \quad \text{for all} \quad (v_1, \dots, v_m) \in \mathcal{V}_T.
\]
Thus, the solution \( (v^*_1, \dots, v^*_m) \) satisfies \eqref{oureqfornash}, and hence it represents the desired Nash equilibrium for the functionals \( (J_1^F,\dots,J_m^F) \). This completes the proof of Proposition \ref{propp4.1}.
\end{proof}

To characterize the Nash equilibrium \( (v^*_1, \dots, v^*_m) \), we introduce the following backward system
\begin{equation}\label{backadjforjiF}
\begin{cases}
\begin{array}{l}
dz_i = (-\mathcal{A}^*z_i - F(t)^*z_i - G(t)^*Z_i - \alpha_i K_i^2(y - y_{i,d}))dt + Z_i dW(t)\quad \text{in} \,\,[0,T),\\
z_i(T) = 0, \quad i = 1, 2, \dots, m,
\end{array}
\end{cases}
\end{equation}
where \( y = y(u_1, u_2; v_1, \dots, v_m) \) is the solution of \eqref{asteqq1.1}. Applying Itô's formula to the equations \eqref{1.10} and \eqref{backadjforjiF}, we find that
\begin{equation}\label{equa3.6}
\alpha_i \left\langle K_i^2(y - y_{i,d}), \Lambda_i(v_i) \right\rangle_{L^2_{\mathcal{F}}(0,T; \mathcal{H})} = \left\langle z_i, B_iv_i \right\rangle_{L^2_{\mathcal{F}}(0,T; \mathcal{H})}, \quad i = 1, 2, \dots, m.
\end{equation}
By combining \eqref{equa3.6}, \eqref{4.3nashcara}, and \eqref{equofy}, we conclude that \( (v^*_1, \dots, v^*_m) \) is a Nash equilibrium for the functionals \( (J_1^F,\dots,J_m^F) \) if and only if the following holds:
\[
\left\langle z_i, B_iv_i \right\rangle_{L^2_{\mathcal{F}}(0,T; \mathcal{H})} + \beta_i \left\langle R_iv^*_i, v_i \right\rangle_{L^2_{\mathcal{F}}(0,T; \mathcal{V}_i)} = 0 \quad \text{for all} \quad v_i \in L^2_{\mathcal{F}}(0,T; \mathcal{V}_i), \quad i = 1, 2, \dots, m.
\]
This gives the following formula for the Nash equilibrium:
\begin{equation}\label{charaofvi}
v^*_i = -\frac{1}{\beta_i} R_i^{-1} B_i^* z_i, \quad i = 1, 2, \dots, m.
\end{equation}

\subsection{For the backward system \eqref{asteqq1.1back}}
In this subsection, we prove the existence and uniqueness of the Nash equilibrium for the functionals \( (J_1^B, \dots, J_m^B) \) and provide its characterization. The proof follows the same approach as in Proposition \ref{propp4.1}, and thus we omit some details for simplicity.
\begin{prop}\label{propp4.1back}
There exists a sufficiently large constant \( \overline{\beta} \geq 1 \) such that, if \( \beta_i \geq \overline{\beta} \) for \( i = 1, 2, \dots, m \), then for each pair \( (u_1, u_2) \in \mathcal{U}_T \), there exists a unique Nash equilibrium \( (v_1^*(u_1, u_2), \dots, v_m^*(u_1, u_2)) \in \mathcal{V}_T \) for the functionals \( (J_1^B, \dots, J_m^B) \) associated with \( (u_1, u_2) \).
\end{prop}
\begin{proof}
Define the linear bounded operators \( \Lambda_i \in \mathcal{L}\left(L^2_\mathcal{F}(0,T;\mathcal{V}_i); L_\mathcal{F}^2(0,T; \mathcal{H}) \times L_\mathcal{F}^2(0,T; \mathcal{L}_2^0)\right) \) by
\begin{align*}
\Lambda_i(v_i) = (\Lambda_i^1(v_i), \Lambda_i^2(v_i)) = (y_i, Y_i), \quad i = 1, 2, \dots, m,
\end{align*}
where \( (y_i, Y_i) \) represents the solution of the following backward equation
\begin{equation}\label{1.10back}
\begin{cases}
\begin{array}{l}
dy_i = \left[-\mathcal{A}y_i + F(t)y_i + G(t)Y_i + B_iv_i\right] dt + Y_i \, dW(t), \quad \textnormal{in} \, [0,T), \\
y_i(T) = 0, \quad i = 1, 2, \dots, m.
\end{array}
\end{cases}
\end{equation}
Note that the solution \( (y, Y) \) of the system \eqref{asteqq1.1back} can be written as
\begin{align}\label{equofyback}
(y, Y) = \sum_{i=1}^m \Lambda_i(v_i) + (r, R),
\end{align}
where \( (r, R) \) is the solution of the backward equation
\begin{equation}\label{1.132back}
\begin{cases}
\begin{array}{l}
dr = \left[-\mathcal{A}r + F(t)r + G(t)R + D_1 u_1\right] dt + \left[ R + D_2 u_2 \right] dW(t), \quad \textnormal{in} \, [0,T), \\
r(T) = y_T.
\end{array}
\end{cases}
\end{equation}
For every \( (u_1, u_2) \in \mathcal{U}_T \), it is straightforward to verify that for any \( v_i \in L_\mathcal{F}^2(0,T; \mathcal{V}_i) \),
\begin{align*}
\left\langle \frac{\partial J_i^B}{\partial v_i}(u_1, u_2; v^*_1, \dots,v^*_m), v_i \right\rangle_{L_\mathcal{F}^2(0,T; \mathcal{V}_i)} &= \alpha_i \left\langle K_i^2\left(\sum_{j=1}^m \Lambda_j^1(v^*_j) + r - y_{i,d}\right), \Lambda_i^1(v_i) \right\rangle_{L_\mathcal{F}^2(0,T; \mathcal{H})} \\
&\quad+ \widetilde{\alpha}_i \left\langle \widetilde{K}_i^2\left(\sum_{j=1}^m \Lambda_j^2(v^*_j) + R - Y_{i,d}\right), \Lambda_i^2(v_i) \right\rangle_{L_\mathcal{F}^2(0,T; \mathcal{L}_2^0)} \\
&\quad+ \beta_i \left\langle R_i v^*_i, v_i \right\rangle_{L_\mathcal{F}^2(0,T; \mathcal{V}_i)}, \quad i = 1, 2, \dots, m.
\end{align*}
Thus, \( (v^*_1, \dots,v^*_m) \) is a Nash equilibrium for \( (J_1^B,\dots,J_m^B) \) if and only if for all \( v_i \in L_\mathcal{F}^2(0,T; \mathcal{V}_i) \),
\begin{align}\label{4.3nashcarback}
\begin{aligned}
&\alpha_i \left\langle K_i^2\left(\sum_{j=1}^m \Lambda_j^1(v^*_j) + r - y_{i,d}\right), \Lambda_i^1(v_i) \right\rangle_{L_\mathcal{F}^2(0,T; \mathcal{H})} \\
&\quad + \widetilde{\alpha}_i \left\langle \widetilde{K}_i^2\left(\sum_{j=1}^m \Lambda_j^2(v^*_j) + R - Y_{i,d}\right), \Lambda_i^2(v_i) \right\rangle_{L_\mathcal{F}^2(0,T; \mathcal{L}_2^0)} \\
&\quad + \beta_i \left\langle R_i v^*_i, v_i \right\rangle_{L_\mathcal{F}^2(0,T; \mathcal{V}_i)} = 0,
\end{aligned}
\end{align}
which implies that
\[
\Lambda_i^*\begin{bmatrix} \alpha_i K_i^2 \sum_{j=1}^m \Lambda_j^1(v^*_j) \\ \widetilde{\alpha}_i \widetilde{K}_i^2 \sum_{j=1}^m \Lambda_j^2(v^*_j) \end{bmatrix} + \beta_i R_i v^*_i = \Lambda_i^* \begin{bmatrix} \alpha_i K_i^2(y_{i,d} - r) \\ \widetilde{\alpha}_i \widetilde{K}_i^2(Y_{i,d} - R) \end{bmatrix} \quad \text{in} \;\; L_\mathcal{F}^2(0,T; \mathcal{V}_i), \quad i = 1, 2, \dots, m,
\]
where \( \Lambda_i^* \in \mathcal{L}\left(L_\mathcal{F}^2(0,T; \mathcal{H}) \times L_\mathcal{F}^2(0,T; \mathcal{L}_2^0); L_\mathcal{F}^2(0,T; \mathcal{V}_i)\right) \) is the adjoint operator of \( \Lambda_i \). The problem is now reduced to proving that there exists a unique \( (v_1^*, \dots,v_m^*) \in \mathcal{V}_T \) such that
\begin{align}\label{oureqfornashback}
\mathcal{M}(v_1^*, \dots,v_m^*) = \left(\Lambda_1^* \begin{bmatrix} \alpha_1 K_1^2(y_{1,d} - r) \\ \widetilde{\alpha}_1 \widetilde{K}_1^2(Y_{1,d} - R) \end{bmatrix}, \dots, \Lambda_m^* \begin{bmatrix} \alpha_m K_m^2(y_{m,d} - r) \\ \widetilde{\alpha}_m \widetilde{K}_m^2(Y_{m,d} - R) \end{bmatrix}\right),
\end{align}
where \( \mathcal{M} : \mathcal{V}_T \to \mathcal{V}_T \) is the bounded operator defined by
\[
\mathcal{M}(v_1, \dots,v_m) = \left(\Lambda_1^*\begin{bmatrix} \alpha_1 K_1^2 \sum_{j=1}^m \Lambda_j^1(v_j) \\ \widetilde{\alpha}_1 \widetilde{K}_1^2 \sum_{j=1}^m \Lambda_j^2(v_j) \end{bmatrix} + \beta_1 R_1 v_1, \dots, \Lambda_m^*\begin{bmatrix} \alpha_m K_m^2 \sum_{j=1}^m \Lambda_j^1(v_j) \\ \widetilde{\alpha}_m \widetilde{K}_m^2 \sum_{j=1}^m \Lambda_j^2(v_j) \end{bmatrix} + \beta_m R_m v_m \right).
\]
Similarly to the proof of Proposition \ref{propp4.1}, by applying the Lax-Milgram theorem, we establish the existence and uniqueness of the Nash equilibrium \( (v^*_1, \dots, v^*_m) \) for the functionals \( (J_1^B, \dots, J_m^B) \). This concludes the proof of Proposition \ref{propp4.1back}.
\end{proof}
Let us now provide the characterization of the Nash equilibrium \( (v^*_1, \dots,v^*_m) \) for the functionals \( (J_1^B,\dots,J_m^B) \). To this end, we introduce the following forward system
\begin{equation}\label{forwkadj}
\begin{cases}
\begin{array}{l}
dz_i = \left[\mathcal{A}^*z_i - F(t)^*z_i - \alpha_i K_i^2(y - y_{i,d})\right] dt + \left[-G(t)^*z_i - \widetilde{\alpha}_i \widetilde{K}_i^2(Y - Y_{i,d})\right] dW(t), \quad \textnormal{in} \, (0,T],\\
z_i(0) = 0, \quad i = 1, 2, \dots, m,
\end{array}
\end{cases}
\end{equation}
where \( (y, Y) = (y, Y)(u_1, u_2; v_1, \dots, v_m) \) represents the solution of \eqref{asteqq1.1back}. By applying Itô's formula to the systems \eqref{1.10back} and \eqref{forwkadj}, we derive the following relation for each \( i = 1, 2, \dots, m \):
\begin{equation}\label{equa3.6back}
\alpha_i \left\langle K_i^2(y - y_{i,d}), \Lambda_i^1(v_i) \right\rangle_{L^2_{\mathcal{F}}(0,T; \mathcal{H})} + \widetilde{\alpha}_i \left\langle \widetilde{K}_i^2(Y - Y_{i,d}), \Lambda_i^2(v_i) \right\rangle_{L^2_{\mathcal{F}}(0,T; \mathcal{L}^0_2)} = \left\langle z_i, B_i v_i \right\rangle_{L^2_{\mathcal{F}}(0,T; \mathcal{H})}.
\end{equation}
By combining \eqref{equa3.6back}, \eqref{4.3nashcarback}, and \eqref{equofyback}, we conclude that \( (v^*_1, \dots, v^*_m) \) forms a Nash equilibrium for the functionals \( (J_1^B, \dots, J_m^B) \) if and only if for all \( v_i \in L^2_{\mathcal{F}}(0,T; \mathcal{V}_i) \),
\[
\left\langle z_i, B_i v_i \right\rangle_{L^2_{\mathcal{F}}(0,T; \mathcal{H})} + \beta_i \left\langle R_i v^*_i, v_i \right\rangle_{L^2_{\mathcal{F}}(0,T; \mathcal{V}_i)} = 0 \quad \text{for all} \quad v_i \in L^2_{\mathcal{F}}(0,T; \mathcal{V}_i), \quad i = 1, 2, \dots, m.
\]
This leads to the following characterization of the Nash equilibrium for \( (J_1^B,\dots,J_m^B) \):
\begin{equation}\label{charaofviback}
v^*_i = -\frac{1}{\beta_i} R_i^{-1} B_i^* z_i, \quad i = 1, 2, \dots, m.
\end{equation}

\section{Controllability characterization}\label{sec04}
In this section, we reduce the controllability in the sense of the Stackelberg-Nash  for the equations \eqref{asteqq1.1} and \eqref{asteqq1.1back} to the classical controllability of a system of coupled stochastic evolution equations. Thus, the controllability will be characterized through the corresponding observability property.
\subsection{Stackelberg-Nash controllability for the equation \eqref{asteqq1.1}}\label{subsec3}
Using the characterization in \eqref{charaofvi}, the Stackelberg-Nash controllability problem for the system \eqref{asteqq1.1} (as defined in Definition \ref{deff2.1}) is reduced to studying the controllability of the following coupled forward-backward stochastic system
\begin{equation}\label{coeqq1.1abscon}
\begin{cases}
\begin{array}{l}
dy = \left[\mathcal{A}y+F(t)y+D_1u_1-\sum_{i=1}^m\frac{1}{\beta_i}B_iR_i^{-1}B_i^* z_i\right]dt + \left[G(t)y+D_2u_2\right] dW(t)\quad \textnormal{in} \,\,(0,T],\\
dz_i = \left[-\mathcal{A}^*z_i-F(t)^*z_i-G(t)^*Z_i-\alpha_i K_i^2(y-y_{i,d})\right]dt + Z_i dW(t)\quad \textnormal{in} \,\,[0,T),\\
y(0)=y_0, \quad z_i(T)=0,\quad i=1,2,\dots,m.
\end{array}
\end{cases}
\end{equation}
Hence, we present the following definition for the notions of controllability of \eqref{coeqq1.1abscon}.

\begin{df}
\begin{enumerate}[label=\textbf{\arabic*)}]
\item System \eqref{coeqq1.1abscon} is exactly  controllable at time $T$ if for any   $y_0\in L^2_{\mathcal{F}_0}(\Omega;\mathcal{H})$, $y_T\in L^2_{\mathcal{F}_T}(\Omega;\mathcal{H})$, and target functions $y_{i,d}\in L^{2,\rho}_{\mathcal{F}}(0,T;\mathcal{H})$, $(i=1,2,\dots,m)$, there exist leader controls  $(\widehat{u}_1,\widehat{u}_2)\in \mathcal{U}_T$  minimizing the functional  $J$ such that the corresponding state   $\widehat{y}$, given from the unique solution \( (\widehat{y};\widehat{z}_i,\widehat{Z}_i) \) of the system \eqref{coeqq1.1abscon} satisfies that
$\widehat{y}(T)=y_T,\;\mathbb{P}\textnormal{-a.s.}$ Moreover, the leaders $(\widehat{u}_1,\widehat{u}_2)$ can be chosen such that 
\begin{align}\label{estimm1}
\begin{aligned}
|\widehat{u}_1|^2_{L^2_\mathcal{F}(0,T;\mathcal{U}_1)} + |\widehat{u}_2|^2_{L^2_\mathcal{F}(0,T;\mathcal{U}_2)} \leq C \bigg( \mathbb{E} |y_0|^2_{\mathcal{H}} + \mathbb{E} |y_T|^2_{\mathcal{H}} + \sum_{i=1}^m \mathbb{E} \int_0^T  \rho^2|K_iy_{i,d}|_{\mathcal{H}}^2 \,dt \bigg),
\end{aligned}
\end{align}
for some constant $C>0$.
\item System \eqref{coeqq1.1abscon}  is null controllable at time $T$ if for every initial condition $y_0\in L^2_{\mathcal{F}_0}(\Omega;\mathcal{H})$, and target functions $y_{i,d}\in L^{2,\rho}_{\mathcal{F}}(0,T;\mathcal{H})$, $(i=1,2,\dots,m)$, there exist leader controls $(\widehat{u}_1,\widehat{u}_2)\in \mathcal{U}_T$  minimizing the functional  $J$ such that the corresponding state   $\widehat{y}$, given from the unique solution \( (\widehat{y};\widehat{z}_i,\widehat{Z}_i) \) of the system \eqref{coeqq1.1abscon} satisfies that  $\widehat{y}(T)=0,\;\mathbb{P}\textnormal{-a.s.}$ Moreover, the leaders $(\widehat{u}_1,\widehat{u}_2)$ can be chosen such that
\begin{align}\label{estimm2}
\begin{aligned}
|\widehat{u}_1|^2_{L^2_\mathcal{F}(0,T;\mathcal{U}_1)} + |\widehat{u}_2|^2_{L^2_\mathcal{F}(0,T;\mathcal{U}_2)} \leq C \bigg( \mathbb{E} |y_0|^2_{\mathcal{H}} + \sum_{i=1}^m \mathbb{E} \int_0^T  \rho^2|K_iy_{i,d}|_{\mathcal{H}}^2 \,dt \bigg),
\end{aligned}
\end{align}
for some constant $C>0$.
\item System \eqref{coeqq1.1abscon} is approximately  controllable at time $T$ if for any   $y_0\in L^2_{\mathcal{F}_0}(\Omega;\mathcal{H})$, $y_T\in L^2_{\mathcal{F}_T}(\Omega;\mathcal{H})$, $\varepsilon>0$  and target functions $y_{i,d}\in L^{2}_{\mathcal{F}}(0,T;\mathcal{H})$, $(i=1,2,\dots,m)$, there exist  leader controls  $(\widehat{u}_1,\widehat{u}_2)\in \mathcal{U}_T$  minimizing the functional  $J$  such that the corresponding state   $\widehat{y}$, given from the unique solution \( (\widehat{y};\widehat{z}_i,\widehat{Z}_i) \) of the system \eqref{coeqq1.1abscon} satisfies that
$\mathbb{E}|\widehat{y}(T)-y_T|_{\mathcal{H}}^2\leq \varepsilon.$
\end{enumerate}
\end{df}
By classical duality arguments, the controllability of \eqref{coeqq1.1abscon} can be reformulated as an observability problem for the following adjoint backward-forward system
\begin{equation}\label{coeqq1.adjoint}
\begin{cases}
\begin{array}{l}
d\phi = \left[-\mathcal{A}^*\phi-F(t)^*\phi-G(t)^*\Phi+\sum_{i=1}^m\alpha_i K_i^2\psi_i\right]dt + \Phi dW(t)\quad \textnormal{in} \,\,[0,T),\\
d\psi_i = \left[\mathcal{A}\psi_i+F(t)\psi_i+\frac{1}{\beta_i}B_iR_i^{-1}B^*_i\phi\right]dt + G(t)\psi_i dW(t)\quad \textnormal{in} \,\,(0,T],\\
\phi(T)=\phi_T, \quad \psi_i(0)=0,\quad i=1,2,\dots,m.
\end{array}
\end{cases}
\end{equation}

We first prove the following duality relation between the solutions of \eqref{coeqq1.1abscon} and \eqref{coeqq1.adjoint}.
\begin{prop}\label{propp4.1dual}
The solutions of \eqref{coeqq1.1abscon} and \eqref{coeqq1.adjoint} satisfy that, for all \( y_0 \in L^2_{\mathcal{F}_0}(\Omega; \mathcal{H}) \) and \( \phi_T \in L^2_{\mathcal{F}_T}(\Omega; \mathcal{H}) \), 
\begin{align}\label{dulirelat}
    \begin{aligned}
\mathbb{E}\langle y(T),\phi_T\rangle_{\mathcal{H}}-\mathbb{E}\langle y_0,\phi(0)\rangle_{\mathcal{H}}=&\,\mathbb{E}\int_0^T \langle u_1,D_1^*\phi\rangle_{\mathcal{U}_1} dt+\mathbb{E}\int_0^T \langle u_2,D_2^*\Phi\rangle_{\mathcal{U}_2} dt\\
&+\sum_{i=1}^m\alpha_i\mathbb{E}\int_0^T \langle y_{i,d},K_i^2\psi_i\rangle_{\mathcal{H}} dt.
\end{aligned}
\end{align}
\end{prop}
\begin{proof}
Using Itô's formula for \( d\langle y, \phi \rangle_{\mathcal{H}} \), we integrate the resulting equality over \( (0,T) \) and take the expectation on both sides. This yields
\begin{align}\label{eqqq11}
\begin{aligned}
\mathbb{E}\langle y(T), \phi_T \rangle_{\mathcal{H}} - \mathbb{E}\langle y_0, \phi(0) \rangle_{\mathcal{H}} &= \mathbb{E}\int_0^T \langle u_1, D_1^* \phi \rangle_{\mathcal{U}_1} \, dt + \mathbb{E}\int_0^T \langle u_2, D_2^* \Phi \rangle_{\mathcal{U}_2} \, dt \\
& \quad - \sum_{i=1}^m \frac{1}{\beta_i} \mathbb{E} \int_0^T \langle z_i, B_i R_i^{-1} B_i^* \phi \rangle_{\mathcal{H}} \, dt + \sum_{i=1}^m \alpha_i \mathbb{E} \int_0^T \langle K_i^2 y, \psi_i \rangle_{\mathcal{H}} \, dt.
\end{aligned}
\end{align}
Similarly, applying Itô's formula to \( \sum_{i=1}^m d\langle z_i, \psi_i \rangle_{\mathcal{H}} \), we obtain 
\begin{align}\label{eqqq33}
\begin{aligned}
0 &= - \sum_{i=1}^m \alpha_i \mathbb{E} \int_0^T \langle K_i^2 y, \psi_i \rangle_{\mathcal{H}} \, dt + \sum_{i=1}^m \alpha_i \mathbb{E} \int_0^T \langle K_i^2 y_{i,d}, \psi_i \rangle_{\mathcal{H}} \, dt \\
& \quad + \sum_{i=1}^m \frac{1}{\beta_i} \mathbb{E} \int_0^T \langle z_i, B_i R_i^{-1} B_i^* \phi \rangle_{\mathcal{H}} \, dt.
\end{aligned}
\end{align}
Combining \eqref{eqqq11} and \eqref{eqqq33}, we deduce the identity \eqref{dulirelat}.
\end{proof}
First, we have the following characterization result for the exact controllability of \eqref{coeqq1.1abscon}.
\begin{thm}\label{thmm4.1ex}
System \eqref{coeqq1.1abscon} is exactly controllable at time \( T \) if and only if there exists a constant \( C > 0 \) such that the solutions of \eqref{coeqq1.adjoint} satisfy the following (exact) observability inequality:
\begin{align}\label{observaineqexa}
\begin{aligned}
\mathbb{E} |\phi_T|^2_{\mathcal{H}} + \sum_{i=1}^m \mathbb{E} \int_0^T \rho^{-2} |K_i \psi_i|_{\mathcal{H}}^2 \, dt \leq C \mathbb{E} \int_0^T \left( |D_1^* \phi|_{\mathcal{U}_1}^2 + |D_2^* \Phi|_{\mathcal{U}_2}^2 \right) \, dt.
\end{aligned}
\end{align}
\end{thm}
\begin{proof}
Let us assume that the system \eqref{coeqq1.1abscon} is exactly controllable at time \( T \). From \eqref{dulirelat}, we have
\begin{align*}
    \begin{aligned}
\mathbb{E}\langle y_T, \phi_T \rangle_{\mathcal{H}} - \sum_{i=1}^m \alpha_i \mathbb{E} \int_0^T \langle y_{i,d}, K_i^2 \psi_i \rangle_{\mathcal{H}} \, dt &= \mathbb{E} \int_0^T \langle \widehat{u}_1, D_1^* \phi \rangle_{\mathcal{U}_1} \, dt + \mathbb{E} \int_0^T \langle \widehat{u}_2, D_2^* \Phi \rangle_{\mathcal{U}_2} \, dt \\
&\quad + \mathbb{E} \langle y_0, \phi(0) \rangle_{\mathcal{H}}.
\end{aligned}
\end{align*}
Then, for any \( \tau > 0 \), it is easy to see that
\begin{align*}
    \begin{aligned}
\mathbb{E}\langle y_T, \phi_T \rangle_{\mathcal{H}} - \sum_{i=1}^m \alpha_i \mathbb{E} \int_0^T \langle y_{i,d}, K_i^2 \psi_i \rangle_{\mathcal{H}} \, dt \leq & \, \tau \left( |\widehat{u}_1|^2_{L^2_\mathcal{F}(0,T; \mathcal{U}_1)} + |\widehat{u}_2|^2_{L^2_\mathcal{F}(0,T; \mathcal{U}_2)} \right) \\
&+ \frac{1}{4 \tau} \left( |D_1^* \phi|_{L^2_\mathcal{F}(0,T; \mathcal{U}_1)}^2 + |D_2^* \Phi|_{L^2_\mathcal{F}(0,T; \mathcal{U}_2)}^2 \right) \\
&+ \mathbb{E} \langle y_0, \phi(0) \rangle_{\mathcal{H}}.
\end{aligned}
\end{align*}
Recalling \eqref{estimm1}, we obtain the inequality
\begin{align}\label{ineqqobs}
    \begin{aligned}
\mathbb{E}\langle y_T, \phi_T \rangle_{\mathcal{H}} - \sum_{i=1}^m \alpha_i \mathbb{E} \int_0^T \langle K_i y_{i,d}, K_i \psi_i \rangle_{\mathcal{H}} \, dt \leq & \, C \tau \left( \mathbb{E} |y_0|^2_{\mathcal{H}} + \mathbb{E} |y_T|^2_{\mathcal{H}} + \sum_{i=1}^m \mathbb{E} \int_0^T \rho^2 |K_i y_{i,d}|_{\mathcal{H}}^2 \, dt \right) \\
&+ \frac{1}{4 \tau} \left( |D_1^* \phi|_{L^2_\mathcal{F}(0,T; \mathcal{U}_1)}^2 + |D_2^* \Phi|_{L^2_\mathcal{F}(0,T; \mathcal{U}_2)}^2 \right) \\
&+ \mathbb{E} \langle y_0, \phi(0) \rangle_{\mathcal{H}}.
\end{aligned}
\end{align}
By choosing \( y_0 \equiv 0 \), \( y_T \equiv \phi_T \), and \( y_{i,d} \equiv -\frac{1}{\alpha_i} \rho^{-2} \psi_i \) (note that \( y_{i,d} \) belongs to \( L^{2, \rho}_\mathcal{F}(0,T; \mathcal{H}) \) by \eqref{esttforrho0}), we get
\begin{align}\label{ineqqobs2}
    \begin{aligned}
\mathbb{E} |\phi_T|^2_{\mathcal{H}} + \sum_{i=1}^m \mathbb{E} \int_0^T \rho^{-2} |K_i \psi_i|_{\mathcal{H}}^2 \, dt \leq & \, C \tau \left( \mathbb{E} |\phi_T|^2_{\mathcal{H}} + \sum_{i=1}^m \frac{1}{\alpha_i^2} \mathbb{E} \int_0^T \rho^{-2} |K_i \psi_i|_{\mathcal{H}}^2 \, dt \right) \\
&+ \frac{1}{4 \tau} \left( |D_1^* \phi|_{L^2_\mathcal{F}(0,T; \mathcal{U}_1)}^2 + |D_2^* \Phi|_{L^2_\mathcal{F}(0,T; \mathcal{U}_2)}^2 \right).
\end{aligned}
\end{align}
Therefore, by selecting a sufficiently small \( \tau \) in \eqref{ineqqobs2}, we obtain the observability inequality \eqref{observaineqexa}.

We now assume that the observability inequality \eqref{observaineqexa} holds. Let \( y_0 \in L^2_{\mathcal{F}_0}(\Omega; \mathcal{H}) \), \( y_T \in L^2_{\mathcal{F}_T}(\Omega; \mathcal{H}) \), and \( y_{i,d} \in L^{2,\rho}_{\mathcal{F}}(0,T; \mathcal{H}) \). We define the following functional on \( L^2_{\mathcal{F}_T}(\Omega; \mathcal{H}) \) as
\begin{align*}
J^{ext}(\phi_T) = & \frac{1}{2} \mathbb{E} \int_0^T \left( |D_1^* \phi|_{\mathcal{U}_1}^2 + |D_2^* \Phi|_{\mathcal{U}_2}^2 \right) dt - \mathbb{E} \langle y_T, \phi_T \rangle_{\mathcal{H}} + \mathbb{E} \langle y_0, \phi(0) \rangle_{\mathcal{H}} \\
& + \sum_{i=1}^m \alpha_i\mathbb{E} \int_0^T \langle y_{i,d}, K_i^2 \psi_i \rangle_{\mathcal{H}} \, dt.
\end{align*}
It is easy to verify that \( J^{ext}: L^2_{\mathcal{F}_T}(\Omega; \mathcal{H}) \longrightarrow \mathbb{R} \) is continuous and strictly convex. Let us now prove that \( J^{ext} \) is coercive.  For any \( \delta > 0 \), it is straightforward to see that
\begin{align}\label{ineqq11}
\begin{aligned}
-\mathbb{E} \langle y_T, \phi_T \rangle_{\mathcal{H}} & \geq -\frac{\delta}{2} \mathbb{E} |\phi_T|_{\mathcal{H}}^2 - \frac{1}{2 \delta} \mathbb{E} |y_T|_{\mathcal{H}}^2,
\end{aligned}
\end{align}
\begin{align}\label{ineqq22}
\begin{aligned}
\mathbb{E} \langle y_0, \phi(0) \rangle_{\mathcal{H}} & \geq -\frac{\delta}{2} C \mathbb{E} |\phi_T|_{\mathcal{H}}^2 - \frac{1}{2 \delta} \mathbb{E} |y_0|_{\mathcal{H}}^2,
\end{aligned}
\end{align}
and
\begin{align}\label{ineqq33}
    \begin{aligned}
\sum_{i=1}^m \alpha_i\mathbb{E} \int_0^T \langle y_{i,d}, K_i^2 \psi_i \rangle_{\mathcal{H}} \, dt & \geq -\frac{\delta}{2} \sum_{i=1}^m \mathbb{E} \int_0^T \rho^{-2} |K_i \psi_i|^2_{\mathcal{H}} \, dt - \frac{1}{2 \delta} \sum_{i=1}^m \alpha_i^2\mathbb{E} \int_0^T \rho^2 |K_i y_{i,d}|^2_{\mathcal{H}} \, dt.
    \end{aligned}
\end{align}
Using the observability inequality \eqref{observaineqexa}, combining \eqref{ineqq11}, \eqref{ineqq22}, and \eqref{ineqq33}, and choosing a sufficiently small \( \delta \), we obtain that
\[
J^{ext}(\phi_T) \geq C \mathbb{E} |\phi_T|^2_{\mathcal{H}} - C \mathbb{E} |y_0|_{\mathcal{H}}^2 - C \mathbb{E} |y_T|_{\mathcal{H}}^2 - C \sum_{i=1}^m \alpha_i^2\mathbb{E} \int_0^T \rho^2 |K_i y_{i,d}|^2_{\mathcal{H}} \, dt,
\]
which implies that the functional \( J^{ext} \) is coercive in \( L^2_{\mathcal{F}_T}(\Omega; \mathcal{H}) \). Thus, it admits a unique minimum \( \widehat{\phi}_T \in L^2_{\mathcal{F}_T}(\Omega;\mathcal{H}) \). Then, the Gateaux derivative of \( J^{ext} \) verifies the Euler equation, i.e., for all \( \phi_T \in L^2_{\mathcal{F}_T}(\Omega;\mathcal{H}) \),
\begin{align}\label{limmGDD}
\lim_{h \to 0} \frac{J^{ext}(\widehat{\phi}_T + h\phi_T) - J^{ext}(\widehat{\phi}_T)}{h} = 0.
\end{align}
By a direct computation of the limit in \eqref{limmGDD}, we find that
\begin{align}\label{idenfronjvare}
    \begin{aligned}
    & \mathbb{E}\int_0^T \langle D_1^*\widehat{\phi}, D_1^*\phi \rangle_{\mathcal{U}_1} \, dt + \mathbb{E}\int_0^T \langle D_2^*\widehat{\Phi}, D_2^*\Phi \rangle_{\mathcal{U}_2} \, dt - \mathbb{E}\langle y_T, \phi_T \rangle_{\mathcal{H}} \\
    & + \mathbb{E}\langle y_0, \phi(0) \rangle_{\mathcal{H}} + \sum_{i=1}^m \alpha_i \mathbb{E}\int_0^T \langle y_{i,d}, K_i^2 \psi_i \rangle_{\mathcal{H}} \, dt = 0,
    \end{aligned}
\end{align}
where \( (\widehat{\phi}, \widehat{\Phi}; \widehat{\psi}_i) \) is the solution of \eqref{coeqq1.adjoint} associated with the final state \( \widehat{\phi}_T \). 

We now choose the controls \( (\widehat{u}_1, \widehat{u}_2) = (D_1^*\widehat{\phi}, D_2^*\widehat{\Phi}) \) in \eqref{dulirelat}. We have that
\begin{align}\label{dulirelatexact}
\begin{aligned}
& -\mathbb{E}\langle \widehat{y}(T), \phi_T \rangle_{\mathcal{H}} + \mathbb{E}\langle y_0, \phi(0) \rangle_{\mathcal{H}} + \mathbb{E}\int_0^T \langle D_1^*\widehat{\phi}, D_1^*\phi \rangle_{\mathcal{U}_1} \, dt \\
& + \mathbb{E}\int_0^T \langle D_2^*\widehat{\Phi}, D_2^*\Phi \rangle_{\mathcal{U}_2} \, dt + \sum_{i=1}^m \alpha_i \mathbb{E}\int_0^T \langle y_{i,d}, K_i^2 \psi_i \rangle_{\mathcal{H}} \, dt = 0,
\end{aligned}
\end{align}
where \( (\widehat{y}; \widehat{z}_i, \widehat{Z}_i) \) is the solution of \eqref{coeqq1.1abscon} associated with the controls \( (\widehat{u}_1, \widehat{u}_2) \). 

Combining \eqref{idenfronjvare} and \eqref{dulirelatexact}, we deduce that
\begin{align*}
    \mathbb{E} \langle \widehat{y}(T) - y_T, \phi_T \rangle_{\mathcal{H}} = 0, \qquad \forall \phi_T \in L^2_{\mathcal{F}_T}(\Omega; \mathcal{H}),
\end{align*}
which implies that
\begin{align}\label{apprforexactcon}
    \widehat{y}(T) = y_T, \quad \text{in} \; L^2_{\mathcal{F}_T}(\Omega; \mathcal{H}).
\end{align}
Taking \( \phi_T = \widehat{\phi}_T \) in \eqref{idenfronjvare}, we have that
\begin{align}\label{idenfsecd}
    \begin{aligned}
    \mathbb{E}\int_0^T |\widehat{u}_1|^2_{\mathcal{U}_1} \, dt + \mathbb{E}\int_0^T |\widehat{u}_2|^2_{\mathcal{U}_2} \, dt = & \, \mathbb{E}\langle y_T, \phi_T \rangle_{\mathcal{H}} - \mathbb{E}\langle y_0, \phi(0) \rangle_{\mathcal{H}} \\
    & - \sum_{i=1}^m \alpha_i \mathbb{E}\int_0^T \langle y_{i,d}, K_i^2 \psi_i \rangle_{\mathcal{H}} \, dt.
    \end{aligned}
\end{align}
Applying the observability inequality \eqref{observaineqexa} to the right-hand side of \eqref{idenfsecd}, we deduce that
\begin{align*}
\begin{aligned}
|\widehat{u}_1|^2_{L^2_\mathcal{F}(0,T;\mathcal{U}_1)} + |\widehat{u}_2|^2_{L^2_\mathcal{F}(0,T;\mathcal{U}_2)} \leq C \Bigg( \mathbb{E} |y_0|^2_{\mathcal{H}} + \mathbb{E} |y_T|^2_{\mathcal{H}} + \sum_{i=1}^m \mathbb{E} \int_0^T \rho^2 |K_iy_{i,d}|_{\mathcal{H}}^2 \, dt \Bigg).
\end{aligned}
\end{align*}
This completes the proof of Theorem \ref{thmm4.1ex}.
\end{proof}
For the null controllability, we have the following characterization.
\begin{thm}\label{thmm4.2ex}
System \eqref{coeqq1.1abscon} is null controllability at time \( T \) if and only if there exists a constant \( C > 0 \) such that the solutions of \eqref{coeqq1.adjoint} satisfy the following (null) observability inequality:
\begin{align}\label{observaineqnuu}
\begin{aligned}
\mathbb{E}|\phi(0)|^2_{\mathcal{H}} + \sum_{i=1}^m \mathbb{E} \int_0^T \rho^{-2}|K_i\psi_i|_{\mathcal{H}}^2 \, dt \leq C \mathbb{E} \int_0^T \left( |D_1^*\phi|_{\mathcal{U}_1}^2 + |D_2^*\Phi|_{\mathcal{U}_2}^2 \right) \, dt.
\end{aligned}
\end{align}
\end{thm}
\begin{proof}
Let us assume that the system \eqref{coeqq1.1abscon} is null controllability at time $T$. From \eqref{dulirelat}, we have that
\begin{align*}
    \begin{aligned}
    -\mathbb{E}\langle y_0,\phi(0)\rangle_{\mathcal{H}}-\sum_{i=1}^m\alpha_i\mathbb{E}\int_0^T \langle y_{i,d},K_i^2\psi_i\rangle_{\mathcal{H}} dt = &\, \mathbb{E}\int_0^T \langle \widehat{u}_1,D_1^*\phi\rangle_{\mathcal{U}_1} dt + \mathbb{E}\int_0^T \langle \widehat{u}_2,D_2^*\Phi\rangle_{\mathcal{U}_2} dt.
    \end{aligned}
\end{align*}
For any \( \tau > 0 \), it follows that
\begin{align*}
    \begin{aligned}
    -\mathbb{E}\langle y_0,\phi(0)\rangle_{\mathcal{H}}-\sum_{i=1}^m\alpha_i\mathbb{E}\int_0^T \langle y_{i,d},K_i^2\psi_i\rangle_{\mathcal{H}} dt \leq &\, \tau\left( |\widehat{u}_1|^2_{L^2_\mathcal{F}(0,T;\mathcal{U}_1)} + |\widehat{u}_2|^2_{L^2_\mathcal{F}(0,T;\mathcal{U}_2)} \right) \\
    &+ \frac{1}{4\tau} \left( |D_1^*\phi|_{L^2_\mathcal{F}(0,T;\mathcal{U}_1)}^2 + |D_2^*\Phi|_{L^2_\mathcal{F}(0,T;\mathcal{U}_2)}^2 \right).
    \end{aligned}
\end{align*}
Recalling the estimate \eqref{estimm2}, we deduce that
\begin{align}\label{ineqqobsnull}
    \begin{aligned}
    -\mathbb{E}\langle y_0,\phi(0)\rangle_{\mathcal{H}}-\sum_{i=1}^m\alpha_i\mathbb{E}\int_0^T \langle K_iy_{i,d},K_i\psi_i\rangle_{\mathcal{H}} dt \leq &\, C\tau \left( \mathbb{E} |y_0|^2_{\mathcal{H}} + \sum_{i=1}^m \mathbb{E} \int_0^T \rho^2 |K_iy_{i,d}|_{\mathcal{H}}^2 \, dt \right) \\
    &+ \frac{1}{4\tau} \left( |D_1^*\phi|_{L^2_\mathcal{F}(0,T;\mathcal{U}_1)}^2 + |D_2^*\Phi|_{L^2_\mathcal{F}(0,T;\mathcal{U}_2)}^2 \right).
    \end{aligned}
\end{align}
Choosing \( y_0 \equiv -\phi(0) \) and \( y_{i,d} \equiv -\frac{1}{\alpha_i}\rho^{-2}\psi_i \) in \eqref{ineqqobsnull}, we get
\begin{align}\label{ineqqobs2null}
    \begin{aligned}
    \mathbb{E}|\phi(0)|^2_{\mathcal{H}} + \sum_{i=1}^m \mathbb{E} \int_0^T \rho^{-2} |K_i \psi_i|^2_{\mathcal{H}} dt \leq &\, C\tau \left( \mathbb{E} |\phi(0)|^2_{\mathcal{H}} + \sum_{i=1}^m \frac{1}{\alpha_i^2} \mathbb{E} \int_0^T \rho^{-2} |K_i \psi_i|_{\mathcal{H}}^2 \, dt \right) \\
    &+ \frac{1}{4\tau} \left( |D_1^*\phi|_{L^2_\mathcal{F}(0,T;\mathcal{U}_1)}^2 + |D_2^*\Phi|_{L^2_\mathcal{F}(0,T;\mathcal{U}_2)}^2 \right).
    \end{aligned}
\end{align}
Taking a sufficiently small \( \tau \) in \eqref{ineqqobs2null}, we obtain the observability inequality \eqref{observaineqnuu}.

We now suppose that the inequality \eqref{observaineqnuu} holds. Let $y_0 \in L^2_{\mathcal{F}_0}(\Omega; \mathcal{H})$, $y_{i,d} \in L^{2,\rho}_{\mathcal{F}}(0,T; \mathcal{H})$, and $\varepsilon > 0$. We define the following functional on $L^2_{\mathcal{F}_T}(\Omega; \mathcal{H})$ as follows:
\begin{align*}
J^{nul}_\varepsilon(\phi_T) = &\frac{1}{2} \mathbb{E} \int_0^T \left( |D_1^* \phi|_{\mathcal{U}_1}^2 + |D_2^* \Phi|_{\mathcal{U}_2}^2 \right) \, dt + \mathbb{E} \langle y_0, \phi(0) \rangle_{\mathcal{H}} \\
&+ \varepsilon |\phi_T|_{L^2_{\mathcal{F}_T}(\Omega; \mathcal{H})} + \sum_{i=1}^m \alpha_i\mathbb{E} \int_0^T \langle y_{i,d}, K_i^2 \psi_i \rangle_{\mathcal{H}} \, dt.
\end{align*}
It is easy to verify that $J^{nul}_\varepsilon: L^2_{\mathcal{F}_T}(\Omega; \mathcal{H}) \longrightarrow \mathbb{R}$ is continuous and strictly convex. Applying the observability inequality \eqref{observaineqnuu}, and using the same computations as in the proof of Theorem \ref{thmm4.1ex}, it follows that the functional $J^{nul}_\varepsilon$ is also coercive. Therefore, $J^{nul}_\varepsilon$ admits a unique minimum $\phi_T^\varepsilon \in L^2_{\mathcal{F}_T}(\Omega; \mathcal{H})$. If $\phi_T^\varepsilon \neq 0$, by the Euler equation, we deduce that for all $\phi_T \in L^2_{\mathcal{F}_T}(\Omega; \mathcal{H})$,
\begin{align} \label{idenfronjvarenull}
    \begin{aligned}
    & \mathbb{E} \int_0^T \langle D_1^* \phi^\varepsilon, D_1^* \phi \rangle_{\mathcal{U}_1} \, dt + \mathbb{E} \int_0^T \langle D_2^* \Phi^\varepsilon, D_2^* \Phi \rangle_{\mathcal{U}_2} \, dt + \mathbb{E} \langle y_0, \phi(0) \rangle_{\mathcal{H}} \\
    &+ \varepsilon \frac{\langle \phi_T^\varepsilon, \phi_T \rangle_{L^2_{\mathcal{F}_T}(\Omega; \mathcal{H})}}{|\phi_T^\varepsilon|_{L^2_{\mathcal{F}_T}(\Omega; \mathcal{H})}} + \sum_{i=1}^m \alpha_i \mathbb{E} \int_0^T \langle y_{i,d}, K_i^2 \psi_i \rangle_{\mathcal{H}} \, dt = 0,
    \end{aligned}
\end{align}
where $(\phi^\varepsilon, \Phi^\varepsilon; \psi_i^\varepsilon)$ is the solution of \eqref{coeqq1.adjoint} associated with the final state $\phi_T^\varepsilon$.

Choosing the controls $(u_1^\varepsilon, u_2^\varepsilon) = (D_1^* \phi^\varepsilon, D_2^* \Phi^\varepsilon)$ in \eqref{dulirelat}, we have that
\begin{align} \label{dulirelatnull}
\begin{aligned}
    & - \mathbb{E} \langle y_\varepsilon(T), \phi_T \rangle_{\mathcal{H}} + \mathbb{E} \langle y_0, \phi(0) \rangle_{\mathcal{H}} + \mathbb{E} \int_0^T \langle D_1^* \phi^\varepsilon, D_1^* \phi \rangle_{\mathcal{U}_1} \, dt \\
    &+ \mathbb{E} \int_0^T \langle D_2^* \Phi^\varepsilon, D_2^* \Phi \rangle_{\mathcal{U}_2} \, dt + \sum_{i=1}^m \alpha_i \mathbb{E} \int_0^T \langle y_{i,d}, K_i^2 \psi_i \rangle_{\mathcal{H}} \, dt = 0,
\end{aligned}
\end{align}
where $(y_\varepsilon; z_i^\varepsilon, Z_i^\varepsilon)$ is the solution of \eqref{coeqq1.1abscon} associated with the controls $(u_1^\varepsilon, u_2^\varepsilon)$. Combining \eqref{idenfronjvarenull} and \eqref{dulirelatnull}, we deduce that
\begin{align*}
    \mathbb{E} \langle y_\varepsilon(T), \phi_T \rangle_{\mathcal{H}} + \varepsilon \frac{\langle \phi_T^\varepsilon, \phi_T \rangle_{L^2_{\mathcal{F}_T}(\Omega; \mathcal{H})}}{|\phi_T^\varepsilon|_{L^2_{\mathcal{F}_T}(\Omega; \mathcal{H})}} = 0, \quad \forall \phi_T \in L^2_{\mathcal{F}_T}(\Omega; \mathcal{H}),
\end{align*}
which implies that
\begin{align} \label{apprforexactconnull}
    |y_\varepsilon(T)|_{L^2_{\mathcal{F}_T}(\Omega; \mathcal{H})} \leq \varepsilon.
\end{align}
Taking $\phi_T = \phi_T^\varepsilon$ in \eqref{idenfronjvarenull}, we have that
\begin{align*}
    \begin{aligned}
    \mathbb{E} \int_0^T |u_1^\varepsilon|^2_{\mathcal{U}_1} \, dt + \mathbb{E} \int_0^T |u_2^\varepsilon|^2_{\mathcal{U}_2} \, dt \leq & \, - \mathbb{E} \langle y_0, \phi(0) \rangle_{\mathcal{H}} - \sum_{i=1}^m \alpha_i \mathbb{E} \int_0^T \langle y_{i,d}, K_i^2 \psi_i \rangle_{\mathcal{H}} \, dt.
    \end{aligned}
\end{align*}
Using the observability inequality \eqref{observaineqnuu}, we deduce that
\begin{align} \label{estimm1forexatnull}
\begin{aligned}
    |u_1^\varepsilon|^2_{L^2_\mathcal{F}(0,T; \mathcal{U}_1)} + |u_2^\varepsilon|^2_{L^2_\mathcal{F}(0,T; \mathcal{U}_2)} \leq C \left( \mathbb{E} |y_0|^2_{\mathcal{H}} + \sum_{i=1}^m \mathbb{E} \int_0^T \rho^2 |K_i y_{i,d}|_{\mathcal{H}}^2 \, dt \right).
\end{aligned}
\end{align}
If \( \phi_T^\varepsilon = 0 \), we have (see, e.g., \cite{FabPuZuaz95})
\begin{align} \label{Ezu.1}
\lim_{h \to 0^{+}} \frac{J^{nul}_{\varepsilon}(h \phi_T)}{h} \geq 0, \quad \forall \phi_T \in L^2_{\mathcal{F}_T}(\Omega; \mathcal{H}).
\end{align}
By \eqref{Ezu.1}, and choosing \( (u_1^\varepsilon, u_2^\varepsilon) = (0, 0) \), we observe that \eqref{apprforexactconnull} and \eqref{estimm1forexatnull} still hold. From \eqref{estimm1forexatnull}, we deduce the existence of a subsequence \( (u_1^\varepsilon, u_2^\varepsilon) \) (of the sequence \( (u_1^\varepsilon, u_2^\varepsilon) \)) such that, as \( \varepsilon \to 0 \),
\begin{align} \label{weakconvrnull}
\begin{aligned}
    u_1^\varepsilon & \longrightarrow \widehat{u}_1 \quad \text{weakly in} \, L^2((0,T) \times \Omega; \mathcal{U}_1), \\
    u_2^\varepsilon & \longrightarrow \widehat{u}_2 \quad \text{weakly in} \, L^2((0,T) \times \Omega; \mathcal{U}_2).
\end{aligned}
\end{align}
Thus, by \eqref{weakconvrnull}, we also have that
\begin{equation} \label{Eq56null}
    y_\varepsilon(T) \longrightarrow \widehat{y}(T) \quad \text{weakly in} \; L^2_{\mathcal{F}_T}(\Omega; \mathcal{H}), \quad \text{as} \, \varepsilon \to 0,
\end{equation}
where the state $\widehat{y}$ is given by the unique solution \( (\widehat{y}; \widehat{z}_i, \widehat{Z}_i) \) of the system \eqref{coeqq1.1abscon} associated with the controls \( (\widehat{u}_1, \widehat{u}_2) \). Finally, combining \eqref{apprforexactconnull} and \eqref{Eq56null}, we conclude that
\[
\widehat{y}(T) = 0, \quad \mathbb{P}\textnormal{-a.s.}
\]
Moreover, the inequality \eqref{estimm2} can be derived from \eqref{estimm1forexatnull} and \eqref{weakconvrnull}. This completes the proof of Theorem \ref{thmm4.2ex}.
\end{proof}
As for the approximate controllability, we have the following characterization.
\begin{thm}\label{thmmapp43app}
System \eqref{coeqq1.1abscon} is approximately controllable at time $T$ if and only if the solutions of \eqref{coeqq1.adjoint} satisfy the following unique continuation property:
\begin{align}\label{observainuniqc}
D_1^*\phi = D_2^*\Phi = 0, \quad \forall t \in (0,T), \quad \mathbb{P}\textnormal{-a.s.} \quad \Longrightarrow \quad \phi_T = 0, \quad \mathbb{P}\textnormal{-a.s.}
\end{align}
\end{thm}
\begin{proof}
Assume that the system \eqref{coeqq1.1abscon} is approximately controllable at time $T$, and define the following operator:
\begin{equation*}
\mathcal{L}_T: \mathcal{U}_T \longrightarrow L^2_{\mathcal{F}_T}(\Omega; \mathcal{H}), \quad (u_1, u_2) \longmapsto y(T),
\end{equation*}
where the state $y$ is given by the unique solution $(y; z_i, Z_i)$ of the system \eqref{coeqq1.1abscon} with $y_0 \equiv 0$ and $y_{i,d} \equiv 0$ ($i = 1, 2, \dots, m$). Then, it is easy to see that $\mathcal{R}(\mathcal{L}_T)$ is dense in $L^2_{\mathcal{F}_T}(\Omega; \mathcal{H})$, which is equivalent to the adjoint operator $\mathcal{L}_T^*$ being injective. On the other hand, from the duality relation \eqref{dulirelat}, it is easy to see that
\[
\langle \mathcal{L}_T(u_1, u_2), \phi_T \rangle_{L^2_{\mathcal{F}_T}(\Omega; \mathcal{H})} = \langle (u_1, u_2), (D_1^*\phi, D_2^*\Phi) \rangle_{\mathcal{U}_T},
\]
which implies that $\mathcal{L}_T^*(\phi_T) = (D_1^*\phi, D_2^*\Phi)$. Therefore, if the system \eqref{coeqq1.1abscon} is approximately controllable at time $T$, we deduce that the unique continuation property \eqref{observainuniqc} holds.

Let us now assume that the property \eqref{observainuniqc} holds. Let \( y_0 \in L^2_{\mathcal{F}_0}(\Omega; \mathcal{H}) \), \( y_T \in L^2_{\mathcal{F}_T}(\Omega; \mathcal{H}) \), \( y_{i,d} \in L^{2}_{\mathcal{F}}(0,T; \mathcal{H}) \), and \( \varepsilon > 0 \). We define the following functional on \( L^2_{\mathcal{F}_T}(\Omega; \mathcal{H}) \):
\[
J^{app}_\varepsilon(\phi_T) = \frac{1}{2} \mathbb{E} \int_0^T \left( |D_1^* \phi|_{\mathcal{U}_1}^2 + |D_2^* \Phi|_{\mathcal{U}_2}^2 \right) dt - \mathbb{E} \langle y_T, \phi_T \rangle_{\mathcal{H}} + \mathbb{E} \langle y_0, \phi(0) \rangle_{\mathcal{H}}
\]
\[
+ \varepsilon \left| \phi_T \right|_{L^2_{\mathcal{F}_T}(\Omega; \mathcal{H})} + \sum_{i=1}^m \alpha_i\mathbb{E} \int_0^T \langle y_{i,d}, K_i^2 \psi_i \rangle_{\mathcal{H}} dt.
\]
It is straightforward to observe that \( J^{app}_\varepsilon: L^2_{\mathcal{F}_T}(\Omega; \mathcal{H}) \longrightarrow \mathbb{R} \) is continuous and strictly convex. We now demonstrate that \( J^{app}_\varepsilon \) is coercive, which suffices to prove that
\[
\liminf_{|\phi_T|_{L^2_{\mathcal{F}_T}(\Omega; \mathcal{H})} \to +\infty} \frac{J_\varepsilon^{app}(\phi_T)}{|\phi_T|_{L^2_{\mathcal{F}_T}(\Omega; \mathcal{H})}} \geq \varepsilon.
\]
Let \( \{\phi_T^n\}_{n \geq 1} \subset L^2_{\mathcal{F}_T}(\Omega; \mathcal{H}) \) be a sequence of final data of the system \eqref{coeqq1.adjoint} such that \( |\phi_T^n|_{L^2_{\mathcal{F}_T}(\Omega; \mathcal{H})} \to +\infty \) as \( n \to +\infty \). Set \( \widetilde{\phi}_T^n = \frac{\phi_T^n}{|\phi_T^n|_{L^2_{\mathcal{F}_T}(\Omega; \mathcal{H})}} \), and denote by \( \{\widetilde{\phi}^n, \widetilde{\Phi}^n; \widetilde{\psi}_i^n\}_{n \geq 1} \) the solution of \eqref{coeqq1.adjoint} corresponding to the final data \( \{\widetilde{\phi}_T^n\}_{n \geq 1} \). It is easy to see that
\begin{align}\label{prolinminf}
\begin{aligned}
    \frac{J_\varepsilon^{app}(\phi_T^n)}{|\phi_T^n|_{L^2_{\mathcal{F}_T}(\Omega; \mathcal{H})}} = & \frac{|\phi_T^n|_{L^2_{\mathcal{F}_T}(\Omega; \mathcal{H})}}{2} \mathbb{E} \int_0^T \left( |D_1^* \widetilde{\phi}^n|_{\mathcal{U}_1}^2 + |D_2^* \widetilde{\Phi}^n|_{\mathcal{U}_2}^2 \right) dt \\
    &- \mathbb{E} \langle y_T, \widetilde{\phi}_T^n \rangle_{\mathcal{H}} + \mathbb{E} \langle y_0, \widetilde{\phi}^n(0) \rangle_{\mathcal{H}} + \varepsilon + \sum_{i=1}^m\alpha_i \mathbb{E} \int_0^T \langle y_{i,d}, K_i^2 \widetilde{\psi}_i^n \rangle_{\mathcal{H}} dt.
\end{aligned}
\end{align}
This implies that
\begin{align}\label{ineqqforcoerjvare}
\begin{aligned}
    \frac{J_\varepsilon^{app}(\phi_T^n)}{|\phi_T^n|_{L^2_{\mathcal{F}_T}(\Omega; \mathcal{H})}} \geq & \frac{|\phi_T^n|_{L^2_{\mathcal{F}_T}(\Omega; \mathcal{H})}}{2} \mathbb{E} \int_0^T \left( |D_1^* \widetilde{\phi}^n|_{\mathcal{U}_1}^2 + |D_2^* \widetilde{\Phi}^n|_{\mathcal{U}_2}^2 \right) dt \\
    & - C \left( |y_0|_{L^2_{\mathcal{F}_0}(\Omega; \mathcal{H})} + |y_T|_{L^2_{\mathcal{F}_T}(\Omega; \mathcal{H})} + \sum_{i=1}^m |K_i y_{i,d}|_{L^2_{\mathcal{F}}(0,T; \mathcal{H})} \right) + \varepsilon.
\end{aligned}
\end{align}
If $\liminf_{n \to +\infty} \mathbb{E} \int_0^T \left( |D_1^* \widetilde{\phi}^n|_{\mathcal{U}_1}^2 + |D_2^* \widetilde{\Phi}^n|_{\mathcal{U}_2}^2 \right) dt > 0,$ then, by \eqref{ineqqforcoerjvare}, one has that
\[
\frac{J_\varepsilon^{app}(\phi_T^n)}{|\phi_T^n|_{L^2_{\mathcal{F}_T}(\Omega; \mathcal{H})}} \longrightarrow +\infty \quad \text{as} \quad n \to +\infty.
\]
If $\liminf_{n \to +\infty} \mathbb{E} \int_0^T \left( |D_1^* \widetilde{\phi}^n|_{\mathcal{U}_1}^2 + |D_2^* \widetilde{\Phi}^n|_{\mathcal{U}_2}^2 \right) dt = 0$. Since \( |\widetilde{\phi}_T^n|_{L^2_{\mathcal{F}_T}(\Omega; \mathcal{H})} = 1 \), one can extract a subsequence of \( \{\widetilde{\phi}_T^n\}_{n \geq 1} \) (still denoted by the same symbol) such that 
\[
\widetilde{\phi}_T^n \longrightarrow \widetilde{\phi}_T \quad \text{in} \quad L^2_{\mathcal{F}_T}(\Omega; \mathcal{H}).
\]
Let us denote by \( (\widetilde{\phi}, \widetilde{\Phi}; \widetilde{\psi}_i) \) the solution of \eqref{coeqq1.adjoint} associated with the final state \( \widetilde{\phi}_T \). By the definition of the weak solution of the system \eqref{coeqq1.adjoint}, the corresponding solution \( (\widetilde{\phi}^n, \widetilde{\Phi}^n; \widetilde{\psi}_i^n) \) (to the final state \( \widetilde{\phi}_T^n \)) converges to \( (\widetilde{\phi}, \widetilde{\Phi}; \widetilde{\psi}_i) \) in \( L^2_\mathcal{F}(0,T; \mathcal{H}) \). Therefore,
\[
\mathbb{E} \int_0^T \left( |D_1^* \widetilde{\phi}|_{\mathcal{U}_1}^2 + |D_2^* \widetilde{\Phi}|_{\mathcal{U}_2}^2 \right) dt \leq \liminf_{n \to +\infty} \mathbb{E} \int_0^T \left( |D_1^* \widetilde{\phi}^n|_{\mathcal{U}_1}^2 + |D_2^* \widetilde{\Phi}^n|_{\mathcal{U}_2}^2 \right) dt.
\]
Hence, \( D_1^* \widetilde{\phi} = D_2^* \widetilde{\Phi} = 0 \), and since the unique continuation property \eqref{observainuniqc} holds, we deduce that \( \widetilde{\phi}_T = 0 \), i.e., \( (\widetilde{\phi}, \widetilde{\Phi}; \widetilde{\psi}_i) = (0, 0; 0) \). Then, from \eqref{prolinminf}, it follows that
\[
\liminf_{n \to +\infty} \frac{J_\varepsilon^{app}(\widetilde{\phi}_T^n)}{|\widetilde{\phi}_T^n|_{L^2_{\mathcal{F}_T}(\Omega; \mathcal{H})}} \geq \varepsilon,
\]
which implies that $\liminf_{n \to +\infty} \frac{J_\varepsilon^{app}(\phi_T^n)}{|\phi_T^n|_{L^2_{\mathcal{F}_T}(\Omega; \mathcal{H})}} \geq \varepsilon.$ Therefore, the functional \( J_\varepsilon^{app} \) is coercive, and hence it admits a unique minimum \( \phi_T^\varepsilon \in L^2_{\mathcal{F}_T}(\Omega; \mathcal{H}) \). If \( \phi_T^\varepsilon \neq 0 \), by the Euler equation, for all \( \phi_T \in L^2_{\mathcal{F}_T}(\Omega; \mathcal{H}) \), we have that
\begin{align}\label{idenfronjvareapp}
    \begin{aligned}
        & \mathbb{E} \int_0^T \langle D_1^* \phi^\varepsilon, D_1^* \phi \rangle_{\mathcal{U}_1} \, dt + \mathbb{E} \int_0^T \langle D_2^* \Phi^\varepsilon, D_2^* \Phi \rangle_{\mathcal{U}_2} \, dt - \mathbb{E} \langle y_T, \phi_T \rangle_{\mathcal{H}} \\
        & + \mathbb{E} \langle y_0, \phi(0) \rangle_{\mathcal{H}} + \varepsilon \frac{\langle \phi_T^\varepsilon, \phi_T \rangle_{L^2_{\mathcal{F}_T}(\Omega; \mathcal{H})}}{|\phi_T^\varepsilon|_{L^2_{\mathcal{F}_T}(\Omega; \mathcal{H})}} + \sum_{i=1}^m \alpha_i \mathbb{E} \int_0^T \langle y_{i,d}, K_i^2 \psi_i \rangle_{\mathcal{H}} \, dt = 0,
    \end{aligned}
\end{align}
where \( (\phi^\varepsilon, \Phi^\varepsilon; \psi_i^\varepsilon) \) is the solution of \eqref{coeqq1.adjoint} associated with the final state \( \phi_T^\varepsilon \).

Taking the controls \( (u_1^\varepsilon, u_2^\varepsilon) = (D_1^* \phi^\varepsilon, D_2^* \Phi^\varepsilon) \) in \eqref{dulirelat}, we have that
\begin{align}\label{dulirelatexactapp}
\begin{aligned}
&-\mathbb{E} \langle y_\varepsilon(T), \phi_T \rangle_{\mathcal{H}} + \mathbb{E} \langle y_0, \phi(0) \rangle_{\mathcal{H}} + \mathbb{E} \int_0^T \langle D_1^* \phi^\varepsilon, D_1^* \phi \rangle_{\mathcal{U}_1} \, dt \\
& + \mathbb{E} \int_0^T \langle D_2^* \Phi^\varepsilon, D_2^* \Phi \rangle_{\mathcal{U}_2} \, dt + \sum_{i=1}^m \alpha_i \mathbb{E} \int_0^T \langle y_{i,d}, K_i^2 \psi_i \rangle_{\mathcal{H}} \, dt = 0,
\end{aligned}
\end{align}
where \( (y_\varepsilon; z_i^\varepsilon, Z_i^\varepsilon) \) is the solution of \eqref{coeqq1.1abscon} associated with the controls \( (u_1^\varepsilon, u_2^\varepsilon) \). Combining \eqref{idenfronjvareapp} and \eqref{dulirelatexactapp}, we deduce that
\[
\mathbb{E} \langle y_\varepsilon(T) - y_T, \phi_T \rangle_{\mathcal{H}} + \varepsilon \frac{\langle \phi_T^\varepsilon, \phi_T \rangle_{L^2_{\mathcal{F}_T}(\Omega; \mathcal{H})}}{|\phi_T^\varepsilon|_{L^2_{\mathcal{F}_T}(\Omega; \mathcal{H})}} = 0, \quad \forall \phi_T \in L^2_{\mathcal{F}_T}(\Omega; \mathcal{H}),
\]
which leads to
\begin{align}\label{apprforexactconapp}
    |y_\varepsilon(T) - y_T|_{L^2_{\mathcal{F}_T}(\Omega; \mathcal{H})} \leq \varepsilon.
\end{align}
The inequality \eqref{apprforexactconapp} still holds when \( \phi_T^\varepsilon = 0 \). This concludes the proof of Theorem \ref{thmmapp43app}.
\end{proof}

\subsection{Stackelberg-Nash controllability for the equation \eqref{asteqq1.1back}}\label{subsec3back}
The characterization in \eqref{charaofviback} allows us to reduce the Stackelberg-Nash controllability problem for the equation \eqref{asteqq1.1back} to the analysis of the controllability of the following coupled backward-forward stochastic system
\begin{equation}\label{coeqq1.1absconback}
\begin{cases}
\begin{array}{l}
dy = \left[-\mathcal{A}y+F(t)y+G(t)Y+D_1u_1-\sum_{i=1}^m\frac{1}{\beta_i}B_iR_i^{-1}B_i^* z_i\right]dt + (Y+D_2u_2) dW(t)\quad \textnormal{in} \,\,[0,T),\\
dz_i = \left[\mathcal{A}^*z_i-F(t)^*z_i-\alpha_i K_i^2(y-y_{i,d})\right]dt +\left[-G(t)^*z_i-\widetilde{\alpha}_i \widetilde{K}_i^2(Y-Y_{i,d})\right] dW(t)\quad \textnormal{in} \,\,(0,T],\\
y(T)=y_T, \quad z_i(0)=0,\quad i=1,2,\dots,m.
\end{array}
\end{cases}
\end{equation}
We now introduce different notions of controllability for the system \eqref{coeqq1.1absconback}.
\begin{df}
\begin{enumerate}[label=\textbf{\arabic*)}]
\item System \eqref{coeqq1.1absconback} is exactly controllable at time $0$ if for any $y_T \in L^2_{\mathcal{F}_T}(\Omega; \mathcal{H})$, $y_0 \in L^2_{\mathcal{F}_0}(\Omega; \mathcal{H})$, and target functions $y_{i,d} \in L^{2,\rho}_{\mathcal{F}}(0,T; \mathcal{H})$ and $Y_{i,d} \in L^{2,\rho}_{\mathcal{F}}(0,T; \mathcal{L}^0_2)$, $(i=1,2,\dots,m)$, there exist leader controls $(\widehat{u}_1, \widehat{u}_2) \in \mathcal{U}_T$ minimizing the functional $J$ such that the corresponding state $(\widehat{y}, \widehat{Y})$, given from the unique solution \( (\widehat{y}, \widehat{Y}; \widehat{z}_i) \) of \eqref{coeqq1.1absconback}, satisfies that $\widehat{y}(0) = y_0, \;\mathbb{P}\text{-a.s.}$ Moreover, the leaders $(\widehat{u}_1, \widehat{u}_2)$ can be chosen such that 
\begin{align}\label{estimm1back}
\begin{aligned}
|\widehat{u}_1|^2_{L^2_\mathcal{F}(0,T; \mathcal{U}_1)} + |\widehat{u}_2|^2_{L^2_\mathcal{F}(0,T; \mathcal{U}_2)} \leq C \bigg( &\mathbb{E} |y_T|^2_{\mathcal{H}} + \mathbb{E} |y_0|^2_{\mathcal{H}} + \sum_{i=1}^m \mathbb{E} \int_0^T \rho^{2} |K_iy_{i,d}|_{\mathcal{H}}^2 \,dt \\
&+ \sum_{i=1}^m \mathbb{E} \int_0^T \rho^{2} |\widetilde{K}_iY_{i,d}|_{\mathcal{L}^0_2}^2 \,dt \bigg),
\end{aligned}
\end{align}
for some constant $C > 0$.
\item System \eqref{coeqq1.1absconback} is null controllable at time $0$ if for every initial condition $y_T \in L^2_{\mathcal{F}_T}(\Omega; \mathcal{H})$, and target functions $y_{i,d} \in L^{2,\rho}_{\mathcal{F}}(0,T; \mathcal{H})$ and $Y_{i,d} \in L^{2,\rho}_{\mathcal{F}}(0,T; \mathcal{L}^0_2)$, $(i=1,2,\dots,m)$, there exist leader controls $(\widehat{u}_1, \widehat{u}_2) \in \mathcal{U}_T$ minimizing the functional $J$ such that the corresponding state $(\widehat{y}, \widehat{Y})$, given from the unique solution \( (\widehat{y}, \widehat{Y}; \widehat{z}_i) \) of \eqref{coeqq1.1absconback}, satisfies that $\widehat{y}(0) = 0, \;\mathbb{P}\text{-a.s.}$ Moreover, the leaders $(\widehat{u}_1, \widehat{u}_2)$ can be chosen such that
\begin{align}\label{estimm2back}
\begin{aligned}
|\widehat{u}_1|^2_{L^2_\mathcal{F}(0,T; \mathcal{U}_1)} + |\widehat{u}_2|^2_{L^2_\mathcal{F}(0,T; \mathcal{U}_2)} \leq C \bigg( &\mathbb{E} |y_T|^2_{\mathcal{H}} + \sum_{i=1}^m \mathbb{E} \int_0^T \rho^{2} |K_iy_{i,d}|_{\mathcal{H}}^2 \,dt + \sum_{i=1}^m \mathbb{E} \int_0^T \rho^{2} |\widetilde{K}_iY_{i,d}|_{\mathcal{L}^0_2}^2 \,dt \bigg),
\end{aligned}
\end{align}
for some constant $C > 0$.
\item System \eqref{coeqq1.1absconback} is approximately controllable at time $0$ if for any $y_T \in L^2_{\mathcal{F}_T}(\Omega; \mathcal{H})$, $y_0 \in L^2_{\mathcal{F}_0}(\Omega; \mathcal{H})$, $\varepsilon > 0$ and target functions $y_{i,d} \in L^{2}_{\mathcal{F}}(0,T; \mathcal{H})$ and $Y_{i,d} \in L^{2}_{\mathcal{F}}(0,T; \mathcal{L}^0_2)$, $(i=1,2,\dots,m)$, there exist leader controls $(\widehat{u}_1, \widehat{u}_2) \in \mathcal{U}_T$ minimizing the functional $J$ such that the corresponding state $(\widehat{y}, \widehat{Y})$, given from the unique solution \( (\widehat{y}, \widehat{Y}; \widehat{z}_i) \) of \eqref{coeqq1.1absconback} satisfies that $\mathbb{E} |\widehat{y}(0) - y_0|_{\mathcal{H}}^2 \leq \varepsilon.$
\end{enumerate}
\end{df}

By the classical duality argument, to study the controllability of \eqref{coeqq1.1absconback}, we introduce the following adjoint forward-backward system
\begin{equation}\label{coeqq1.adjointback}
\begin{cases}
\begin{array}{l}
d\phi = \left[\mathcal{A}^*\phi-F(t)^*\phi+\sum_{i=1}^m\alpha_i K_i^2\psi_i\right]dt +\left[-G(t)^*\phi+\sum_{i=1}^m\widetilde{\alpha}_i \widetilde{K}_i^2\Psi_i\right] dW(t)\quad \textnormal{in} \,\,(0,T],\\
d\psi_i = \left[-\mathcal{A}\psi_i+F(t)\psi_i+G(t)\Psi_i+\frac{1}{\beta_i}B_iR_i^{-1}B^*_i\phi\right]dt + \Psi_i dW(t)\quad \textnormal{in} \,\,[0,T),\\
\phi(0)=\phi_0, \quad \psi_i(T)=0,\quad i=1,2,\dots,m.
\end{array}
\end{cases}
\end{equation}

Similarly to the proof of Proposition \ref{propp4.1dual}, we show the following duality relation between the solutions of the systems \eqref{coeqq1.1absconback} and \eqref{coeqq1.adjointback}.
\begin{prop}
The solutions of \eqref{coeqq1.1absconback} and \eqref{coeqq1.adjointback} satisfy that for all $y_T \in L^2_{\mathcal{F}_T}(\Omega; \mathcal{H})$ and $\phi_0 \in L^2_{\mathcal{F}_0}(\Omega; \mathcal{H})$, we have that
\begin{align}\label{dulirelatback}
    \begin{aligned}
\mathbb{E}\langle y_T,\phi(T)\rangle_{\mathcal{H}} - \mathbb{E}\langle y(0),\phi_0\rangle_{\mathcal{H}} = &\, \mathbb{E}\int_0^T \langle u_1, D_1^*\phi \rangle_{\mathcal{U}_1} dt + \mathbb{E}\int_0^T \left\langle u_2, D_2^*\left[\sum_{i=1}^m \widetilde{\alpha}_i \widetilde{K}_i^2 \Psi_i - G(t)^*\phi \right]\right\rangle_{\mathcal{U}_2} dt \\
& + \sum_{i=1}^m \alpha_i \mathbb{E}\int_0^T \langle y_{i,d}, K_i^2 \psi_i \rangle_{\mathcal{H}} dt + \sum_{i=1}^m \widetilde{\alpha}_i \mathbb{E}\int_0^T \langle Y_{i,d}, \widetilde{K}_i^2 \Psi_i \rangle_{\mathcal{L}^0_2} dt.
\end{aligned}
\end{align}
\end{prop}

We have the following characterization of the exact controllability of \eqref{coeqq1.1absconback}.
\begin{thm}\label{thmm4.1exback}
System \eqref{coeqq1.1absconback} is exactly controllable at time $0$ if and only if there exists a constant $C>0$ such that the solutions of \eqref{coeqq1.adjointback} satisfy the following (exact) observability inequality:
\begin{align}\label{observaineqexaback}
\begin{aligned}
&\,\mathbb{E}|\phi_0|^2_{\mathcal{H}} + \sum_{i=1}^m \mathbb{E} \int_0^T  \rho^{-2}|K_i\psi_i|_{\mathcal{H}}^2 dt + \sum_{i=1}^m \mathbb{E} \int_0^T  \rho^{-2}|\widetilde{K}_i\Psi_i|_{\mathcal{L}^0_2}^2 dt \\
&\leq C \mathbb{E} \int_0^T \left(|D_1^*\phi|_{\mathcal{U}_1}^2 + \left|D_2^*\left(\sum_{i=1}^m\widetilde{\alpha}_i\widetilde{K}_i^2\Psi_i-G(t)^*\phi\right)\right|_{\mathcal{U}_2}^2 \right) dt.
\end{aligned}
\end{align}
\end{thm}
\begin{proof}
We assume that the system \eqref{coeqq1.1absconback} is exactly controllable at time $0$. Using \eqref{dulirelatback} and \eqref{estimm1back}, it is easy to see that for any $\tau>0$,
\begin{align}\label{ineqqobsnullnull}
    \begin{aligned}
&\,\mathbb{E}\langle y_0,\phi_0\rangle_{\mathcal{H}} + \sum_{i=1}^m \alpha_i \mathbb{E}\int_0^T \langle y_{i,d},K_i^2\psi_i\rangle_{\mathcal{H}} dt + \sum_{i=1}^m \widetilde{\alpha}_i \mathbb{E}\int_0^T \langle Y_{i,d},\widetilde{K}_i^2\Psi_i\rangle_{\mathcal{L}^0_2} dt \\
&\leq C\tau \bigg(\mathbb{E} |y_T|^2_{\mathcal{H}} + \mathbb{E} |y_0|^2_{\mathcal{H}} + \sum_{i=1}^m \mathbb{E} \int_0^T  \rho^{2}|K_iy_{i,d}|_{\mathcal{H}}^2 \,dt + \sum_{i=1}^m \mathbb{E} \int_0^T  \rho^{2}|\widetilde{K}_i Y_{i,d}|_{\mathcal{L}^0_2}^2 \,dt \bigg) \\
&\quad + \frac{1}{4\tau}\left( |D_1^*\phi|_{L^2_\mathcal{F}(0,T;\mathcal{U}_1)}^2 + \left| D_2^*\left( \sum_{i=1}^m \widetilde{\alpha}_i \widetilde{K}_i^2 \Psi_i - G(t)^*\phi \right) \right|_{L^2_\mathcal{F}(0,T;\mathcal{U}_2)}^2 \right) \\
&\quad + \mathbb{E}\langle y_T,\phi(T)\rangle_{\mathcal{H}}.
\end{aligned}
\end{align}
Choosing $y_T \equiv 0$, $y_0 \equiv \phi_0$, $y_{i,d} \equiv \frac{1}{\alpha_i} \rho^{-2} \psi_i$, and $Y_{i,d} \equiv \frac{1}{\widetilde{\alpha}_i} \rho^{-2} \Psi_i$ in \eqref{ineqqobsnullnull}, we get
\begin{align}\label{ineqqobs2nullnu}
    \begin{aligned}
&\,\mathbb{E}|\phi_0|^2_{\mathcal{H}} + \sum_{i=1}^m \mathbb{E}\int_0^T \rho^{-2} |K_i \psi_i|^2_{\mathcal{H}} dt + \sum_{i=1}^m \mathbb{E}\int_0^T \rho^{-2} |\widetilde{K}_i \Psi_i|^2_{\mathcal{L}^0_2} dt \\
&\leq C\tau \left( \mathbb{E} |\phi_0|^2_{\mathcal{H}} + \sum_{i=1}^m \frac{1}{\alpha_i^2} \mathbb{E} \int_0^T \rho^{-2} |K_i \psi_i|_{\mathcal{H}}^2 \,dt + \sum_{i=1}^m \frac{1}{\widetilde{\alpha}_i^2} \mathbb{E} \int_0^T \rho^{-2} |\widetilde{K}_i \Psi_i|_{\mathcal{L}^0_2}^2 \,dt \right) \\
&\quad + \frac{1}{4\tau} \left( |D_1^*\phi|_{L^2_\mathcal{F}(0,T;\mathcal{U}_1)}^2 + \left| D_2^*\left( \sum_{i=1}^m \widetilde{\alpha}_i \widetilde{K}_i^2 \Psi_i - G(t)^*\phi \right) \right|_{L^2_\mathcal{F}(0,T;\mathcal{U}_2)}^2 \right).
\end{aligned}
\end{align}
Taking a small enough $\tau$ in \eqref{ineqqobs2nullnu}, we obtain the observability inequality \eqref{observaineqexaback}.

Let us now suppose that the estimate \eqref{observaineqexaback} holds. Let $y_T \in L^2_{\mathcal{F}_T}(\Omega; \mathcal{H})$, $y_0 \in L^2_{\mathcal{F}_0}(\Omega; \mathcal{H})$, $y_{i,d} \in L^{2,\rho}_{\mathcal{F}}(0,T; \mathcal{H})$, and $Y_{i,d} \in L^{2,\rho}_{\mathcal{F}}(0,T; \mathcal{L}^0_2)$. We define the following functional on $L^2_{\mathcal{F}_0}(\Omega; \mathcal{H})$ as:
\begin{align*}
J^{ext}(\phi_0) = & \frac{1}{2} \mathbb{E} \int_0^T \left( |D_1^*\phi|_{\mathcal{U}_1}^2 + \left| D_2^* \left( \sum_{i=1}^m \widetilde{\alpha}_i \widetilde{K}_i^2 \Psi_i - G(t)^*\phi \right) \right|_{\mathcal{U}_2}^2 \right) dt \\
&-\mathbb{E} \langle y_T, \phi(T) \rangle_{\mathcal{H}}+ \mathbb{E} \langle y_0, \phi_0 \rangle_{\mathcal{H}}  + \sum_{i=1}^m \alpha_i\mathbb{E} \int_0^T \langle y_{i,d}, K_i^2 \psi_i \rangle_{\mathcal{H}} dt\\
& + \sum_{i=1}^m \widetilde{\alpha}_i\mathbb{E} \int_0^T \langle Y_{i,d}, \widetilde{K}_i^2 \Psi_i \rangle_{\mathcal{L}^0_2} dt.
\end{align*}
It is easy to verify that $J^{ext}: L^2_{\mathcal{F}_0}(\Omega; \mathcal{H}) \longrightarrow \mathbb{R}$ is continuous and strictly convex. Proceeding as in the proof of Theorem \ref{thmm4.1ex}, by applying the observability inequality \eqref{observaineqexaback}, we prove that the functional $J^{ext}$ is coercive. Thus, it admits a unique minimum, denoted by $\widehat{\phi}_0 \in L^2_{\mathcal{F}_0}(\Omega; \mathcal{H})$. Consequently, the Euler equation implies that for all $\phi_0 \in L^2_{\mathcal{F}_0}(\Omega; \mathcal{H})$,
\begin{align}\label{idenfronjvareback}
    \begin{aligned}
&\mathbb{E}\int_0^T \langle D_1^*\widehat{\phi}, D_1^*\phi \rangle_{\mathcal{U}_1} dt + \mathbb{E}\int_0^T \left\langle D_2^*\left( \sum_{i=1}^m \widetilde{\alpha}_i \widetilde{K}_i^2 \widehat{\Psi}_i - G(t)^*\widehat{\phi} \right), D_2^*\left( \sum_{i=1}^m \widetilde{\alpha}_i \widetilde{K}_i^2 \Psi_i - G(t)^*\phi \right) \right\rangle_{\mathcal{U}_2} dt \\
&-\mathbb{E}\langle y_T, \phi(T) \rangle_{\mathcal{H}} + \mathbb{E}\langle y_0, \phi_0 \rangle_{\mathcal{H}} + \sum_{i=1}^m \alpha_i \mathbb{E}\int_0^T \langle y_{i,d}, K_i^2 \psi_i \rangle_{\mathcal{H}} dt \\
&+ \sum_{i=1}^m \widetilde{\alpha}_i \mathbb{E}\int_0^T \langle Y_{i,d}, \widetilde{K}_i^2 \Psi_i \rangle_{\mathcal{L}^0_2} dt = 0,
    \end{aligned}
\end{align}
where $(\widehat{\phi}; \widehat{\psi}_i, \widehat{\Psi}_i)$ is the solution of \eqref{coeqq1.adjointback} associated with the initial state $\widehat{\phi}_0$.

We choose the controls $(\widehat{u}_1, \widehat{u}_2) = \left( D_1^* \widehat{\phi}, D_2^* \left( \sum_{i=1}^m \widetilde{\alpha}_i \widetilde{K}_i^2 \widehat{\Psi}_i - G(t)^* \widehat{\phi} \right) \right)$ in \eqref{dulirelatback}, we have that
\begin{align}\label{dulirelatexactback}
\begin{aligned}
&-\mathbb{E}\langle y_T, \phi(T) \rangle_{\mathcal{H}} + \mathbb{E} \langle \widehat{y}(0), \phi_0 \rangle_{\mathcal{H}} + \mathbb{E} \int_0^T \langle D_1^* \widehat{\phi}, D_1^* \phi \rangle_{\mathcal{U}_1} \, dt \\
&+ \mathbb{E} \int_0^T \left\langle D_2^* \left( \sum_{i=1}^m \widetilde{\alpha}_i \widetilde{K}_i^2 \widehat{\Psi}_i - G(t)^* \widehat{\phi} \right), D_2^* \left( \sum_{i=1}^m \widetilde{\alpha}_i \widetilde{K}_i^2 \Psi_i - G(t)^* \phi \right) \right\rangle_{\mathcal{U}_2} \, dt \\
&+ \sum_{i=1}^m \alpha_i \mathbb{E} \int_0^T \langle y_{i,d}, K_i^2 \psi_i \rangle_{\mathcal{H}} \, dt + \sum_{i=1}^m \widetilde{\alpha}_i \mathbb{E} \int_0^T \langle Y_{i,d}, \widetilde{K}_i^2 \Psi_i \rangle_{\mathcal{L}^0_2} \, dt = 0,
\end{aligned}
\end{align}
where $(\widehat{y}, \widehat{Y}; \widehat{z}_i)$ is the solution of \eqref{coeqq1.1absconback} associated with the controls $(\widehat{u}_1, \widehat{u}_2)$.  By combining \eqref{idenfronjvareback} and \eqref{dulirelatexactback}, we deduce that
\begin{align*}
    \mathbb{E} \langle \widehat{y}(0) - y_0, \phi_0 \rangle_{\mathcal{H}} = 0, \quad \forall \, \phi_0 \in L^2_{\mathcal{F}_0}(\Omega; \mathcal{H}),
\end{align*}
which leads to
\begin{align*}
\widehat{y}(0) = y_0, \quad \text{in} \, L^2_{\mathcal{F}_0}(\Omega; \mathcal{H}).
\end{align*}
On the other hand, taking $\phi_0 = \widehat{\phi}_0$ in \eqref{idenfronjvareback}, and using the observability inequality \eqref{observaineqexaback}, we deduce that
\begin{align*}
\begin{aligned}
|\widehat{u}_1|^2_{L^2_\mathcal{F}(0,T; \mathcal{U}_1)} + |\widehat{u}_2|^2_{L^2_\mathcal{F}(0,T; \mathcal{U}_2)} \leq C \bigg( &\mathbb{E} |y_T|^2_{\mathcal{H}} + \mathbb{E} |y_0|^2_{\mathcal{H}} + \sum_{i=1}^m \mathbb{E} \int_0^T \rho^2 |K_i y_{i,d}|_{\mathcal{H}}^2 \, dt \\
&+ \sum_{i=1}^m \mathbb{E} \int_0^T \rho^2 |\widetilde{K}_i Y_{i,d}|_{\mathcal{L}^0_2}^2 \, dt \bigg).
\end{aligned}
\end{align*}
This completes the proof of Theorem \ref{thmm4.1exback}.
\end{proof}

For the null controllability, we have the following characterization.
\begin{thm}\label{thmm4.2exback}
System \eqref{coeqq1.1absconback} is null controllable at time $0$ if and only if there exists a constant $C > 0$ such that the solutions of \eqref{coeqq1.adjointback} satisfy the following (null) observability inequality:
\begin{align}\label{observaineqnuuback}
\begin{aligned}
&\mathbb{E}|\phi(T)|^2_{\mathcal{H}} + \sum_{i=1}^m \mathbb{E} \int_0^T \rho^{-2} |K_i\psi_i|_{\mathcal{H}}^2 \, dt + \sum_{i=1}^m \mathbb{E} \int_0^T \rho^{-2}|\widetilde{K}_i\Psi_i|_{\mathcal{L}^0_2}^2 \, dt \\
&\leq C \mathbb{E} \int_0^T \left(|D_1^*\phi|_{\mathcal{U}_1}^2 + \left|D_2^*\left(\sum_{i=1}^m\widetilde{\alpha}_i\widetilde{K}_i^2\Psi_i-G(t)^*\phi\right)\right|_{\mathcal{U}_2}^2 \right) \, dt.
\end{aligned}
\end{align}
\end{thm}
\begin{proof}
Assume that the system \eqref{coeqq1.1absconback} is null controllable at time $0$. From \eqref{dulirelatback} and \eqref{estimm2back}, it is easy to see that for any $\tau > 0$,
\begin{align}\label{ineqqobsback}
    \begin{aligned}
&\,\mathbb{E}\langle y_T, \phi(T)\rangle_{\mathcal{H}} - \sum_{i=1}^m \alpha_i \mathbb{E} \int_0^T \langle y_{i,d}, K_i^2 \psi_i \rangle_{\mathcal{H}} \, dt - \sum_{i=1}^m \widetilde{\alpha}_i \mathbb{E} \int_0^T \langle Y_{i,d}, \widetilde{K}_i^2 \Psi_i \rangle_{\mathcal{L}^0_2} \, dt \\
&\leq C \tau \bigg( \mathbb{E} |y_T|^2_{\mathcal{H}} + \sum_{i=1}^m \mathbb{E} \int_0^T \rho^2 |K_i y_{i,d}|_{\mathcal{H}}^2 \, dt + \sum_{i=1}^m \mathbb{E} \int_0^T \rho^2 |\widetilde{K}_i Y_{i,d}|_{\mathcal{L}^0_2}^2 \, dt \bigg) \\
&\quad + \frac{1}{4\tau} \left( |D_1^* \phi|_{L^2_\mathcal{F}(0,T;\mathcal{U}_1)}^2 + \left| D_2^* \left( \sum_{i=1}^m \widetilde{\alpha}_i \widetilde{K}_i^2 \Psi_i - G(t)^* \phi \right) \right|_{L^2_\mathcal{F}(0,T;\mathcal{U}_2)}^2 \right).
\end{aligned}
\end{align}
Set $y_T \equiv \phi(T)$, $y_{i,d} \equiv -\frac{1}{\alpha_i} \rho^{-2} \psi_i$ and $Y_{i,d} \equiv -\frac{1}{\widetilde{\alpha}_i} \rho^{-2} \Psi_i$ in \eqref{ineqqobsback}, we get
\begin{align}\label{ineqqobs2nullsecnu}
    \begin{aligned}
&\, \mathbb{E} |\phi(T)|^2_{\mathcal{H}} + \sum_{i=1}^m \mathbb{E} \int_0^T \rho^{-2} |K_i \psi_i|^2_{\mathcal{H}} \, dt + \sum_{i=1}^m \mathbb{E} \int_0^T \rho^{-2} |\widetilde{K}_i \Psi_i|^2_{\mathcal{L}^0_2} \, dt \\
&\leq C \tau \left( \mathbb{E} |\phi(T)|^2_{\mathcal{H}} + \sum_{i=1}^m \frac{1}{\alpha_i^2} \mathbb{E} \int_0^T \rho^{-2} |K_i \psi_i|_{\mathcal{H}}^2 \, dt + \sum_{i=1}^m \frac{1}{\widetilde{\alpha}_i^2} \mathbb{E} \int_0^T \rho^{-2} |\widetilde{K}_i \Psi_i|_{\mathcal{L}^0_2}^2 \, dt \right) \\
&\quad + \frac{1}{4\tau} \left( |D_1^* \phi|_{L^2_\mathcal{F}(0,T;\mathcal{U}_1)}^2 + \left| D_2^* \left( \sum_{i=1}^m \widetilde{\alpha}_i \widetilde{K}_i^2 \Psi_i - G(t)^* \phi \right) \right|_{L^2_\mathcal{F}(0,T;\mathcal{U}_2)}^2 \right).
\end{aligned}
\end{align}
By choosing a sufficiently small $\tau$ in \eqref{ineqqobs2nullsecnu}, we obtain the observability inequality \eqref{observaineqnuuback}.

Let us now suppose that the inequality \eqref{observaineqnuuback} holds. Let $y_T \in L^2_{\mathcal{F}_T}(\Omega; \mathcal{H})$, $y_{i,d} \in L^{2,\rho}_{\mathcal{F}}(0,T; \mathcal{H})$, $Y_{i,d} \in L^{2,\rho}_{\mathcal{F}}(0,T; \mathcal{L}^0_2)$, and $\varepsilon > 0$. We define the following functional on $L^2_{\mathcal{F}_0}(\Omega; \mathcal{H})$ as follows:
\begin{align*}
J^{nul}_\varepsilon(\phi_0) = & \frac{1}{2} \mathbb{E} \int_0^T \left( |D_1^* \phi|_{\mathcal{U}_1}^2 + \left| D_2^* \left( \sum_{i=1}^m \widetilde{\alpha}_i \widetilde{K}_i^2 \Psi_i - G(t)^* \phi \right) \right|_{\mathcal{U}_2}^2 \right) dt \\
& - \mathbb{E} \langle y_T, \phi(T) \rangle_{\mathcal{H}} + \varepsilon |\phi_0|_{L^2_{\mathcal{F}_0}(\Omega; \mathcal{H})} + \sum_{i=1}^m \alpha_i \mathbb{E} \int_0^T \langle y_{i,d}, K_i^2 \psi_i \rangle_{\mathcal{H}} dt \\
& + \sum_{i=1}^m \widetilde{\alpha}_i \mathbb{E} \int_0^T \langle Y_{i,d}, \widetilde{K}_i^2 \Psi_i \rangle_{\mathcal{L}^0_2} dt.
\end{align*}
It is easy to verify that $J^{nul}_\varepsilon: L^2_{\mathcal{F}_0}(\Omega; \mathcal{H}) \to \mathbb{R}$ is continuous and strictly convex. Applying the observability inequality \eqref{observaineqnuuback}, it follows that the functional $J^{nul}_\varepsilon$ is also coercive. Therefore, $J^{nul}_\varepsilon$ admits a unique minimum $\phi_0^\varepsilon \in L^2_{\mathcal{F}_0}(\Omega; \mathcal{H})$. If $\phi_0^\varepsilon \neq 0$, by the Euler equation, we deduce that for all $\phi_0 \in L^2_{\mathcal{F}_0}(\Omega; \mathcal{H})$,
\begin{align}\label{idenfronjvarenullback}
    \begin{aligned}
& \mathbb{E} \int_0^T \langle D_1^* \phi^\varepsilon, D_1^* \phi \rangle_{\mathcal{U}_1} dt + \mathbb{E} \int_0^T \left\langle D_2^* \left( \sum_{i=1}^m \widetilde{\alpha}_i \widetilde{K}_i^2 \Psi_i^\varepsilon - G(t)^* \phi^\varepsilon \right), D_2^* \left( \sum_{i=1}^m \widetilde{\alpha}_i \widetilde{K}_i^2 \Psi_i - G(t)^* \phi \right) \right\rangle_{\mathcal{U}_2} dt \\
& - \mathbb{E} \langle y_T, \phi(T) \rangle_{\mathcal{H}} + \varepsilon \frac{\langle \phi_0^\varepsilon, \phi_0 \rangle_{L^2_{\mathcal{F}_0}(\Omega; \mathcal{H})}}{|\phi_0^\varepsilon|_{L^2_{\mathcal{F}_0}(\Omega; \mathcal{H})}} + \sum_{i=1}^m \alpha_i \mathbb{E} \int_0^T \langle y_{i,d}, K_i^2 \psi_i \rangle_{\mathcal{H}} dt \\
& + \sum_{i=1}^m \widetilde{\alpha}_i \mathbb{E} \int_0^T \langle Y_{i,d}, \widetilde{K}_i^2 \Psi_i \rangle_{\mathcal{L}^0_2} dt = 0,
    \end{aligned}
\end{align}
where $(\phi^\varepsilon; \psi_i^\varepsilon, \Psi_i^\varepsilon)$ is the solution of \eqref{coeqq1.adjointback} associated with the initial state $\phi_0^\varepsilon$.

We now choose the controls \( (u_1^\varepsilon, u_2^\varepsilon) = \left( D_1^* \phi^\varepsilon, D_2^* \left( \sum_{i=1}^m \widetilde{\alpha}_i \widetilde{K}_i^2 \Psi_i^\varepsilon - G(t)^* \phi^\varepsilon \right) \right) \) in \eqref{dulirelatback}. We then have the following:
\begin{align}\label{dulirelatnullback}
\begin{aligned}
& - \mathbb{E} \langle y_T, \phi(T) \rangle_{\mathcal{H}} + \mathbb{E} \langle y_\varepsilon(0), \phi_0 \rangle_{\mathcal{H}} + \mathbb{E} \int_0^T \langle D_1^* \phi^\varepsilon, D_1^* \phi \rangle_{\mathcal{U}_1} \, dt \\
& + \mathbb{E} \int_0^T \left\langle D_2^* \left( \sum_{i=1}^m \widetilde{\alpha}_i \widetilde{K}_i^2 \Psi_i^\varepsilon - G(t)^* \phi^\varepsilon \right), D_2^* \left( \sum_{i=1}^m \widetilde{\alpha}_i \widetilde{K}_i^2 \Psi_i - G(t)^* \phi \right) \right\rangle_{\mathcal{U}_2} \, dt \\
& + \sum_{i=1}^m \alpha_i \mathbb{E} \int_0^T \langle y_{i,d}, K_i^2 \psi_i \rangle_{\mathcal{H}} \, dt + \sum_{i=1}^m \widetilde{\alpha}_i \mathbb{E} \int_0^T \langle Y_{i,d}, \widetilde{K}_i^2 \Psi_i \rangle_{\mathcal{L}^0_2} \, dt = 0,
\end{aligned}
\end{align}
where \( (y_\varepsilon, Y_\varepsilon; z_i^\varepsilon) \) is the solution of \eqref{coeqq1.1absconback} associated with the controls \( (u_1^\varepsilon, u_2^\varepsilon) \). Combining \eqref{idenfronjvarenullback} and \eqref{dulirelatnullback}, we deduce that
\[
-\mathbb{E} \langle y_\varepsilon(0), \phi_0 \rangle_{\mathcal{H}} + \varepsilon \frac{\langle \phi_0^\varepsilon, \phi_0 \rangle_{L^2_{\mathcal{F}_0}(\Omega; \mathcal{H})}}{|\phi_0^\varepsilon|_{L^2_{\mathcal{F}_0}(\Omega; \mathcal{H})}} = 0, \quad \forall \phi_0 \in L^2_{\mathcal{F}_0}(\Omega; \mathcal{H}),
\]
which implies that
\begin{equation}\label{apprforexactconnullback}
|y_\varepsilon(0)|_{L^2_{\mathcal{F}_0}(\Omega; \mathcal{H})} \leq \varepsilon.
\end{equation}
Taking \( \phi_0 = \phi_0^\varepsilon \) in \eqref{idenfronjvarenullback}, and using the observability inequality \eqref{observaineqnuuback}, we obtain that
\begin{align}\label{estimm1forexatnullback}
\begin{aligned}
|u_1^\varepsilon|^2_{L^2_\mathcal{F}(0,T; \mathcal{U}_1)} + |u_2^\varepsilon|^2_{L^2_\mathcal{F}(0,T; \mathcal{U}_2)} \leq C \Bigg( & \mathbb{E} |y_T|^2_{\mathcal{H}} + \sum_{i=1}^m \mathbb{E} \int_0^T \rho^2 |K_i y_{i,d}|_{\mathcal{H}}^2 \, dt \\
& + \sum_{i=1}^m \mathbb{E} \int_0^T \rho^2 | \widetilde{K}_i Y_{i,d} |_{\mathcal{L}^0_2}^2 \, dt \Bigg).
\end{aligned}
\end{align}
If \( \phi_0^\varepsilon = 0 \), it is easy to see that the inequalities \eqref{apprforexactconnullback} and \eqref{estimm1forexatnullback} hold. Therefore, by \eqref{estimm1forexatnullback}, we deduce the existence of a subsequence \( (u_1^\varepsilon, u_2^\varepsilon) \) such that, as \( \varepsilon \to 0 \),
\begin{align}\label{weakconvrnullback}
\begin{aligned}
u_1^\varepsilon & \longrightarrow \widehat{u}_1 \quad \text{weakly in} \, L^2((0,T) \times \Omega; \mathcal{U}_1), \\
u_2^\varepsilon & \longrightarrow \widehat{u}_2 \quad \text{weakly in} \, L^2((0,T) \times \Omega; \mathcal{U}_2).
\end{aligned}
\end{align}
Thus, we deduce that
\begin{equation}\label{Eq56nullback}
y_\varepsilon(0) \longrightarrow \widehat{y}(0) \quad \text{weakly in} \, L^2_{\mathcal{F}_0}(\Omega; \mathcal{H}), \quad \text{as} \, \varepsilon \to 0,
\end{equation}
where the state \( (\widehat{y}, \widehat{Y}) \) is the one given by the unique solution \( (\widehat{y}, \widehat{Y}; \widehat{z}_i) \) of the system \eqref{coeqq1.1absconback} associated with the controls \( (\widehat{u}_1, \widehat{u}_2) \). Finally, combining \eqref{apprforexactconnullback} and \eqref{Eq56nullback}, we conclude that
\[
\widehat{y}(0) = 0, \quad \mathbb{P}\text{-a.s.}
\]
Moreover, the inequality \eqref{estimm2back} can be easily derived from \eqref{estimm1forexatnullback} and \eqref{weakconvrnullback}, which provides the null controllability of the system \eqref{coeqq1.1absconback}. This completes the proof of Theorem \ref{thmm4.2exback}.
\end{proof}

For the case of approximate controllability, we have the following characterization.
\begin{thm}\label{thmmapp43appback}
System \eqref{coeqq1.1absconback} is approximately controllable at time \(0\) if and only if the solutions of \eqref{coeqq1.adjointback} satisfy the following unique continuation property:
\begin{align}\label{observainuniqcback}
D_1^*\phi = D_2^*\left[\sum_{i=1}^m\widetilde{\alpha}_i\widetilde{K}_i^2\Psi_i - G(t)^*\phi\right] = 0, \quad \forall t \in (0,T), \quad \mathbb{P}\textnormal{-a.s.} \quad \Longrightarrow \quad \phi_0 = 0, \quad \mathbb{P}\textnormal{-a.s.}
\end{align}
\end{thm}
\begin{proof}
Assume that \eqref{coeqq1.1absconback} is approximately controllable at time \(0\), and define the following operator:
\begin{equation*}
\mathcal{L}_0: \mathcal{U}_T \longrightarrow L^2_{\mathcal{F}_0}(\Omega; \mathcal{H}), \quad (u_1, u_2) \longmapsto y(0),
\end{equation*}
where the state \((y,Y)\) is given by the unique solution \((y,Y;z_i)\) of the system \eqref{coeqq1.1absconback} with \(y_T \equiv 0\) and \(y_{i,d} \equiv Y_{i,d} \equiv 0\) (\(i = 1, 2, \dots, m\)).  Then, it is easy to see that \(\mathcal{R}(\mathcal{L}_0)\) is dense in \(L^2_{\mathcal{F}_0}(\Omega; \mathcal{H})\), which is equivalent to the adjoint operator \(\mathcal{L}_0^*\) being injective. On the other hand, from \eqref{dulirelatback}, it is straightforward to deduce that
\[
\langle \mathcal{L}_0(u_1, u_2), \phi_0 \rangle_{L^2_{\mathcal{F}_0}(\Omega; \mathcal{H})} = \left\langle (u_1, u_2), \left( -D_1^*\phi, -D_2^*\left[ \sum_{i=1}^m \widetilde{\alpha}_i \widetilde{K}_i^2 \Psi_i - G(t)^*\phi \right] \right) \right\rangle_{\mathcal{U}_T},
\]
which implies that \(\mathcal{L}_0^*(\phi_0) = \left( -D_1^*\phi, -D_2^*\left[ \sum_{i=1}^m \widetilde{\alpha}_i \widetilde{K}_i^2 \Psi_i - G(t)^*\phi \right] \right)\). Therefore, the unique continuation property \eqref{observainuniqcback} holds.

We now suppose that the property \eqref{observainuniqcback} holds. Let \( y_T \in L^2_{\mathcal{F}_T}(\Omega; \mathcal{H}) \), \( y_0 \in L^2_{\mathcal{F}_0}(\Omega; \mathcal{H}) \), \( y_{i,d} \in L^{2}_{\mathcal{F}}(0,T; \mathcal{H}) \), \( Y_{i,d} \in L^{2}_{\mathcal{F}}(0,T; \mathcal{L}^0_2) \), and \( \varepsilon > 0 \), we define the following functional on \( L^2_{\mathcal{F}_0}(\Omega; \mathcal{H}) \) as follows:
\begin{align*}
J^{app}_\varepsilon(\phi_0) = & \frac{1}{2} \mathbb{E} \int_0^T \left( |D_1^* \phi|_{\mathcal{U}_1}^2 + \left| D_2^* \left( \sum_{i=1}^m \widetilde{\alpha}_i \widetilde{K}_i^2 \Psi_i - G(t)^* \phi \right) \right|_{\mathcal{U}_2}^2 \right) dt- \mathbb{E} \langle y_T, \phi(T) \rangle_{\mathcal{H}} + \mathbb{E} \langle y_0, \phi_0 \rangle_{\mathcal{H}} \\
& + \varepsilon | \phi_0 |_{L^2_{\mathcal{F}_0}(\Omega; \mathcal{H})} + \sum_{i=1}^m \alpha_i \mathbb{E} \int_0^T \langle y_{i,d}, K_i^2 \psi_i \rangle_{\mathcal{H}} dt + \sum_{i=1}^m \widetilde{\alpha}_i \mathbb{E} \int_0^T \langle Y_{i,d}, \widetilde{K}_i^2 \Psi_i \rangle_{\mathcal{L}^0_2} dt.
\end{align*}
It is easy to see that \( J^{app}_\varepsilon: L^2_{\mathcal{F}_0}(\Omega; \mathcal{H}) \longrightarrow \mathbb{R} \) is continuous and strictly convex. Similarly to the proof of Theorem \ref{thmmapp43app}, we show that the functional \( J_\varepsilon^{app} \) is coercive in \( L^2_{\mathcal{F}_0}(\Omega; \mathcal{H}) \). Therefore, it admits a unique minimum \( \phi_0^\varepsilon \in L^2_{\mathcal{F}_0}(\Omega; \mathcal{H}) \).  If \( \phi_0^\varepsilon \neq 0 \), by the Euler equation: for all \( \phi_0 \in L^2_{\mathcal{F}_0}(\Omega; \mathcal{H}) \), we have that
\begin{align}\label{idenfronjvareappbbackapp}
  \begin{aligned}
    & \mathbb{E} \int_0^T \langle D_1^* \phi^\varepsilon, D_1^* \phi \rangle_{\mathcal{U}_1} dt + \mathbb{E} \int_0^T \left\langle D_2^* \left( \sum_{i=1}^m \widetilde{\alpha}_i \widetilde{K}_i^2 \Psi^\varepsilon_i - G(t)^* \phi^\varepsilon \right), D_2^* \left( \sum_{i=1}^m \widetilde{\alpha}_i \widetilde{K}_i^2 \Psi_i - G(t)^* \phi \right) \right\rangle_{\mathcal{U}_2} dt \\
    & - \mathbb{E} \langle y_T, \phi(T) \rangle_{\mathcal{H}}+ \mathbb{E} \langle y_0, \phi_0 \rangle_{\mathcal{H}} + \varepsilon \frac{\langle \phi_0^\varepsilon, \phi_0 \rangle_{L^2_{\mathcal{F}_0}(\Omega; \mathcal{H})}}{|\phi_0^\varepsilon|_{L^2_{\mathcal{F}_0}(\Omega; \mathcal{H})}} + \sum_{i=1}^m \alpha_i \mathbb{E} \int_0^T \langle y_{i,d}, K_i^2 \psi_i \rangle_{\mathcal{H}} dt \\
    & + \sum_{i=1}^m \widetilde{\alpha}_i \mathbb{E} \int_0^T \langle Y_{i,d}, \widetilde{K}_i^2 \psi_i \rangle_{\mathcal{L}^0_2} dt = 0,
  \end{aligned}
\end{align}
where \( (\phi^\varepsilon; \psi_i^\varepsilon, \Psi_i^\varepsilon) \) is the solution of \eqref{coeqq1.adjointback} associated with the initial state \( \phi_0^\varepsilon \).

We now choose the controls \( (u_1^\varepsilon, u_2^\varepsilon) = \left( D_1^* \phi^\varepsilon, D_2^* \left( \sum_{i=1}^m \widetilde{\alpha}_i \widetilde{K}_i^2 \Psi^\varepsilon_i - G(t)^* \phi^\varepsilon \right) \right) \) in \eqref{dulirelatback}, and we have that
\begin{align}\label{dulirelatexactbackapp}
\begin{aligned}
    & -\mathbb{E} \langle y_T, \phi(T) \rangle_{\mathcal{H}} + \mathbb{E} \langle y_\varepsilon(0), \phi_0 \rangle_{\mathcal{H}} + \mathbb{E} \int_0^T \langle D_1^* \phi^\varepsilon, D_1^* \phi \rangle_{\mathcal{U}_1} dt \\
    & + \mathbb{E} \int_0^T \left\langle D_2^* \left( \sum_{i=1}^m \widetilde{\alpha}_i \widetilde{K}_i^2 \Psi^\varepsilon_i - G(t)^* \phi^\varepsilon \right), D_2^* \left( \sum_{i=1}^m \widetilde{\alpha}_i \widetilde{K}_i^2 \Psi_i - G(t)^* \phi \right) \right\rangle_{\mathcal{U}_2} dt \\
    & + \sum_{i=1}^m \alpha_i \mathbb{E} \int_0^T \langle y_{i,d}, K_i^2 \psi_i \rangle_{\mathcal{H}} dt + \sum_{i=1}^m \widetilde{\alpha}_i \mathbb{E} \int_0^T \langle Y_{i,d}, \widetilde{K}_i^2 \Psi_i \rangle_{\mathcal{L}^0_2} dt = 0,
\end{aligned}
\end{align}
where \( (y_\varepsilon, Y_\varepsilon; z_i^\varepsilon) \) is the solution of \eqref{coeqq1.1absconback} associated with the controls \( (u_1^\varepsilon, u_2^\varepsilon) \). Combining \eqref{idenfronjvareappbbackapp} and \eqref{dulirelatexactbackapp}, we deduce that
\[
    \mathbb{E} \langle y_0 - y_\varepsilon(0), \phi_0 \rangle_{\mathcal{H}} + \varepsilon \frac{\langle \phi_0^\varepsilon, \phi_0 \rangle_{L^2_{\mathcal{F}_0}(\Omega; \mathcal{H})}}{|\phi_0^\varepsilon|_{L^2_{\mathcal{F}_0}(\Omega; \mathcal{H})}} = 0, \quad \forall \phi_0 \in L^2_{\mathcal{F}_0}(\Omega; \mathcal{H}).
\]
This implies that
\begin{align}\label{apprforexactconappback}
|y_\varepsilon(0) - y_0|_{L^2_{\mathcal{F}_0}(\Omega; \mathcal{H})} \leq \varepsilon.
\end{align}
The estimate \eqref{apprforexactconappback} also holds in the case where \( \phi_0^\varepsilon = 0 \). Therefore, we conclude the proof of Theorem \ref{thmmapp43appback}.
\end{proof}

\section{Illustrative examples}\label{sec55}
In this section, we study the Stackelberg-Nash null controllability problem for two types of equations: the forward and backward stochastic heat equations with Dirichlet boundary conditions. Furthermore, we provide additional details and prove the main result concerning the backward stochastic heat equation.

Let \( T > 0 \), \( \mathcal{G} \subset \mathbb{R}^N \) (\( N \geq 1 \)) be an open, bounded domain with a $C^4$ boundary \( \Gamma \), and let \( m \geq 2 \) be an integer. Let \( \mathcal{G}_0 \), \( \mathcal{G}_i \), \( \mathcal{O}_{i,d} \), and \( \widetilde{\mathcal{O}}_{i,d} \) (for \( i = 1, 2, \dots, m \)) be nonempty open subsets of \( \mathcal{G} \). We denote by \( \mathbbm{1}_{\mathscr{S}} \) the characteristic function of a subset \( \mathscr{S} \subset \mathcal{G} \). Let \( (\Omega, \mathcal{F}, \{\mathcal{F}_t\}_{t \in [0,T]}, \mathbb{P}) \) be a fixed complete filtered probability space on which a one-dimensional standard Brownian motion \( W(\cdot) \) is defined. Here, \( \{\mathcal{F}_t\}_{t \in [0,T]} \) denotes the natural filtration generated by \( W(\cdot) \), and augmented by all the \( \mathbb{P} \)-null sets in \( \mathcal{F} \). We define the following domains:
\[
Q = (0,T) \times \mathcal{G}, \quad \Sigma = (0,T) \times \Gamma, \quad \text{and} \quad Q_0 = (0,T) \times \mathcal{G}_0.
\]
Next, we introduce the following spaces:
\[
\mathcal{U}_T = L^2_\mathcal{F}(0,T; L^2(\mathcal{G}_0)) \times L^2_\mathcal{F}(0,T; L^2(\mathcal{G})),
\]
\[
\mathcal{V}_T = L^2_\mathcal{F}(0,T; L^2(\mathcal{G}_1)) \times \dots \times L^2_\mathcal{F}(0,T; L^2(\mathcal{G}_m)).
\]
\subsection{Application to the forward stochastic heat equation}
We consider the following forward stochastic heat equation
\begin{equation}\label{eqq1.1}
\begin{cases}
dy - \Delta y \, dt = \left[a_1 y + u_1 \mathbbm{1}_{\mathcal{G}_0} + \sum_{i=1}^m v_i \mathbbm{1}_{\mathcal{G}_i} \right] \, dt
+ \left[a_2 y + u_2 \right] \, dW(t) & \text{in } Q, \\
y = 0 & \text{on } \Sigma, \\
y(0) = y_0 & \text{in } \mathcal{G},
\end{cases}
\end{equation}
where \( y \) represents the state variable, \( y_0 \in L^2_{\mathcal{F}_0}(\Omega; L^2(\mathcal{G})) \) is the initial state, and \( a_1, a_2 \in L^\infty_\mathcal{F}(0,T; L^\infty(\mathcal{G})) \) are given functions. The control inputs are \( (u_1, u_2) \in \mathcal{U}_T \) for the leaders, and \( (v_1, \dots, v_m) \in \mathcal{V}_T \) for the followers. It is easy to see that the equation \eqref{eqq1.1} can be written within the framework of the abstract system \eqref{asteqq1.1} under the following conditions:
\begin{align}\label{asummmofabstset}
\left\{
\begin{array}{l}
\text{- The state space: } \mathcal{H} = L^2(\mathcal{G}). \\
\text{- The operator: } \mathcal{A} = \Delta \, \text{(Dirichlet Laplacian)}, \text{ with domain } \mathcal{D}(\mathcal{A}) = H^2(\mathcal{G}) \cap H^1_0(\mathcal{G}),\\
\text{\;\; which generates a } C_0\text{-semigroup on } L^2(\mathcal{G}).\\
\text{- The control spaces: } \mathcal{U}_1 = L^2(\mathcal{G}_0),\, \mathcal{U}_2 = L^2(\mathcal{G}), \text{ and } \mathcal{V}_i = L^2(\mathcal{G}_i) \, \text{for } i = 1, 2, \dots, m. \\
\text{- The control operators: } D_1 = \mathbbm{1}_{\mathcal{G}_0},\, D_2 = \text{Id}, \text{ and } B_i = \mathbbm{1}_{\mathcal{G}_i} \, \text{for } i = 1, 2, \dots, m. \\
\text{- The coefficients: } F(t) = a_1(t, \cdot) \text{Id} \;\, \text{and }\; G(t) = a_2(t, \cdot) \text{Id}.
\end{array}
\right.
\end{align}

For fixed \( y_{i,d}\in L^{2,\rho}_{\mathcal{F}}(0,T; L^2(\mathcal{G})) \), \( i = 1, 2, \dots, m \), representing the target functions, where the weight function \( \rho = \rho(t) \) will be defined below, we now define the following functionals:

\begin{enumerate}[(1)]
    \item For the leaders, the main functional \( J \) is given by
    \begin{equation*}
    J(u_1, u_2) = \frac{1}{2} \mathbb{E} \iint_Q \left( \mathbbm{1}_{\mathcal{G}_0} |u_1|^2 + |u_2|^2 \right) \,dx\,dt.
    \end{equation*}
    
    \item For the followers, the secondary functionals \( J^F_i \)  are defined as follows: For each \( i = 1, 2, \dots, m \),
    
\begin{align}\label{functji1FOR}
    \begin{aligned}
    J^F_i(u_1, u_2; v_1, \dots, v_m) = &\, \frac{\alpha_i}{2} \mathbb{E} \iint_{(0,T) \times \mathcal{O}_{i,d}} |y - y_{i,d}|^2 \,dx\,dt + \frac{\beta_i}{2} \mathbb{E} \iint_{(0,T) \times \mathcal{G}_i} |v_i|^2 \,dx\,dt, 
    \end{aligned}
    \end{align}
\end{enumerate}
where \( \alpha_i > 0 \), \( \beta_i \geq 1 \) are constants, and the state \( y \) is the solution of \eqref{eqq1.1}. By comparing \eqref{functji1FOR}  with our abstract setting, we choose the operators \( K_i \)  and \( R_i \) as follows:
\[
K_i = \mathbbm{1}_{\mathcal{O}_{i,d}} \quad \text{and} \quad R_i = \text{Id}.
\]

From Section \ref{sec3}, we deduce the following result on the Nash equilibrium for the functionals \( (J_1^F,\dots,J_m^F)\).
\begin{prop}\label{proppse5forw}
There exists a large \( \overline{\beta} \geq 1 \) such that, if \( \beta_i \geq \overline{\beta} \) for \( i = 1, 2, \dots, m \), then for each \( (u_1, u_2) \in \mathcal{U}_T \), there exists a unique Nash equilibrium \( (v^*_1(u_1,u_2), \dots, v^*_m(u_1,u_2)) \in \mathcal{V}_T \) for the functionals \( (J_1^F,\dots,J_m^F)\) associated with \( (u_1,u_2) \). Moreover, we have that
    \begin{align*}
    v^*_i = -\frac{1}{\beta_i} z_i \Big|_{(0,T) \times \mathcal{G}_i}, \qquad i = 1, 2, \dots, m,
    \end{align*}
    where \( (z_i, Z_i) \) is the solution of the backward equation
    \begin{equation}\label{backadj}
    \begin{cases}
    \begin{array}{ll}
    dz_i + \Delta z_i \, dt = \big[ -a_1 z_i - a_2 Z_i - \alpha_i(y - y_{i,d}) \mathbbm{1}_{\mathcal{O}_{i,d}} \big] \, dt + Z_i \, dW(t) & \quad \textnormal{in} \, Q, \\
    z_i = 0 & \quad \textnormal{on} \, \Sigma, \\
    z_i(T) = 0 & \quad \textnormal{in} \, \mathcal{G}.
    \end{array}
    \end{cases}
    \end{equation}
\end{prop}
From Proposition \ref{proppse5forw}, the Stackelberg-Nash controllability problem for the equation \eqref{eqq1.1} can be reduced to proving the null controllability for the following coupled forward-backward system
\begin{equation}\label{eqq4.7}
\begin{cases}
\begin{array}{ll}
dy - \Delta y \,dt = \left[a_1y + u_1\mathbbm{1}_{\mathcal{O}} - \sum_{i=1}^m \frac{1}{\beta_i} z_i \mathbbm{1}_{\mathcal{G}_i}\right] \,dt + \left[a_2 y + u_2\right] \,dW(t) & \quad \textnormal{in} \, Q, \\
dz_i + \Delta z_i \,dt = \left[-a_1 z_i - a_2 Z_i - \alpha_i (y - y_{i,d}) \mathbbm{1}_{\mathcal{O}_{i,d}}\right] \,dt + Z_i \,dW(t) & \quad \textnormal{in} \, Q, \\
y = 0, \,\, z_i = 0 & \quad \textnormal{on} \, \Sigma, \\
y(0) = y_0, \,\, z_i(T) = 0, \quad i = 1, 2, \dots, m, & \quad \textnormal{in} \, \mathcal{G}.
\end{array}
\end{cases}
\end{equation}
Then, such a problem is equivalent to an observability inequality for the following adjoint system
\begin{equation}\label{ADJSO1}
\begin{cases}
d\phi + \Delta \phi \,dt = \left[-a_1 \phi - a_2 \Phi + \sum_{i=1}^m \alpha_i \psi_i \mathbbm{1}_{\mathcal{O}_{i,d}}\right] \,dt + \Phi \,dW(t) & \quad \textnormal{in} \, Q, \\
d\psi_i - \Delta \psi_i \,dt = \left[a_1 \psi_i + \frac{1}{\beta_i} \mathbbm{1}_{\mathcal{G}_i} \phi\right] \,dt + a_2 \psi_i \,dW(t) & \quad \textnormal{in} \, Q, \\
\phi = 0, \,\, \psi_i = 0 & \quad \textnormal{on} \, \Sigma, \\
\phi(T) = \phi_T, \,\, \psi_i(0) = 0, \quad i = 1, 2, \dots, m, & \quad \textnormal{in} \, \mathcal{G}.
\end{cases}
\end{equation}

According to Theorem \ref{thmm4.2ex}, we have the following null controllability characterization in the sense of Stackelberg-Nash for the equation \eqref{eqq1.1}.
\begin{prop}\label{propo5.2}
Let \( \rho = \rho(t) \) be a suitable weight function, and let \( \beta_i \geq 1 \) (\( i = 1, 2, \dots, m \)) be sufficiently large. The following assertions are equivalent:
\begin{enumerate}
    \item System \eqref{eqq1.1} is optimally null controllable at time \( T \) in the sense of Stackelberg-Nash.
    \item The solution \( (\phi, \Phi; \psi_i) \) of the adjoint system \eqref{ADJSO1} satisfies the observability inequality:
    \begin{align}\label{observaineq}
    \mathbb{E} |\phi(0)|^2_{L^2(\mathcal{G})} + \sum_{i=1}^m \mathbb{E} \int_0^T \int_{\mathcal{O}_{i,d}} \rho^{-2} |\psi_i|^2 \, dx \, dt \leq C \left[ \mathbb{E} \iint_{Q_0} \phi^2 \, dx \, dt + \mathbb{E} \iint_Q \Phi^2 \, dx \, dt \right].
    \end{align}
\end{enumerate}
\end{prop}

Using the method of Carleman estimates (see \cite{stakNahSPEs}), we can prove the observability inequality \eqref{observaineq} under certain assumptions on the observation regions \( \mathcal{O}_{i,d} \) and the weight function \( \rho \). The following result then holds.
\begin{prop}\label{Pro4.2}
Assume that \( \mathcal{O}_{i,d} = \mathcal{O}_d \), \( \mathcal{O}_d \cap \mathcal{G}_0 \neq \emptyset \), and that \( \beta_i \geq 1 \) are large enough, \( i = 1, 2, \dots, m \). Then, there exist a constant \( C > 0 \) and a positive weight function \( \rho = \rho(t) \), which blows up as \( t \to T \), such that for any \( \phi_T \in L^2_{\mathcal{F}_T}(\Omega; L^2(\mathcal{G})) \), the solution \( (\phi, \Phi; \psi_i) \) of \eqref{ADJSO1} satisfies the observability inequality:
\begin{align*}
\mathbb{E} |\phi(0)|^2_{L^2(\mathcal{G})} + \sum_{i=1}^m \mathbb{E} \iint_Q \rho^{-2} |\psi_i|^2 \, dx \, dt \leq C \left[ \mathbb{E} \iint_{Q_0} \phi^2 \, dx \, dt + \mathbb{E} \iint_Q \Phi^2 \, dx \, dt \right].
\end{align*}
\end{prop}

From Propositions \ref{propo5.2} and \ref{Pro4.2}, we obtain the following result.
\begin{thm}
Equation \eqref{eqq1.1} is optimally null controllable at time \( T \) in the sense of Stackelberg-Nash.
\end{thm}

\subsection{Application to the backward stochastic heat equation}
In this section, we study the Stackelberg-Nash null controllability for the backward stochastic heat equation using the method of Carleman estimates.
\subsubsection{Problem statement}
We consider the following backward stochastic heat equation
\begin{equation}\label{eqq1.1back}
\begin{cases}
dy + \Delta y \, dt = \left[a_1 y + a_2 Y + u_1 \mathbbm{1}_{\mathcal{G}_0} + \sum_{i=1}^m v_i \mathbbm{1}_{\mathcal{G}_i} \right] \, dt
+ \left[Y + u_2 \right] \, dW(t) & \text{in } Q, \\
y = 0 & \text{on } \Sigma, \\
y(T) = y_T & \text{in } \mathcal{G},
\end{cases}
\end{equation}
where \( (y, Y) \) represents the state variable, \( y_T \in L^2_{\mathcal{F}_T}(\Omega; L^2(\mathcal{G})) \) is the terminal state, and \( a_1, a_2 \in L^\infty_\mathcal{F}(0,T; L^\infty(\mathcal{G})) \). The control inputs are \( (u_1, u_2) \in \mathcal{U}_T \) for the leaders, and \( (v_1, \dots, v_m) \in \mathcal{V}_T \) for the followers. Notice that, under the conditions given in \eqref{asummmofabstset}, the equation \eqref{eqq1.1back} can also be expressed within the framework of the abstract backward equation \eqref{asteqq1.1back}.

For fixed target functions \( y_{i,d}, Y_{i,d} \in L^{2,\rho}_{\mathcal{F}}(0,T; L^2(\mathcal{G})) \), \( i = 1, 2, \dots, m \), where the weight function \( \rho = \rho(t) \) will be defined later, we now introduce the following functionals:
\begin{enumerate}[(1)]
    \item For the leaders, the \textbf{main functional} \( J \) is given by
    \begin{equation*}
    J(u_1, u_2) = \frac{1}{2} \mathbb{E} \iint_Q \left( \mathbbm{1}_{\mathcal{G}_0} |u_1|^2 + |u_2|^2 \right) \,dx\,dt.
    \end{equation*}
    
    \item For the followers, the \textbf{secondary functionals}  \( J^B_i \) are defined as follows: For each \( i = 1, 2, \dots, m \),
    
\begin{align}\label{functji1seBACKc}
    \begin{aligned}
    J^B_i(u_1, u_2; v_1, \dots, v_m) =  &\, \frac{\alpha_i}{2} \mathbb{E} \iint_{(0,T) \times \mathcal{O}_{i,d}} |y - y_{i,d}|^2 \,dx\,dt + \frac{\widetilde{\alpha}_i}{2} \mathbb{E} \iint_{(0,T) \times \widetilde{\mathcal{O}}_{i,d}} |Y - Y_{i,d}|^2 \,dx\,dt \\
    &+ \frac{\beta_i}{2} \mathbb{E} \iint_{(0,T) \times \mathcal{G}_i} |v_i|^2 \,dx\,dt,
    \end{aligned}
    \end{align}
\end{enumerate}
where \( \alpha_i, \widetilde{\alpha}_i > 0 \), \( \beta_i \geq 1 \) are constants, and \( (y,Y) \) is the solution of \eqref{eqq1.1back}. Comparing \eqref{functji1seBACKc} with our abstract framework, we introduce the operators \( K_i \), \( \widetilde{K}_i \), and \( R_i \) as follows:
\[
K_i = \mathbbm{1}_{\mathcal{O}_{i,d}}, \quad \widetilde{K}_i = \mathbbm{1}_{\widetilde{\mathcal{O}}_{i,d}}, \quad \text{and} \quad R_i = \text{Id}.
\]

Based on the results from Section \ref{sec3}, we have the following result on the Nash equilibrium for \( (J_1^B,\dots,J_m^B)\).
\begin{prop}\label{proppse5back}
There exists a large \( \overline{\beta} \geq 1 \) such that, if \( \beta_i \geq \overline{\beta} \) for \( i = 1, 2, \dots, m \), then for each \( (u_1, u_2) \in \mathcal{U}_T \), there exists a unique Nash equilibrium \( (v^*_1(u_1,u_2), \dots, v^*_m(u_1,u_2)) \in \mathcal{V}_T \) for the functionals \( (J_1^B,\dots,J_m^B)\) associated with \( (u_1,u_2) \). Moreover, we have that
    \begin{align*}
    v^*_i = -\frac{1}{\beta_i} z_i \Big|_{(0,T) \times \mathcal{G}_i}, \qquad i = 1, 2, \dots, m,
    \end{align*}
    where \( z_i \) is the solution of the forward equation
    \begin{equation}\label{backadjback}
    \begin{cases}
    \begin{array}{ll}
    dz_i - \Delta z_i \, dt = \big[ -a_1 z_i - \alpha_i(y - y_{i,d}) \mathbbm{1}_{\mathcal{O}_{i,d}} \big] \, dt + \big[ -a_2 z_i - \widetilde{\alpha}_i (Y - Y_{i,d}) \mathbbm{1}_{\widetilde{\mathcal{O}}_{i,d}} \big] \, dW(t) & \quad \textnormal{in} \, Q, \\
    z_i = 0 & \quad \textnormal{on} \, \Sigma, \\
    z_i(0) = 0 & \quad \textnormal{in} \, \mathcal{G}.
    \end{array}
    \end{cases}
    \end{equation}
\end{prop}

From Proposition \ref{proppse5back}, we can reduce the Stackelberg-Nash controllability problem for the system \eqref{eqq1.1back} to the classical null controllability for the following coupled backward-forward system
\begin{equation}\label{eqq4.7back}
\begin{cases}
\begin{array}{ll}
dy + \Delta y \, dt = \left[a_1 y + a_2 Y + u_1 \mathbbm{1}_{\mathcal{G}_0} - \sum_{i=1}^m \frac{1}{\beta_i} z_i \mathbbm{1}_{\mathcal{G}_i} \right] \, dt
+ \left[Y + u_2 \right] \, dW(t) & \quad \textnormal{in} \, Q, \\
 dz_i - \Delta z_i \, dt = \big[ -a_1 z_i - \alpha_i(y - y_{i,d}) \mathbbm{1}_{\mathcal{O}_{i,d}} \big] \, dt + \big[ -a_2 z_i - \widetilde{\alpha}_i (Y - Y_{i,d}) \mathbbm{1}_{\widetilde{\mathcal{O}}_{i,d}} \big] \, dW(t) & \quad \textnormal{in} \, Q, \\
y = 0, \,\, z_i = 0 & \quad \textnormal{on} \, \Sigma, \\
y(T) = y_T, \,\, z_i(0) = 0, \quad i = 1, 2, \dots, m, & \quad \textnormal{in} \, \mathcal{G}.
\end{array}
\end{cases}
\end{equation}
This problem is equivalent to an observability inequality for the following adjoint coupled system
\begin{equation}\label{ADJSO1back}
\begin{cases}
d\phi - \Delta \phi \,dt = \left[-a_1 \phi  + \sum_{i=1}^m \alpha_i \psi_i \mathbbm{1}_{\mathcal{O}_{i,d}}\right] \,dt +\left[- a_2\phi+ \sum_{i=1}^m \widetilde{\alpha}_i \Psi_i \mathbbm{1}_{\widetilde{\mathcal{O}}_{i,d}}\right] \,dW(t) & \quad \textnormal{in} \, Q, \\
d\psi_i + \Delta \psi_i \,dt = \left[a_1 \psi_i+a_2 \Psi_i + \frac{1}{\beta_i} \mathbbm{1}_{\mathcal{G}_i} \phi\right] \,dt + \Psi_i \,dW(t) & \quad \textnormal{in} \, Q, \\
\phi = 0, \,\, \psi_i = 0 & \quad \textnormal{on} \, \Sigma, \\
\phi(0) = \phi_0, \,\, \psi_i(T) = 0, \quad i = 1, 2, \dots, m, & \quad \textnormal{in} \, \mathcal{G}.
\end{cases}
\end{equation}

By Theorem \ref{thmm4.2exback}, we have the following null controllability characterization in the sense of Stackelberg-Nash for the backward equation \eqref{eqq1.1back}.
\begin{prop}\label{propo5.2back}
Let \( \rho = \rho(t) \) be a suitable weight function, and let \( \beta_i \geq 1 \) (\( i = 1, 2, \dots, m \)) be sufficiently large. The following assertions are equivalent:
\begin{enumerate}
    \item System \eqref{eqq1.1back} is optimally null controllable at time \( 0 \) in the sense of Stackelberg-Nash.
    \item The solution \( (\phi; \psi_i, \Psi_i) \) of the adjoint system \eqref{ADJSO1back} satisfies the observability inequality:
\begin{align}\label{observaineqbacksec}
\begin{aligned}
    &\,\mathbb{E} |\phi(T)|^2_{L^2(\mathcal{G})} + \sum_{i=1}^m \mathbb{E} \int_0^T \int_{\mathcal{O}_{i,d}} \rho^{-2} |\psi_i|^2 \, dx \, dt + \sum_{i=1}^m \mathbb{E} \int_0^T \int_{\widetilde{\mathcal{O}}_{i,d}} \rho^{-2} |\Psi_i|^2 \, dx \, dt \\
    &\leq C \left[ \mathbb{E} \iint_{Q_0} \phi^2 \, dx \, dt + \mathbb{E} \iint_Q \left| \sum_{i=1}^m \widetilde{\alpha}_i \Psi_i \mathbbm{1}_{\widetilde{\mathcal{O}}_{i,d}} - a_2\phi \right|^2 \, dx \, dt \right].
    \end{aligned}
    \end{align}
\end{enumerate}
\end{prop}

To prove the observability inequality \eqref{observaineqbacksec}, we assume the following assumptions: For each \( i = 1, 2, \dots, m \),
\begin{align}\label{asumpforcar}
\begin{aligned}
&\widetilde{\mathcal{O}}_{i,d} = \mathcal{G},\quad \quad  \widetilde{\alpha}_i=\alpha_i, \\
\mathcal{O}&_{i,d}   = \mathcal{O}_d,\quad 
 \mathcal{O}_d\cap \mathcal{G}_0 \neq \emptyset.
\end{aligned}
\end{align}
We have the following observability inequality for the system \eqref{ADJSO1back}. The proof of this key result will be derived in the next subsection.
\begin{prop}\label{Pro4.2back}
Assume that \eqref{asumpforcar} holds. Then for \( \beta_i \geq 1 \) large enough for each \( i = 1, 2, \dots, m \), there exists a constant \( C > 0 \) and a positive weight function \( \rho = \rho(t) \), which blows up as \( t \to 0 \), such that for any \( \phi_0 \in L^2_{\mathcal{F}_0}(\Omega; L^2(\mathcal{G})) \), the solution \( (\phi; \psi_i, \Psi_i) \) of \eqref{ADJSO1back} satisfies
\begin{align}\label{observaineqback}
\begin{aligned}
&\,\mathbb{E} |\phi(T)|^2_{L^2(\mathcal{G})} + \sum_{i=1}^m \mathbb{E} \iint_Q \rho^{-2} |\psi_i|^2 \, dx \, dt + \sum_{i=1}^m \mathbb{E} \iint_Q \rho^{-2} |\Psi_i|^2 \, dx \, dt \\
&\leq C \left[ \mathbb{E} \iint_{Q_0} \phi^2 \, dx \, dt + \mathbb{E} \iint_Q \left| \sum_{i=1}^m \alpha_i \Psi_i - a_2\phi \right|^2 \, dx \, dt \right].
\end{aligned}
\end{align}
\end{prop}
By Propositions \ref{propo5.2back} and \ref{Pro4.2back}, we have the following main result.
\begin{thm}\label{thmm5.22}
Equation \eqref{eqq1.1back} is optimally null controllable at time \(0\) in the sense of Stackelberg-Nash.
\end{thm}
\begin{rmk}
The assumptions in \eqref{asumpforcar} are crucial for deriving the observability inequality \eqref{observaineqback} in the next subsection using Carleman estimates, which are then employed to prove Theorem \ref{thmm5.22}. Therefore, it would be of great interest to investigate whether these assumptions are indeed necessary.
\end{rmk}
\subsubsection{Proof of the observability inequality \eqref{observaineqback}}
The proof of the estimate \eqref{observaineqback} relies on a new Carleman estimate for the coupled system \eqref{ADJSO1back}. To apply this method, we first introduce some notations and functions. For large parameters \( \lambda, \mu \geq 1 \), we define the functions
\begin{align*}
\alpha \equiv \alpha(t, x) = \frac{e^{\mu \eta_0(x)} - e^{2\mu |\eta_0|_\infty}}{t(T-t)}, \qquad  \theta  = e^{\lambda \alpha}, \qquad \gamma \equiv \gamma(t) = \frac{1}{t(T-t)},
\end{align*}
where \( \eta_0 \in C^4(\overline{\mathcal{G}}) \) is a function satisfying the following conditions:
\begin{equation*}
    \eta_0 > 0 \quad \text{in } \mathcal{G}, \quad \eta_0 = 0 \quad \text{on } \Gamma, \quad |\nabla \eta_0| > 0 \quad \text{in } \overline{\mathcal{G} \setminus \mathcal{B}},
\end{equation*}
for a non-empty open subset \( \mathcal{B} \Subset \mathcal{G} \). The existence of such a function \( \eta_0 \) can be found in \cite{BFurIman}.

For a process $z$, we denote by
$$I(z)=\lambda^3\mathbb{E}\iint_Q\theta^2\gamma^3 z^2\,dx\,dt+\lambda\mathbb{E}\iint_Q \theta^2\gamma|\nabla z|^2\,dx\,dt.$$
We first recall the following Carleman estimate (see \cite[Theorem 1.1]{liu2014global}).
\begin{lm}\label{lm1.1}
There exist a large $\mu_0\geq1$ such that for all $\mu=\mu_0$, one can find constants $C>0$ and $\lambda_1\geq1$ depending only on $\mathcal{G}$, $\mathcal{B}$, $\mu_0$ and $T$ such that for all $\lambda\geq\lambda_1$, $F_0,F_1\in L^2_\mathcal{F}(0,T;L^2(\mathcal{G}))$, and $z_0\in L^2_{\mathcal{F}_0}(\Omega;L^2(\mathcal{G}))$, the solution $z$ of the forward equation \begin{equation*}
\begin{cases}
\begin{array}{ll}
dz - \Delta z \,dt = F_0 \,dt + F_1\,dW(t)&\quad\textnormal{in}\,\,Q,\\
z=0&\quad\textnormal{on}\,\,\Sigma,\\
z(0)=z_0&\quad\textnormal{in}\,\,\mathcal{G},
\end{array}
\end{cases}
\end{equation*}
satisfies that
\begin{align}\label{carfor5.6}
\begin{aligned}
I(z)\leq C \left[ \lambda^3\mathbb{E}\int_0^T\int_\mathcal{B} \theta^2\gamma^3 z^2 \,dx\,dt+ \mathbb{E}\iint_Q \theta^2F_0^2 \,dx\,dt+\lambda^2\mathbb{E}\iint_Q \theta^2\gamma^2F_1^2 \,dx\,dt\right]. 
\end{aligned}\end{align}
\end{lm}
We also have the following Carleman estimate (we refer to \cite[Theorem 6.1]{tang2009null}).
\begin{lm}\label{lm1.13.2252}
There exist a large $\mu_1\geq1$ such that for $\mu=\mu_1$, one can find constants $C>0$ and $\lambda_2\geq1$ depending only on $\mathcal{G}$, $\mathcal{B}$, $\mu_1$ and $T$ such that for all $\lambda\geq\lambda_2$, $F_2\in L^2_\mathcal{F}(0,T;L^2(\mathcal{G}))$ and $z_T\in L^2_{\mathcal{F}_T}(\Omega;L^2(\mathcal{G}))$, the solution $(z,Z)$ of the backward equation 
\begin{equation*}
\begin{cases}
dz+\Delta z\,dt=F_2 \,dt+ Z \,dW(t) & \textnormal{in}\,\,Q,\\ 
z=0 & \textnormal{on}\,\,\Sigma,\\
z(T)=z_T & \textnormal{in}\,\, \mathcal{G},
\end{cases}
\end{equation*} satisfies 
\begin{align}\label{carback5.8}
\begin{aligned}
I(z)\leq C \left[ \lambda^3\mathbb{E}\int_0^T\int_\mathcal{B} \theta^2\gamma^3 z^2 \,dx\,dt+ \mathbb{E}\iint_Q \theta^2 F_2^2 \,dx\,dt+\lambda^2\mathbb{E}\iint_Q \theta^2\gamma^2 Z^2 \,dx\,dt\right]. 
\end{aligned}
\end{align}
\end{lm}

In what follows, we fix \( \mu = \overline{\mu} = \max(\mu_0, \mu_1) \), where \( \mu_0 \) (resp. \( \mu_1 \)) is the constant given in Lemma \ref{lm1.1} (resp. Lemma \ref{lm1.13.2252}). Moreover, to derive the appropriate Carleman estimate for the system \eqref{ADJSO1back}, we  define the following modified functions:
\begin{align}\label{rec1}
\overline{\alpha} \equiv \overline{\alpha}(t,x) = \frac{e^{\overline{\mu}\eta_0(x)} - e^{2\overline{\mu}\vert\eta_0\vert_\infty}}{\ell(t)}, \qquad
\overline{\theta} = e^{\lambda\overline{\alpha}}, \qquad \overline{\gamma} \equiv \overline{\gamma}(t) = \frac{1}{\ell(t)},
\end{align}
where
\begin{equation*}
\ell(t) = 
\begin{cases}
    \begin{array}{ll}
        t(T-t)  & \textnormal{for }\, 0 \leq t \leq T/2, \\
        T^2/4 & \textnormal{for }\, T/2 \leq t \leq T.
    \end{array}
\end{cases}
\end{equation*}
We also denote by
\begin{align}\label{inteIbar}
\overline{I}_{t_1,t_2}(z) = \mathbb{E} \int_{t_1}^{t_2} \int_\mathcal{G} \overline{\theta}^2 \overline{\gamma}^3 z^2 \,dx\,dt 
+ \mathbb{E} \int_{t_1}^{t_2} \int_\mathcal{G} \overline{\theta}^2 \overline{\gamma} \vert\nabla z\vert^2 \,dx\,dt, \quad 0 \leq t_1 \leq t_2 \leq T.
\end{align}

We have the following Carleman estimate for the system \eqref{ADJSO1back}.
\begin{lm}\label{lem4.5stthm} 
Assume that \eqref{asumpforcar} holds. Then, there exists a constant \( C > 0 \), and large constants \( \overline{\beta}, \overline{\lambda} \geq 1 \), depending only on \( \mathcal{G} \), \( \mathcal{G}_0 \), \( \mathcal{O}_d \), \( \overline{\mu} \), and \( T \), such that for all \( \beta_i \geq \overline{\beta} \) (for \( i = 1, 2, \dots, m \)) and \( \lambda \geq \overline{\lambda} \), the solution \( (\phi; \psi_i, \Psi_i) \) of \eqref{ADJSO1back} satisfies that
\begin{align}\label{improvedCarl}
\begin{aligned}
&\,\mathbb{E} |\phi(T)|^2_{L^2(\mathcal{G})}+\overline{I}_{0,T}(\phi)+\overline{I}_{0,T}(h)+\mathbb{E}\iint_Q \overline{\theta}^2\overline{\gamma}^3 H^2 \,dx\,dt\\
&\leq C \left[\lambda^{11}  \mathbb{E}\iint_{Q_0} \theta^2 \gamma^{11} \phi^2 \, \,dx\,dt+\lambda^5 \mathbb{E}\iint_{Q}\theta^2\gamma^5\left| \displaystyle H - a_2\phi\right|^2\,dx\,dt\right],
\end{aligned}
\end{align}
where $(h,H)= \left(\sum_{i=1}^m\alpha_i\psi_i,\sum_{i=1}^m\alpha_i\Psi_i\right)$.
\end{lm}
\begin{proof} The proof is divided into two steps.\\
\textbf{Step 1.} We first note that \(\theta = \overline{\theta}\) and \(\gamma = \overline{\gamma}\) on \((0,T/2)\), then we have
\begin{align}\label{estim3.577sec}
\overline{I}_{0,\frac{T}{2}}(\phi) + \overline{I}_{0,\frac{T}{2}}(h) +\mathbb{E}\int_0^{T/2}\int_\mathcal{G} \overline{\theta}^2\overline{\gamma}^3H^2 \,dx\,dt\leq C \left[ I(\phi) + I(h)+\mathbb{E}\iint_Q \theta^2\gamma^3 H^2 \,dx\,dt\right].
\end{align}
Let us consider the function $\kappa\in C^1([0,T])$ such that
\begin{equation}\label{kappadef}
\kappa=0\quad \text{in}\,\; [0,T/4],\quad\textnormal{and}\quad\ \,\,\kappa= 1\quad \text{in} \,\,\;[T/2, T].
\end{equation}
Set \(\widetilde{\phi} = \kappa \phi\). Then \(\widetilde{\phi}\) is the solution of the forward system
\begin{equation}\label{equaphitildesec}
\begin{cases}
\begin{array}{ll}
d\widetilde{\phi} -\Delta\widetilde{\phi} \, dt = \left[-a_1\widetilde{\phi}+\kappa h\mathbbm{1}_{\mathcal{O}_d}+\kappa'\phi\right] \, dt + [-a_2\widetilde{\phi}+\kappa H] \, dW(t) & \text{in} \, Q, \\
\widetilde{\phi}= 0 & \text{on} \, \Sigma, \\
\widetilde{\phi}(0) = 0 & \text{in} \, \mathcal{G}.
\end{array}
\end{cases}
\end{equation}
By the classical energy estimate for the state $\widetilde{\phi}$, we have that
\begin{align*}
\begin{aligned}
&\,\mathbb{E}|\phi(T)|^2_{L^2(\mathcal{G})}+\mathbb{E}\int_{T/2}^T\int_\mathcal{G} \phi^2  \,dx\,dt\leq C\left[\mathbb{E}\int_{T/4}^{T/2}\int_\mathcal{G}  \phi^2 \,dx\,dt +  \mathbb{E}\int_{T/4}^T\int_\mathcal{G}  h^2 \,dx\,dt+\mathbb{E}\int_{T/4}^T\int_\mathcal{G}  H^2 \,dx\,dt\right].
\end{aligned}
\end{align*}
Recalling \eqref{rec1}, it is easy to see that
\begin{align*}
\begin{aligned}
\mathbb{E}|\phi(T)|^2_{L^2(\mathcal{G})}+\overline{I}_{\frac{T}{2},T}(\phi) \leq C\Bigg[&\mathbb{E}\int_{T/4}^{T/2}\int_\mathcal{G}  (\phi^2+h^2+H^2) \,dx\,dt +  \overline{I}_{\frac{T}{2},T}(h) \\
&+\mathbb{E}\int_{T/2}^T\int_\mathcal{G}  \overline{\theta}^2\overline{\gamma}^3H^2 \,dx\,dt\Bigg],
\end{aligned}
\end{align*}
which implies that
\begin{align}\label{estmmforcar}
\begin{aligned}
&\mathbb{E}|\phi(T)|^2_{L^2(\mathcal{G})}+\overline{I}_{\frac{T}{2},T}(\phi)+  \overline{I}_{\frac{T}{2},T}(h) +\mathbb{E}\int_{T/2}^T\int_\mathcal{G}  \overline{\theta}^2\overline{\gamma}^3H^2 \,dx\,dt\\
&\leq C\Bigg[I(\phi) + I(h)+\mathbb{E}\iint_Q \theta^2\gamma^3H^2 \,dx\,dt  \\
&\hspace{1cm}+\overline{I}_{\frac{T}{2},T}(h)+\mathbb{E}\int_{T/2}^T\int_\mathcal{G}  \overline{\theta}^2\overline{\gamma}^3H^2 \,dx\,dt\Bigg].
\end{aligned}
\end{align}
To estimate the last two terms on the right-hand side of \eqref{estmmforcar}, we use the classical energy estimate for the state $(h,H)$, then we derive that
\begin{align}\label{estibyenergyes}
    \overline{I}_{\frac{T}{2},T}(h)+\mathbb{E}\int_{T/2}^T\int_\mathcal{G}  \overline{\theta}^2\overline{\gamma}^3H^2 \,dx\,dt\leq C\sum_{i=1}^m \frac{\alpha_i^2}{\beta_i^2}\overline{I}_{\frac{T}{2},T}(\phi).
\end{align}
Combining \eqref{estmmforcar} and \eqref{estibyenergyes} and choosing a large $\beta_i$ ($i=1,2,\dots,m$), we find that
\begin{align}\label{estmmforcarstep1}
\begin{aligned}
&\mathbb{E}|\phi(T)|^2_{L^2(\mathcal{G})}+\overline{I}_{\frac{T}{2},T}(\phi)+  \overline{I}_{\frac{T}{2},T}(h) +\mathbb{E}\int_{T/2}^T\int_\mathcal{G}  \overline{\theta}^2\overline{\gamma}^3H^2 \,dx\,dt\\
&\leq C\bigg[I(\phi) + I(h)+\mathbb{E}\iint_Q \theta^2\gamma^3H^2 \,dx\,dt\bigg].
\end{aligned}
\end{align}
Adding \eqref{estim3.577sec} and \eqref{estmmforcarstep1}, we deduce that 
\begin{align}\label{estmmforcarstep1est1}
\begin{aligned}
&\mathbb{E}|\phi(T)|^2_{L^2(\mathcal{G})}+\overline{I}_{0,T}(\phi)+  \overline{I}_{0,T}(h) +\mathbb{E}\iint_Q  \overline{\theta}^2\overline{\gamma}^3H^2 \,dx\,dt\\
&\leq C\bigg[I(\phi) + I(h)+\mathbb{E}\iint_Q \theta^2\gamma^3H^2 \,dx\,dt\bigg].
\end{aligned}
\end{align}
\textbf{Step 2.} We introduce the following nonempty open subsets \( \mathcal{O}_1 \) and \( \mathcal{O}_2 \), such that
\[
\mathcal{O}_2 \Subset \mathcal{O}_1 \Subset \mathcal{O}_0 \equiv  \mathcal{O}_d\cap \mathcal{G}_0,
\]
where $\mathcal{O}_i\Subset \mathcal{O}_{i-1}$ ($i=1,2$) means that $\overline{\mathcal{O}_{i}}\subset \mathcal{O}_{i-1}$. We also define the functions \( \zeta_i \in C^{\infty}(\mathbb{R}^N) \) (for the existence of such functions, refer to, e.g., \cite{HSP18}) satisfying
\begin{align}\label{assmzeta}
\begin{aligned}
& 0 \leq \zeta_i \leq 1, \quad \zeta_i = 1 \,\, \text{in} \,\, \mathcal{O}_{3-i}, \quad \text{Supp}(\zeta_i) \subset \mathcal{O}_{2-i}, \\ 
& \frac{\Delta \zeta_i}{\zeta_i^{1/2}} \in L^\infty(\mathcal{G}), \quad \frac{\nabla \zeta_i}{\zeta_i^{1/2}} \in L^\infty(\mathcal{G}; \mathbb{R}^N), \quad i=1,2.
\end{aligned}
\end{align}
Set $w_j=\theta^2(\lambda\gamma)^j$, it is easy to see that for a large enough $\lambda$ and any $j\in\mathbb{N}$, we have
\begin{align}\label{timedrrives}
|\partial_t w_j| \leq C \lambda^{j+2} \theta^2 \gamma^{j+2},\qquad |\nabla(w_j\zeta_i)| \leq C \lambda^{j+1} \theta^2 \gamma^{j+1}\zeta_i^{1/2}, \quad i=1,2.
\end{align}
Applying Carleman estimate \eqref{carfor5.6} for the state $\phi$ with $\mathcal{B}=\mathcal{O}_2$, we have for a large enough $\lambda$,
\begin{align}\label{carfor5.6caphi}
\begin{aligned}
I(\phi)\leq C \left[ \lambda^3\mathbb{E}\int_0^T\int_{\mathcal{O}_2} \theta^2\gamma^3 \phi^2 \,dx\,dt+ \mathbb{E}\iint_Q \theta^2h^2 \,dx\,dt+\lambda^2\mathbb{E}\iint_Q \theta^2\gamma^2|H - a_2\phi|^2 \,dx\,dt\right]. 
\end{aligned}\end{align}
Using Carleman estimate \eqref{carback5.8} for the state $(h,H)$ with $\mathcal{B}=\mathcal{O}_2$, we have for a large enough $\lambda$,
\begin{align}\label{carback5.8carhH}
\begin{aligned}
I(h)\leq C \left[ \lambda^3\mathbb{E}\int_0^T\int_{\mathcal{O}_2} \theta^2\gamma^3 h^2 \,dx\,dt+ \mathbb{E}\iint_Q \theta^2 \phi^2 \,dx\,dt+\lambda^2\mathbb{E}\iint_Q \theta^2\gamma^2 H^2 \,dx\,dt\right]. 
\end{aligned}
\end{align}
Combining \eqref{carfor5.6caphi} and \eqref{carback5.8carhH}, and taking a large $\lambda$, we obtain that
\begin{align}\label{carsec5.8carhH}
\begin{aligned}
I(\phi)+I(h)\leq C \left[ \lambda^3\mathbb{E}\int_0^T\int_{\mathcal{O}_2} \theta^2\gamma^3 h^2 \,dx\,dt+\lambda^3\mathbb{E}\int_0^T\int_{\mathcal{O}_2} \theta^2\gamma^3 \phi^2 \,dx\,dt+\lambda^2\mathbb{E}\iint_Q \theta^2\gamma^2|H - a_2\phi|^2 \,dx\,dt\right]. 
\end{aligned}
\end{align}
Let us now estimate the first term on the right-hand side of \eqref{carsec5.8carhH}. We first observe that
\begin{align}\label{estfpelh}
\lambda^3\mathbb{E}\int_0^T\int_{\mathcal{O}_2} \theta^2\gamma^3 h^2 \,dx\,dt\leq \mathbb{E}\iint_Q w_3\zeta_1 h^2 \,dx\,dt.
\end{align}
Using Itô's formula for $d(w_3\zeta_1 h\phi)$, we find that
\begin{align}\label{foridentth}
    \begin{aligned}
\mathbb{E}\iint_Q w_3\zeta_1 h^2 \,dx\,dt=&-\mathbb{E}\iint_Q \partial_tw_3 \,\zeta_1 h\phi \,dx\,dt+\mathbb{E}\iint_Q \nabla\phi\cdot\nabla(w_3\zeta_1 h) \,dx\,dt\\
&-\mathbb{E}\iint_Q \nabla h\cdot\nabla(w_3\zeta_1\phi)\,dx\,dt-\sum_{i=1}^m\frac{\alpha_i}{\beta_i}\mathbb{E}\iint_Q w_3\zeta_1\mathbbm{1}_{\mathcal{G}_i}\phi^2 \,dx\,dt\\
&-\mathbb{E}\iint_Q w_3\zeta_1 H^2 \,dx\,dt.
    \end{aligned}
\end{align}
Let \(\varepsilon > 0\) be a small number, and using \eqref{timedrrives} and applying Young's inequality to each term on the right-hand side of \eqref{foridentth}, we obtain for large \(\lambda\),
\begin{align}\label{estm1to5}
    \begin{aligned}
\mathbb{E}\iint_Q w_3\zeta_1 h^2 \,dx\,dt\leq\varepsilon I(h)+\frac{C}{\varepsilon}\bigg[&\lambda^7\mathbb{E}\iint_Q \theta^2\gamma^7\zeta_1\phi^2 \,dx\,dt+\lambda^5\mathbb{E}\iint_Q \theta^2\gamma^5\zeta_1|\nabla \phi|^2 \,dx\,dt\\
&+\lambda^3\mathbb{E}\iint_Q \theta^2\gamma^3|H-a_2\phi|^2 \,dx\,dt\bigg]. 
\end{aligned}
\end{align}
Combining \eqref{carsec5.8carhH}, \eqref{estfpelh}, \eqref{foridentth}, \eqref{estm1to5}, and taking a sufficiently small \(\varepsilon\) and a large \(\lambda\), we deduce that
\begin{align}\label{carsec5.8carhHsec2}
\begin{aligned}
I(\phi)+I(h)\leq C \left[ \lambda^7\mathbb{E}\int_0^T\int_{\mathcal{O}_1} \theta^2\gamma^7 \phi^2 \,dx\,dt+\lambda^5\mathbb{E}\int_0^T\int_{\mathcal{O}_1} \theta^2\gamma^5 |\nabla\phi|^2 \,dx\,dt+\lambda^3\mathbb{E}\iint_Q \theta^2\gamma^3|H - a_2\phi|^2 \,dx\,dt\right]. 
\end{aligned}
\end{align}
Let us now absorb the gradient term on the right-hand side of \eqref{carsec5.8carhHsec2}. Notice that
\begin{align}\label{estiiforgrphi}
\lambda^5\mathbb{E}\int_0^T\int_{\mathcal{O}_1} \theta^2\gamma^5 |\nabla\phi|^2 \,dx\,dt\leq \mathbb{E}\iint_Q w_5\zeta_2 |\nabla\phi|^2 \,dx\,dt.
\end{align}
Computing $d(w_5\zeta_2\phi^2)$, we obtain that
\begin{align}\label{estimate543}
    \begin{aligned}
2\mathbb{E}\iint_Q w_5\zeta_2 |\nabla\phi|^2 \,dx\,dt=&\mathbb{E}\iint_Q \partial_t w_5\,\zeta_2\phi^2\,dx\,dt-2\mathbb{E}\iint_Q \phi\nabla\phi\cdot\nabla(w_5\zeta_2)\,dx\,dt\\
&-2\mathbb{E}\iint_Q w_5\zeta_2a_1\phi^2\,dx\,dt+2\mathbb{E}\iint_Q w_5\zeta_2\phi h \,dx\,dt\\
&+\mathbb{E}\iint_Q w_5\zeta_2|H-a_2\phi|^2\,dx\,dt.
    \end{aligned}
\end{align}
Fix a small \(\varepsilon\), and applying Young's inequality, then for a large \(\lambda\), it follows that
\begin{align}\label{estimfr6to10}
    \begin{aligned}
2\mathbb{E}\iint_Q w_5\zeta_2 |\nabla\phi|^2 \,dx\,dt\leq\varepsilon[I(\phi)+I(h)]+\frac{C}{\varepsilon}\bigg[\lambda^{11}\mathbb{E}\iint_Q \theta^2\gamma^{11}\zeta_2\phi^2 \,dx\,dt+\lambda^5\mathbb{E}\iint_Q \theta^2\gamma^5|H-a_2\phi|^2 \,dx\,dt\bigg].
    \end{aligned}
\end{align}
Combining \eqref{carsec5.8carhHsec2}, \eqref{estiiforgrphi}, \eqref{estimate543}, \eqref{estimfr6to10},  and taking a small enough $\varepsilon$ and large $\lambda$, we conclude that
\begin{align}\label{carsec5.lastestm}
\begin{aligned}
I(\phi)+I(h)\leq C \Bigg[ \lambda^{11}\mathbb{E}\iint_{Q_0} \theta^2\gamma^{11} \phi^2 \,dx\,dt+\lambda^5\mathbb{E}\iint_Q \theta^2\gamma^5|H - a_2\phi|^2 \,dx\,dt\Bigg]. 
\end{aligned}
\end{align}
On the other hand, it is easy to see that
\begin{align}\label{lastesttcar}
    \mathbb{E}\iint_Q \theta^2\gamma^3 H^2 \,dx\,dt\leq C\left[\mathbb{E}\iint_Q \theta^2\gamma^3 |H-a_2\phi|^2 \,dx\,dt+\mathbb{E}\iint_Q \theta^2\gamma^3 \phi^2 \,dx\,dt\right].
\end{align}
Finally, combining \eqref{estmmforcarstep1est1}, \eqref{carsec5.lastestm} and \eqref{lastesttcar} and taking a large enough $\lambda$, we deduce the desired Carleman estimate \eqref{improvedCarl}. This concludes the proof of Lemma \ref{lem4.5stthm}.
\end{proof}

We are now in a position to prove the observability inequality \eqref{observaineqback}.
\begin{proof}[Proof of Proposition \ref{Pro4.2back}]
Define the function \(\rho(t) := e^{-\overline{\lambda} \overline{\alpha}^*(t)}\) where \(\overline{\alpha}^*(t) = \min_{x \in \overline{\mathcal{G}}} \overline{\alpha}(t, x)\).
By Itô's formula, we compute \(d(\rho^{-2} |\psi_i|^2)\), \(i = 1, 2, \dots, m\), then we conclude that for any \( t \in (0, T) \),
\begin{align}
    \begin{aligned}
        \mathbb{E}\int_\mathcal{G} \rho^{-2}(t)|\psi_i(t)|^2 \,dx\leq C\Bigg[\mathbb{E}\int_t^T\int_\mathcal{G} \rho^{-2}|\psi_i|^2 \,dx\,dt+\mathbb{E}\int_t^T\int_\mathcal{G} \rho^{-2}\phi^2 \,dx\,dt+\mathbb{E}\int_t^T\int_\mathcal{G} \rho^{-2}|\Psi_i|^2 \,dx\,dt\Bigg].
    \end{aligned}
\end{align}
Then by Gronwall's inequlity, it follows that
\begin{align}
    \begin{aligned}
\mathbb{E}\iint_Q \rho^{-2}|\psi_i|^2 \,dx\,dt\leq C\bigg[\mathbb{E}\iint_Q \rho^{-2}\phi^2 \,dx\,dt+\mathbb{E}\iint_Q \rho^{-2}|\Psi_i|^2 \,dx\,dt\bigg],
    \end{aligned}
\end{align}
which leads to
\begin{align}\label{esttim1obs}
    \begin{aligned}
\sum_{i=1}^m\mathbb{E}\iint_Q \rho^{-2}|\psi_i|^2 \,dx\,dt+\sum_{i=1}^m\mathbb{E}\iint_Q \rho^{-2}|\Psi_i|^2 \,dx\,dt\leq C\bigg[\mathbb{E}\iint_Q \rho^{-2}\phi^2 \,dx\,dt+\sum_{i=1}^m\mathbb{E}\iint_Q \rho^{-2}|\Psi_i|^2 \,dx\,dt\bigg].
    \end{aligned}
\end{align}
On the other hand, it is easy to see that
\begin{align}\label{esttim2obs}
    \mathbb{E}\iint_Q \rho^{-2}\phi^2 \,dx\,dt+\sum_{i=1}^m\mathbb{E}\iint_Q \rho^{-2}|\Psi_i|^2 \,dx\,dt\leq C\bigg[\overline{I}_{0,T}(\phi)+\mathbb{E}\iint_Q \overline{\theta}^2\overline{\gamma}^3 H^2 \,dx\,dt\bigg],
\end{align}
where  $H=\sum_{i=1}^m\alpha_i\Psi_i$. Finally, combining \eqref{esttim1obs} and \eqref{esttim2obs} with the Carleman estimate \eqref{improvedCarl}, we deduce the desired observability inequality \eqref{observaineqback}.
\end{proof}

\end{document}